\documentclass[notitlepage]{article}
\usepackage[utf8]{inputenc}
\usepackage{amsmath}
\usepackage{titling}
\usepackage{amsthm}
\usepackage{geometry}
\usepackage[pagewise]{lineno}\linenumbers
\nolinenumbers
\usepackage{setspace}
\usepackage{enumitem}
\usepackage{amsmath}
\usepackage{amsfonts}
\usepackage{amssymb}
\usepackage{latexsym}
\usepackage{amsthm}
\usepackage{verbatim}
\usepackage{authblk}
\usepackage[english]{babel}
\usepackage{fancyhdr}
\usepackage{titlepic}
\usepackage{fancyhdr}
\usepackage{graphicx}
\usepackage{lipsum}
\usepackage{titling}
\usepackage{textcomp}
\usepackage{gensymb}
\usepackage{amsfonts}
\usepackage{multirow}
\usepackage{subcaption}
\usepackage{mathrsfs}
\usepackage{setspace}
\usepackage{enumitem}
\usepackage{amsmath}
\usepackage{amsfonts}
\usepackage{amssymb}
\usepackage{latexsym}
\usepackage{amsthm}
\usepackage{verbatim}
\usepackage{authblk}
\usepackage{setspace}
\usepackage{verbatim}
\usepackage{longtable}
\usepackage{flexisym}
 \usepackage{breqn}
 \usepackage{longtable}
 \usepackage{amsmath, amssymb, etoolbox, expl3, mathtools, pgfkeys, pgfopts, xparse, xstring,tikz}
\usepackage[edge-length=1.5cm,root-radius=.2cm,label,ordering=Bourbaki,mark=o]{dynkin-diagrams}

\usepackage{microtype}  
\usepackage[intoc]{nomencl}
\usepackage{color}

\usepackage{keystroke}
\usepackage{xcolor}
\usepackage{listings}
\usepackage{setspace}
\usepackage{amsmath}
\usepackage{bm}
\usepackage{upgreek}
\usepackage{units}
\usepackage{lscape}
\usepackage{pbox}
\usepackage[tableposition=top]{caption}
\usepackage{float}
\usepackage{fancyvrb}
\floatstyle{plaintop}
\restylefloat{table}
\usepackage[version=3]{mhchem}
\usepackage{multicol}
\usepackage{hyperref,bookmark,refcount}
\hypersetup{
    colorlinks=true, 
    linktoc=all,     
    linkcolor=blue,  
}

\newtheorem{theorem*}{Theorem}
\newtheorem{remark*}{Remark}
\newtheorem{conjecture*}{Conjecture}
\newtheorem{corollary*}[theorem*]{Corollary}
\newtheorem{theorem}{Theorem}[section]
\newtheorem{corollary}[theorem]{Corollary}
\newtheorem{lemma}[theorem]{Lemma}
\newtheorem{remark}[theorem]{Remark}
\newtheorem{proposition}[theorem]{Proposition}

\usepackage[parfill]{parskip}
\DeclareMathOperator{\codim} {codim}

\DeclareMathOperator{\antidiag} {antidiag}
\DeclareMathOperator{\diag} {diag}
\DeclareMathOperator{\tran} {Tran}
\usepackage{xltabular}
\usepackage{mathtools}

\usepackage{upref}

\title{Generic stabilizers for simple algebraic groups acting on orthogonal and symplectic Grassmannians}
\author{Aluna Rizzoli}
\affil{EPFL \\ \texttt{aluna.rizzoli@epfl.ch}}
\usepackage[maxbibnames=99,
backend=biber,
style=numeric,
sorting=nyt,
url = false
]{biblatex}
\AtEveryBibitem{\clearfield{doi}}
\AtEveryBibitem{\clearfield{eprint}}
\AtEveryBibitem{\clearfield{isbn}}
\AtEveryBibitem{\clearfield{issn}}

\usepackage{listings}
\begin{document}
\maketitle
\thispagestyle{empty}

\begin{abstract}
    We consider faithful actions of simple algebraic groups on self-dual irreducible modules, and on the associated varieties of totally singular subspaces, under the assumption that the dimension of the group is at least as large as the dimension of the variety. We prove that in all but a finite list of cases, there is a dense open subset where the stabilizer of any point is conjugate to a fixed subgroup, called the generic stabilizer. We use these results to determine whether there exists a dense orbit. This in turn lets us complete the answer to the problem of determining all pairs of maximal connected subgroups of a classical group with a dense double coset. 
\end{abstract}

\section{Introduction and statement of results}

Let $G$ be a simple algebraic group over an algebraically closed field $K$ of characteristic $p$, where we take $p=\infty$ if $K$ has characteristic zero. Let $V$ be a non-trivial irreducible $KG$-module of dimension $d$. For $1\leq k\leq d$ the Grassmannian variety $\mathcal{G}_k(V)$ consists of all $k$-dimensional subspaces of $V$, and is isomorphic to the variety $SL(V)/P$ where $P$ is a maximal parabolic subgroup of $SL(V)$ stabilising a $k$-dimensional subspace. Assume that the module $V$ is self-dual. Then the group $G$ preserves a non-degenerate bilinear form, which is either symmetric or alternating (unless $p = 2$ when it is both). If the form is symmetric and $G$ preserves an associated quadratic form, we say that the module $V$ is \textit{orthogonal}, and we say it is \textit{symplectic} otherwise. For $1\leq k\leq \frac{d}{2}$ we denote by $\mathcal{S}_k(V)$ the variety of totally singular $k$-dimensional subspaces of $V$. Any such variety is irreducible unless $k=\frac{d}{2}$ and $V$ is orthogonal, in which case the two $SO(V)$-orbits on $\mathcal{S}_k(V)$ are its irreducible components, which we shall denote by $\mathcal{S}_{k'}(V)$ and  $\mathcal{S}_{k''}(V)$, or $\mathcal{S}_{k}'(V)$ and  $\mathcal{S}_{k}''(V)$. Note that there is no intrinsic way of choosing which of the two orbits is labeled $\mathcal{S}_{k'}(V)$, and therefore such choice is arbitrary, and usually simply dependent on the order of consideration. According as $V$ is orthogonal or symplectic, each such irreducible variety is isomorphic to $SO(V)/P$ or $Sp(V)/P$, where $P$ is a parabolic subgroup of $SO(V)$ or $Sp(V)$ (maximal unless $V$ is orthogonal and $k=\frac{d}{2}-1$), and the elements of the variety are  \textit{orthogonal Grassmannians} or \textit{symplectic Grassmannians}.

If $G$ acts faithfully on a variety $X$, we say that the action has \textit{generic stabilizer} $S$ if there exists a non-empty open subset $U\subseteq X$ such that the stabilizer $G_u$ is conjugate to $S$ for all $u\in U$. 

We say that the action has \textit{semi-generic stabilizer} $S$ if there exists a non-empty open subset $U\subseteq X$ such that $G_u$ is isomorphic to $S$ for all $u\in U$.

In general we say that $G$ acting on $X$ (not necessarily faithfully) has a (semi-)generic stabilizer if $G/G_X$, where $G_X$ denotes the kernel of the action, has a (semi-)generic stabilizer for its faithful action on $X$. Note that when $X=\mathcal{G}_k(V)$ or $X=\mathcal{S}_k(V)$, the kernel $G_X$ is precisely the center of $G$.

In characteristic zero, generic stabilizers exist under mild hypotheses. 
Richardson proved \cite[Thm.~A]{richardson-principal} that if $G$ is a reductive group acting on a smooth affine variety, then there is a generic stabilizer. Generic stabilizers for complex simple Lie groups $G$ acting on irreducible modules $V$ have been fully determined \cite{VinPopInv}.
In this case it follows from \cite{largeOrbitsVinberg} that the generic stabilizer is positive dimensional if and only if $\dim G\geq \dim X$.  However, even in characteristic $0$, there are examples of actions with no generic stabilizers (see \cite{richardson1972deformations}). 

If $p<\infty$ there are even more instances where generic stabilizers do not exist, as there is no analogue of Richardson's result. For example, in \cite[Example~8.3]{martin2015generic} we find a construction for an $SL_2(K)$-action on an affine variety in positive characteristic, with no generic stabilizer. 

Again let $G$ be a simple algebraic group over an algebraically closed field $K$ of characteristic $p$ and $V$ be a non-trivial irreducible $KG$-module. Recently Guralnick and Lawther have solved the generic stabilizer problem for the action on $X=\mathcal{G}_k(V)$. In \cite{generic} they proved that if $X=V$ or $X=\mathcal{G}_k(V)$, then the action of $G$ on $X$ has a generic stabilizer, unless $G=B_3$, $p=2$, $V$ is the spin module for $G$, and $k=4$, in which case the action has a semi-generic stabilizer but not a generic stabilizer. They showed that for such actions the generic stabilizer is in general trivial, and they otherwise determined all non-trivial (semi-)generic stabilizers explicitly. 

In this paper we treat the action of $G$ on the orthogonal and symplectic Grassmannians of self-dual irreducible $G$-modules, i.e. $X = \mathcal{S}_k(V)$. We only deal with the case $\dim G\geq \dim X$. The reason for this is two-fold. We will be interested in applications to questions about the existence of dense orbits and dense double cosets, for which we only need to be concerned with the cases where $\dim G\geq \dim X$. Secondly, as shown in \cite{generic}, the strategy for dealing with the situation $\dim G<\dim X$ presents entirely different challenges. The case $\dim G < \dim X$ shall be the subject of future work. Note that (semi-)generic stabilizers for the action on $\mathcal{S}_k(V)$ are generally going to be radically different from the ones for the action on $\mathcal{G}_k(V)$. Indeed it is often going to be the case that the generic stabilizer for the action on $\mathcal{G}_k(V)$ is finite, while the generic stabilizer for the action on $\mathcal{S}_k(V)$ is positive-dimensional. We shall now state our first result. The modules $V$ are denoted by their highest weight, the groups $G$ by their Dynkin diagram.

\begin{theorem*}\label{existence theorem}
Let $G$ be a simple algebraic group over an algebraically closed field of characteristic $p$, and $V$ a self-dual non-trivial irreducible $G$-module of dimension $d$ and highest weight $\lambda$. For $1\leq k\leq \frac{d}{2}$ such that $\dim G\geq \dim \mathcal{S}_k(V)$, if the action of $G$ on $\mathcal{S}_k(V)$ has no generic stabilizer, then $p=2$ and $(G,\lambda,k)$ is one of the following:
\begin{enumerate}[label=(\roman*)]
    \item $(E_7,\lambda_7,2)$;
    \item $(D_6,\lambda_6,2)$;
    \item $(A_5,\lambda_3,2)$;
    \item $(B_4,\lambda_4,8)$.
\end{enumerate}
In the first three cases the action of $G$ on $\mathcal{S}_k(V)$ has no generic stabilizer, but does have a semi-generic one.
In the last case the action of $G$ on $\mathcal{S}_k'(V)$ has a generic stabilizer, but the action of $G$ on $\mathcal{S}_k''(V)$ only has a semi-generic one.
\end{theorem*}

While Theorem~\ref{existence theorem} guarantees the existence, in Theorem~\ref{complete list of cases theorem} we shall explicitly determine the structure of every (semi-)generic stabilizer.
The proof of Theorem~\ref{existence theorem} involves a quick reduction to a finite list of families of cases to be considered, followed by a case-by-case analysis using many of the methods adopted in \cite{generic}. This case-by-case analysis is the subject of the majority of this paper.

Before stating our remaining results we shall set up some more notation. If $G$ is semisimple, let $T$ be a fixed maximal torus, and $\Phi$ the root system for $G$ with respect to $T$, described by its Dynkin diagram.
The root system has positive roots $\Phi^+$ and a base $\Delta=\{\alpha_1,\alpha_2,\dots,\alpha_n\}$ of simple roots. For the simple algebraic groups the ordering of the simple roots is taken according to Bourbaki \cite{bourbaki46}. For a $KG$-module $V$ and a weight $\mu$ of $G$, we write $V_\mu$ for the $\mu$-weight space of $V$. If $G$ acts on a set $X$, for $x\in X$ we denote by $G_x$ the stabilizer in $G$ of $x$.
We use $P$ to denote a parabolic subgroup containing a Borel subgroup $B\geq T$ and $P_k$ to denote the maximal parabolic subgroup obtained by deleting the $k$-th node of the Dynkin diagram for $G$. Similarly we use $P_{i,j}$ to denote the parabolic subgroup obtained by deleting the $i$-th and $j$-th nodes of the Dynkin diagram for $G$.
We use $T_i$ to denote an $i$-dimensional torus, $Sym(n)$ and $Alt(n)$ to denote the symmetric group and the alternating group on a set of size $n$, and $Dih(2n)$ for a dihedral group of order $2n$.
Throughout the paper we work modulo field twists and exceptional isogenies. So for example we only treat one of $C_n$ or $B_n$ in characteristic $2$, and we always assume that the highest weight $\lambda$ is not a multiple of $p$.

We now set up further notation to better encapsulate the exact setup we will adopt. As in \cite{generic}, we define a \textit{quadruple} to be a $4$-tuple of the form $(G,\lambda,p,k)$ with the following properties:
\begin{enumerate}[label=(\roman*)]
    \item $G$ is a simple algebraic group over an algebraically closed field of characteristic $p$;
    \item $V=V_G(\lambda)$ is an irreducible $G$-module;
    \item $1\leq k\leq \frac{\dim V}{2}$.
\end{enumerate}

We say that a quadruple $(G,\lambda,p,k)$ is \textit{small} if \[ \dim G\geq \dim\mathcal{G}_k(V).\] We say that a quadruple $(G,\lambda,p,k)$ is self-dual if $V_G(\lambda)$ is self-dual, in which case the quadruple is \textit{ts-small} if \[ \dim G\geq \dim\mathcal{S}_k(V).\] 

We say that the quadruple $(G,\lambda,p,k)$ has a (semi-)generic stabilizer if the action of $G$ on $X=\mathcal{G}_k(V)$ has a (semi-)generic stabilizer. For a self-dual quadruple $(G,\lambda,p,k)$, we say it has a \textit{(semi-)generic $ts$-stabilizer} if the action of $G$ on $X=\mathcal{S}_k(V)$ has a (semi-)generic stabilizer. In this paper we classify (semi-)generic $ts$-stabilizers of $ts$-small quadruples. 

Our main result will be given in a single table (Table~\ref{tab:Generic stabilizers for small quadruples}). In the first column we have the type of our simple algebraic group $G$, in the second column the highest weight of the irreducible $G$-module $V$, in the third column we list the rank $\ell$ of $G$, in the fourth column we have conditions on $p$ and in the fifth column we have the particular $k$ which specifies which variety $\mathcal{S}_k(V)$ we are acting on. We then have a column listing the (semi-)generic stabilizer for the action, denoted by $C_{\mathcal{S}_k(V)}$ in which we use an asterisk to indicate whether the stabilizer is not generic but semi-generic. In the columns 'Orth?' and 'Dense?', we indicate whether the module $V$ is orthogonal and whether there exists a dense orbit for the action on $\mathcal{S}_k(V)$. In the last column we give the number of the Proposition within the paper where the relevant information is obtained.

\begin{theorem*}\label{complete list of cases theorem}
    The (semi-)generic $ts$-stabilizer for a $ts$-small quadruple is given in Table~\ref{tab:Generic stabilizers for small quadruples}. In addition the existence or non-existence of a dense orbit is indicated.
\end{theorem*}

\begin{longtable}{c c c c c c c c c}
\caption{Generic $ts$-stabilizers for $ts$-small quadruples} \label{tab:Generic stabilizers for small quadruples} \\

\hline \multicolumn{1}{c}{\textbf{$G$}} & \multicolumn{1}{c}{\textbf{$\lambda$}} & \multicolumn{1}{c}{\textbf{$\ell$}}& \multicolumn{1}{c}{\textbf{$p$}}&  \multicolumn{1}{c}{\textbf{$k$}} & \multicolumn{1}{c}{\textbf{$C_{\mathcal{S}_k(V)}$}}& \multicolumn{1}{c}{Orth?}& \multicolumn{1}{c}{Dense?}& \multicolumn{1}{c}{Ref} \\ \hline 
\endhead
$A_\ell$ & $\lambda_1$ & $1$ & any & $1$ & $P_1$ & no & yes & \ref{natural modules prop} \\
 & $\lambda_1+p^i\lambda_1$ & $1$ & $<\infty$ & $1$ & $T_1$ & yes & yes & \ref{prop A_1 w1+p^iw1 k=1} \\
 & $\lambda_1+p^i\lambda_1$ & $1$ & $<\infty$ & $2$ & $U_1T_1$ & yes & yes & \ref{many cases from k=1 and k=2 proposition} \\
   & $3\lambda_1$ & $1$ & $>3$ & $1$ & $Sym(3)$ & no & yes & \ref{prop k=1 symplectic} \\
  & $3\lambda_1$ & $1$ & $>3$ & $2$ & $Alt(4)$ & no & yes & \ref{many cases from k=1 and k=2 proposition} \\
    & $4\lambda_1$ & $1$ & $>3$ & $1$ & $Alt(4)$ & yes & yes & \ref{many cases from k=1 and k=2 proposition} \\
    & $4\lambda_1$ & $1$ & $>3$ & $2$ & $Sym(3)$ & yes & yes & \ref{many cases from k=1 and k=2 proposition} \\
 
 & $\lambda_1+\lambda_2$ & $2$ & $3$ & $1$ & $U_2T_1$ & yes & yes & \ref{prop A2 w1+w2 k=1 p=3} \\
 & $\lambda_1+\lambda_2$ & $2$ & $3$ & $2$ & $U_1$ & yes & yes & \ref{many cases from k=1 and k=2 proposition} \\
 & $\lambda_1+\lambda_2$ & $2$ & $3$ & $3$ & $T_2.\mathbb{Z}_3$ & yes & yes & \ref{proposition a2 p=3 k=3} \\
 & $\lambda_1+\lambda_2$ & $2$ & $\neq 3$ & $1$ & $T_2.\mathbb{Z}_3$ & yes & yes & \ref{special adjoint cases} \\
 & $\lambda_1+\lambda_2$ & $2$ & $\neq 3$ & $4',4''$ & $T_2.\mathbb{Z}_3$ & yes & yes & \ref{proposition a2 k=4} \\
  & $\lambda_1+\lambda_3$ & $3$ & $2$ & $1$ & $T_3.Alt(4)$ & yes & yes & \ref{special adjoint cases} \\
  & $\lambda_1+\lambda_\ell$ & $\geq 3$ & $\neq 2$ & $1$ & $T_\ell$ & yes & no & \ref{adjoint module general proposition} \\
  & $\lambda_1+\lambda_\ell$ & $\geq 4, \not\equiv 1\mod 4 $ & $2$ & $1$ & $T_\ell$ & yes & no & \ref{adjoint module general proposition} \\
    & $\lambda_1+\lambda_\ell$ & $\geq 4, \equiv 1\mod 4 $ & $2$ & $1$ & $T_\ell$ & no & no & \ref{adjoint module general proposition} \\
  & $\lambda_3$ & $5$ & $2$ & $1$ & $U_8A_2T_1$ & yes & yes
  & \ref{prop A_5 w3 k=1 p=2} \\
  & $\lambda_3$ & $5$ & $\neq 2$ & $1$ & $A_2^2.\mathbb{Z}_2$ & no & yes& \ref{prop k=1 symplectic}\\
  & $\lambda_3$ & $5$ & $2$ & $2$ & $T_2.U_1.\mathbb{Z}_2 (*)$ & yes & no & \ref{large family of cases} \\
    & $\lambda_3$ & $5$ & $\neq 2$ & $2$ & $T_2.\mathbb{Z}_2.\mathbb{Z}_2$ & no & no & \ref{large family of cases} \\
\hline
$B_\ell$ & $\lambda_1$ & $\geq 2$ & $\neq 2$ & any & $P_k$ & yes & yes & \ref{natural modules prop}  \\
 & $2\lambda_2$ & $2$ & $\neq 2$ & $1$ & $T_2.\mathbb{Z}_4$ & yes & yes & \ref{special adjoint cases} \\
 & $2\lambda_2$ & $2$ & $\neq 2,5$ & $5',5''$ & $\mathbb{Z}_5.\mathbb{Z}_4$ & yes & yes & \ref{C_2 k=5 prop} \\
  & $2\lambda_2$ & $2$ & $5$ & $5'$ & $\mathbb{Z}_5.\mathbb{Z}_4$ & yes & yes & \ref{C_2 k=5 prop} \\
    & $2\lambda_2$ & $2$ & $5$ & $5''$ & $\mathbb{Z}_4$ & yes & yes & \ref{C2 prop p=5 k=5} \\
 & $\lambda_2$ & $\geq 3$ & $\neq 2$ & $1$ & $T_\ell.\mathbb{Z}_2$ & yes & no & \ref{adjoint module general proposition} \\
 & $\lambda_3$ & $3$ & any & $1$ & $U_6A_2T_1$ & yes & yes & \ref{prop B_3 w3 k=1} \\
  & $\lambda_3$ & $3$ & any & $4',4''$ & $U_6A_2T_1$ & yes & yes & \ref{proposition b3 k=4} \\
 & $\lambda_3$ & $3$ & any & $2$ & $U_5A_1A_1T_1$ & yes & yes & \ref{prop B_3 w3 k=2,3} \\
  & $\lambda_3$ & $3$ & any & $3$ & $U_3A_2T_1$ & yes & yes & \ref{prop B_3 w3 k=2,3} \\
   & $\lambda_4$ & $4$ & any & $1$ & $U_7G_2T_1$ & yes & yes & \ref{prop B4 B5 w4 w5 k=1} \\
   & $\lambda_4$ & $4$ & $2$ & $2$ & $U_5A_1A_1$ & yes & yes & \ref{many cases from k=1 and k=2 proposition} \\
   & $\lambda_4$ & $4$ & $\neq 2$ & $2$ & $A_1(A_2.\mathbb{Z}_2)$ & yes & yes & \ref{many cases from k=1 and k=2 proposition} \\
   & $\lambda_4$ & $4$ & any & $3$ & $A_1$ & yes & yes & \ref{proposition b4 k=3} \\
     & $\lambda_4$ & $4$ & any & $8'$ & $A_2.\mathbb{Z}_2$ & yes & yes & \ref{proposition b4 P_7} \\
     & $\lambda_4$ & $4$ & $\neq 2$ & $8''$ & $A_1^3$ & yes & no & \ref{proposition b4 P_8 p not 2} \\
     & $\lambda_4$ & $4$ & $2$ & $8''$ & $A_1^3(*)$ & yes & no & \ref{proposition b4 P_8 p 2} \\
     & $\lambda_4$ & $4$ & any & $7$ & $T_2.\mathbb{Z}_2$ & yes & no & \ref{proposition b4 k = 7} \\
   & $\lambda_5$ & $5$ & $2$ & $1$ & $U_{14}B_2T_1$ & yes & yes & \ref{prop B4 B5 w4 w5 k=1} \\
      & $\lambda_5$ & $5$ & $\neq 2$ & $1$ & $A_4.\mathbb{Z}_2$ & no & yes & \ref{prop k=1 symplectic} \\
         & $\lambda_6$ & $6$ & $2$ & $1$ & $A_2^2.\mathbb{Z}_2$ & no & no & \ref{many cases from k=1 and k=2 proposition} \\
   & $\lambda_6$ & $6$ & $2$ & $1$ & $(A_2.\mathbb{Z}_2)^2.\mathbb{Z}_2$ & yes & yes & \ref{many cases from k=1 and k=2 proposition} \\\hline
    $C_\ell$ & $\lambda_1$ & $\geq 3$ & any & any & $P_k$ & no & yes & \ref{natural modules prop} \\
    & $2\lambda_1$ & $\geq 3$ & $\neq 2$ & $1$ & $T_\ell.\mathbb{Z}_2$ & yes & yes & \ref{adjoint module general proposition} \\
    & $\lambda_2$ & $3$ & $3$ & $1$ & $U_6A_1T_1$ & yes & yes & \ref{prop C3 w2 k=1} \\
    & $\lambda_2$ & $3$ & $3$ & $2$ & $U_1T_1.\mathbb{Z}_2$ & yes & yes & \ref{many cases from k=1 and k=2 proposition} \\
    & $\lambda_2$ & $3$ & $\neq 3$ & $1$ & $A_1^3.3$ & yes & yes & \ref{special cases type C w2} \\
    & $\lambda_2$ & $3$ & $\neq 3,7$ & $7',7''$ & $\mathbb{Z}_7.\mathbb{Z}_6$ & yes & yes & \ref{C_3 k=7 prop} \\
        & $\lambda_2$ & $3$ & $7$ & $7'$ & $\mathbb{Z}_7.\mathbb{Z}_6$ & yes & yes & \ref{C_3 k=7 prop} \\
    & $\lambda_2$ & $3$ & $7$ & $7''$ & $\mathbb{Z}_6$ & yes & yes & \ref{C3 p=7 k=7 prop} \\

    & $\lambda_2$ & $4$ & $2$ & $1$ & $A_1^4.Alt(4)$ & yes & yes & \ref{special cases type C w2} \\
    & $\lambda_2$ & $\geq 4$ & $\neq 2$ & $1$ & $A_1^\ell$ & yes & no & \ref{C_l lambda_2 general proposition} \\
     & $\lambda_2$ & $\geq 5,\not\equiv 2\mod 4$ & $2$ & $1$ & $A_1^\ell$ & yes & no & \ref{C_l lambda_2 general proposition} \\
    & $\lambda_2$ & $\geq 5,\equiv 2\mod 4$ & $2$ & $1$ & $A_1^\ell$ & no & no & \ref{C_l lambda_2 general proposition} \\
      & $\lambda_2$ & $3$ & $\neq 3$ & $2$ & $T_1.Sym(3)$ & yes & no & \ref{F4 and C3 p not 3 k=2} \\
      & $\lambda_3$ & $3$ & $\neq 2$ & $1$ & $A_2.\mathbb{Z}_2$ & no & yes & \ref{prop k=1 symplectic} \\
    \hline
    $D_\ell$ & $\lambda_1$ & $\geq 4$ & any & $\neq \ell-1$ & $P_k$ & yes & yes & \ref{natural modules prop} \\
     & $\lambda_1$ & $\geq 4$ & any & $\ell-1$ & $P_{\ell-1,\ell}$ & yes & yes & \ref{natural modules prop} \\
    & $\lambda_2$ & $4$ & $2$ & $1$ & $T_4.(\mathbb{Z}_2^3.Alt(4))$ & yes & yes & \ref{special cases type C w2} \\
     & $\lambda_2$ & $\geq 4$ & $\neq 2$ & $1$ & $T_\ell.\mathbb{Z}_{(2,\ell)}$ & yes & no & \ref{adjoint module general proposition} \\
     & $\lambda_2$ & $\geq 5,\not\equiv 2 \mod 4$ & $2$ & $1$ & $T_\ell.(\mathbb{Z}_{2})^{\ell-1}$ & yes & no & \ref{adjoint module general proposition} \\
     & $\lambda_2$ & $\geq 5,\equiv 2 \mod 4$ & $2$ & $1$ & $T_\ell.(\mathbb{Z}_{2})^{\ell-1}$ & no & no & \ref{adjoint module general proposition} \\
    & $\lambda_6$ & $6$ & $2$ & $1$ & $U_{14}B_3T_1$ & yes & yes & \ref{D_6 w6 k=1 p=2} \\
        & $\lambda_6$ & $6$ & $\neq 2$ & $1$ & $A_5.\mathbb{Z}_2$ & no & yes & \ref{prop k=1 symplectic} \\
     & $\lambda_6$ & $6$ & $2$ & $2$ & $A_1^3.U_1.\mathbb{Z}_2 (*)$ & yes & no & \ref{large family of cases} \\
    & $\lambda_6$ & $6$ & $\neq 2$ & $2$ & $A_1^3.\mathbb{Z}_2.\mathbb{Z}_2$ & no & no & \ref{large family of cases} \\
    \hline
   $G_2$ & $\lambda_1$ & $2$ & $\neq 2$ & $1$ & $U_{5}A_1T_1$ & yes & yes & \ref{prop G_2 w1 p not 2 k=1} \\
   & $\lambda_1$ & $2$ & $2$ & $1$ & $U_{5}A_1T_1$ & no & yes & \ref{prop k=1 symplectic} \\
    & $\lambda_1$ & $2$ & $2$ & $2$ & $U_{3}A_1T_1$ & no & yes & \ref{prop G_2 w1 k=2} \\
        & $\lambda_1$ & $2$ & $\neq 2$ & $2$ & $U_{3}A_1T_1$ & yes & yes & \ref{prop G_2 w1 k=2} \\
  & $\lambda_1$ & $2$ & $2$ & $3$ & $A_2$ & no & yes & \ref{G2 k=3 p=2} \\
  & $\lambda_1$ & $2$ & $\neq 2$ & $3$ & $A_2$ & yes & yes & \ref{G2 k=3 p not 2} \\
  & $\lambda_2$ & $2$ & $\neq 3$ & $1$ & $T_2.\mathbb{Z}_6$ & yes & yes & \ref{special adjoint cases} \\\hline
   $F_4$ & $\lambda_4$ & $4$ & $3$ & $1$ & $U_{14}G_2T_1$ & yes & yes & \ref{prop f4 w4 p=3 k=1} \\ 
   & $\lambda_4$ & $4$ & $\neq 3$ & $1$ & $D_4.\mathbb{Z}_3$ & yes & yes & \ref{many cases from k=1 and k=2 proposition} \\
   & $\lambda_4$ & $4$ & $3$ & $2$ & $U_1A_2.\mathbb{Z}_2$ & yes & yes & \ref{many cases from k=1 and k=2 proposition} \\
   & $\lambda_4$ & $4$ & $\neq 3$ & $2$ & $A_2.Sym(3)$ & yes & yes & \ref{F4 and C3 p not 3 k=2} \\
    & $\lambda_1$ & $4$ & $\neq 2$ & $1$ & $T_4.\mathbb{Z}_2$ & yes & no & \ref{adjoint module general proposition} \\ \hline
    $E_6$ & $\lambda_2$ & $6$ & any & $1$ & $T_6$ & yes & no & \ref{adjoint module general proposition} \\\hline
   $E_7$ & $\lambda_7$ & $7$ & $2$ & $1$ & $U_{26}F_4T_1$ & yes & yes & \ref{E7 k=1 p=2 w7} \\
    & $\lambda_1$ & $7$ & $\neq 2$ & $1$ & $T_7.\mathbb{Z}_2$ & yes & no & \ref{adjoint module general proposition} \\
    & $\lambda_1$ & $7$ & $2$ & $1$ & $T_7.\mathbb{Z}_2$ & no & no & \ref{adjoint module general proposition} \\
    & $\lambda_7$ & $7$ & $2$ & $2$ & $D_4.U_1.\mathbb{Z}_2 (*)$ & yes & no & \ref{large family of cases} \\
    & $\lambda_7$ & $7$ & $\neq 2$ & $2$ & $D_4.\mathbb{Z}_2.\mathbb{Z}_2$ & no & no & \ref{large family of cases} \\
   \hline
   $E_8$ & $\lambda_8$ & $8$ & any & $1$ & $T_8.\mathbb{Z}_2$ & yes & no & \ref{adjoint module general proposition} \\\hline
\end{longtable}

We are quickly able to determine whether a dense orbit on $X$ exists, as the set of points in $X$ that have stabilizers of minimal dimension is open in $X$ (see for example \cite[Lemma~2.1]{martin2015generic}). This means that the dimension of a (semi-)generic stabilizer is actually the minimal dimension of any stabilizer, and if this is larger than $\dim G-\dim X$, there is no dense orbit. There is extensive interest in the literature around the existence of a dense orbit. In particular if $G$ is a reductive algebraic group and $V$ is an irreducible rational $G$-module such that $G$ has a dense orbit on $V$, the pair $(G,V)$ is called a \textit{prehomogeneous vector space}, often shortened to $PV$-space. A classification of $PV$-spaces was determined in characteristic zero in \cite{satokimura} and extended to positive characteristic in \cite{chen1}\cite{chen2}. The fact that this classification covers all semisimple algebraic groups implies that the density question is already understood for the action on $\mathcal{G}_k(V)$. Indeed $G$ has a dense orbit on $\mathcal{G}_k(V)$ if and only if $GL_k\otimes G$ has a dense orbit on $V_1\otimes V$, where $V_1$ is the natural module for $GL_k$. The next result summarises the answer to the dense-orbit question for the action on orthogonal and symplectic Grassmannians.

\begin{theorem*}\label{density spaces theorem}
Let $G$ be a simple algebraic group over an algebraically closed field of characteristic $p$, and $V$ a self-dual non-trivial irreducible $G$-module of dimension $d$ and highest weight $\lambda$. For $1\leq k\leq \frac{d}{2}$, the action of $G$ on $\mathcal{S}_k(V)$ has a dense orbit if and only if $\dim G\geq \dim \mathcal{S}_k(V)$, the zero-weight space of $V$ has dimension at most $2$, and $(G,\lambda,p,k)$ is not one of the following:
\begin{enumerate}[label=(\roman*)]
    \item $(A_5,\lambda_3,any,2)$;
    \item $(B_4,\lambda_4,any,7)$;
    \item $(B_4,\lambda_4,any,8'')$;
    \item $(D_6,\lambda_6,any,2)$;
    \item $(E_7,\lambda_7,any,2)$.
\end{enumerate}
\end{theorem*}

\begin{remark*}
    The only $ts$-small quadruples $(G,\lambda,p,k)$ with a zero-weight space of dimension at least $3$ have $k=1$, with either $V_G(\lambda)$ a composition factor of the adjoint module for $G$, or $G=C_\ell$ and $\lambda = \lambda_2$. 
\end{remark*}

Denote by $Cl(V)$ a classical group with natural module $V$.
Given that $X=\mathcal{G}_k(V)$ and $X=\mathcal{S}_k(V)$ are varieties of cosets of the form $Cl(V)/P$, where $P$ is a parabolic subgroup of $Cl(V)$, the action of $G$ on $X$ has a dense orbit if and only if there is a dense $(G,P)$-double coset in $Cl(V)$. We now seek to complete the answer to the question of existence of a dense double coset in the following sense. In Theorem~\ref{theorem double coset density} we classify pairs $(G,H)$ of closed maximal connected subgroups of $Cl(V)$ such that there exists a dense $(G,H)$-double coset in $Cl(V)$. In order to do this we first need to determine which semisimple groups acting tensor-decomposably on an irreducible module $V$ have a dense orbit on $\mathcal{S}_k(V)$. The possibilities for such subgroups are given by Proposition~\ref{tensor decomps}. They are $SO_n\otimes SO_m\leq SO_{mn}$, $Sp_{2n}\otimes SO_m\leq Sp_{2mn}$, $Sp_{2n}\otimes Sp_{2m}\leq SO_{4mn}$, with $p\neq 2$ if one of the factors is an orthogonal group acting on an odd-dimensional vector space. 
The strategy is similar to the one employed for the proof of Theorem~\ref{complete list of cases theorem}, although we only care about the connected component of (semi-)generic stabilizers. This is achieved in the following theorem.

\begin{theorem*}\label{maximal semisimple theorem}
Let $G=Cl(V_1)\otimes Cl(V_2)$ be a maximal subgroup of either $SO(V)$ or $Sp(V)$ with $V=V_1\otimes V_2$ and suppose $k\leq \frac{1}{2}\dim V_1\dim V_2$. Then the action of $G$ on $\mathcal{S}_k(V)$ has a dense orbit if and only if $(G,k)$ is one of the following:

\begin{enumerate}[label=(\roman*)]
    \item $(Sp(V_1)\otimes SO(V_2),1)$ and $(K^*G,V)$ is a prehomogeneous vector space, as classified in \cite{chen1}\cite{chen2};
    \item $(Sp_2\otimes Sp_{2n},1)$ with $n\geq 1$;
    \item $(Sp_2\otimes Sp_{2n},2)$ with $n\geq 2$;
    \item $(Sp_2\otimes Sp_{2n},3)$ with $n\geq 3$;
    \item $(Sp_2\otimes Sp_{2n},k)$ with $k=(2n)'$ or $k=(2n)''$ and $1\leq n\leq 3$;
    \item $(Sp_2\otimes Sp_{2n},2n-1)$ with $1\leq n\leq 3$;
    \item $(Sp_4\otimes Sp_{2n},1)$ with $n\geq 2$.
\end{enumerate}
\end{theorem*}

We are now  ready to classify pairs $(G,H)$ of closed maximal connected subgroups of $Cl(V)$ such that there exists a dense $(G,H)$-double coset in $Cl(V)$. One particular class of such pairs (see case $(iii)(b)$ in the following theorem) arises from \textit{spherical} subgroups. These are reductive subgroups $G$ such that there is a dense $(G,B)$-double coset, where $B$ is a Borel subgroup.  Spherical subgroups have been classified by Kr\"amer in characteristic $0$ and by Knop and R\"ohrle in \cite{spherical} for arbitrary characteristic.

\begin{theorem*}\label{theorem double coset density}
    Let $\Gamma$ be a classical group $Cl(V)$. Let $G,H$ be a pair of maximal connected subgroups of $\Gamma$. Then there exists a dense $(G,H)$-double coset in $\Gamma$ if and only if one of the following holds:
    \begin{enumerate}[label=(\roman*)]
        \item $G$ and $H$ are both reductive and $\Gamma = GH$. Such factorizations were classified in \cite{factorizations}.
        \item $G$ and $H$ are both parabolic subgroups.
        \item After possibly interchanging $G$ and $H$, we have that $H$ is parabolic, $G$ is reductive and one of the following holds:
        \begin{enumerate}[label=(\alph*)]
            \item $\Gamma=SL(V)$, $H=P_k$ (or $\Gamma=Sp(V)$, $H=P_k$ with $k=1$) and $(GL_k\otimes G,K^k\otimes V)$ is a prehomogeneous vector space, as classified in \cite{satokimura}\cite{chen1}\cite{chen2};
            \item $G$ is the stabilizer of a subspace $X$ of $V$ and either $X$ is non-degenerate, or $p=2$ and $X$ is non-singular of dimension $1$ with $V$ orthogonal. In this case $G$ is spherical, so $H$ is arbitrary.
            \item $G=Cl(V_1)\otimes Cl(V_2)$, $V = V_1\otimes V_2$, and $G$ has a dense orbit on the variety of totally singular subspaces corresponding to $\Gamma/H$, as detailed in Theorem~\ref{maximal semisimple theorem}.
            \item $G$ is simple and irreducible on $V$ and has a dense orbit on the variety of totally singular subspaces corresponding to $\Gamma/H$, as detailed in Theorem~\ref{density spaces theorem}.
        \end{enumerate}
    \end{enumerate}
    
\end{theorem*}
\begin{remark*}
    Dropping the maximality assumption presents considerable challenges. Even the case where both $G$ and $H$ are reductive does not have a general solution, as the results in \cite{brundan} require a technical condition on $H$ and $K$.
\end{remark*}
\begin{remark*}
    The double coset density question remains open when $\Gamma$ is an exceptional group. In case $(iii)$ of Theorem~\ref{theorem double coset density} the only information available for $\Gamma$ exceptional is when $G$ is of maximal rank (see \cite{duckworth}). The complete classification will be the subject of forthcoming work.
\end{remark*}

Let us conclude with a related question. Given two maximal connected subgroups of $Cl(V)$, are there finitely many double cosets? Of course, the existence of finitely many double cosets implies the existence of a dense double coset, while the opposite need not be true. For example $G=A_{\ell}$ for $\ell\geq 8$ has a dense orbit on $X=\mathcal{G}_2(V)$ where $V=V_{G}(\lambda_2)$, but also has infinitely many orbits (\cite{finite}). Therefore there is a dense $(G,P_2)$-double coset in $SL(V)$, as well as infinitely many double cosets.

On the other hand, it was shown in \cite{finite} that a simple group $G$ having finitely many orbits on $X=\mathcal{G}_1(V)$, where $V$ is an irreducible $G$-module, is equivalent to $G$ having a dense orbit on $X$. This result was replicated for self-dual modules when $X=\mathcal{S}_1(V)$ and $X=\mathcal{S}_2(V)$ in \cite{rizzoli}\cite{Rizzoli2}.

Strikingly, we shall conjecture that there is only one exception to this equivalence when $X$ is an orthogonal or symplectic Grassmannian. 
\begin{conjecture*}
    Let $G$ be a simple connected irreducible subgroup of either $SO(V)$ or $Sp(V)$. Then unless $(G,\lambda,p,k) = (C_3,\lambda_2,p,7)$ with $p\neq 3$, the action of $G$ on $\mathcal{S}_k(V)$ has a dense orbit if and only if $G$ acts on $\mathcal{S}_k(V)$  with finitely many orbits.
\end{conjecture*}

The paper will be structured as follows. In Section~\ref{preliminary results section} we shall consolidate the notation and background material, develop the tools for our analysis, and determine the complete list of $ts$-small quadruples, which we divide into three tables (see Proposition~\ref{proposition candidates small quadruples}). We will devote a section to each table. 

Given a $ts$-small quadruple we then proceed to determine its (semi-)generic $ts$-stabilizer. In Section~\ref{section quadruples with finitely many orbits} we handle the cases where we already had finitely many orbits on all $k$-spaces. In Section~\ref{infinite families section} we deal with the cases having a large zero-weight space and in Section~\ref{remainin quadruples section} we handle the remaining cases. This completes the proof of Theorem~\ref{existence theorem} and Theorem~\ref{complete list of cases theorem}.

We then shift our attention to the double coset density question. In Section~\ref{section proof of theorem maximal semisimple theorem}
we prove Theorem~\ref{maximal semisimple theorem}. Finally, in Section~\ref{section proof of last theorem} we use all of the previous results to prove Theorem~\ref{theorem double coset density}.

\textbf{Acknowledgements}
The author is extremely grateful to Prof. Martin Liebeck to Prof. Donna Testerman for their lasting support and helpful conversations. The author would also like to thank the anonymous reviewer for their very careful reading of the manuscript and their many suggestions for fixing gaps and improving the overall quality. 

The initial part of this work was undertaken during the author's tenure as INI-Simons Post Doctoral Research Fellow. The author would like to thank INI and DPMMS for support and hospitality during this fellowship, which was supported by the Simons Foundation (Award ID 316017) and by EPSRC (grant number EP/R014604/1).

The work was further supported by the Swiss National Science Foundation (grant number 207730).

\section{Preliminary results}\label{preliminary results section}

In this section we gather some useful results, develop notation and conclude with a complete list of $ts$-small quadruples.

\subsection{Bilinear forms}\label{bilinear forms section}
We start by fixing the notation for the action of an orthogonal group on its natural module $V_{nat}$. To do this we use the standard notation for its root system: we take an orthonormal basis $\epsilon_1,\dots,\epsilon_\ell$ of the $\ell$-dimensional Euclidean space, and take simple roots $\alpha_i = \epsilon_i-\epsilon_{i+1}$ for $i<\ell$ and $\alpha_\ell = \epsilon_\ell$ or $\epsilon_{\ell-1}+\epsilon_\ell$ according as $G =B_\ell$ or $D_\ell$.

If $G=D_\ell$, then $V_{nat}$ has (hyperbolic) basis $e_1,f_1,\dots,e_\ell,f_\ell$ on which root elements act by
\begin{flalign*}
x_{\epsilon_i-\epsilon_j}(t) &: e_j\mapsto e_j+te_i,\quad f_i\mapsto f_i-tf_j, &\\
x_{-\epsilon_i+\epsilon_j}(t) &: e_i\mapsto e_i+te_j,\quad f_j\mapsto f_j-tf_i, &\\
x_{\epsilon_i+\epsilon_j}(t) &: f_j\mapsto f_j+te_i,\quad f_i\mapsto f_i-te_j, &\\
x_{-\epsilon_i-\epsilon_j}(t) &: e_j\mapsto e_j-tf_i,\quad e_i\mapsto e_i+tf_j, &
 \end{flalign*}
 while fixing the basis vectors that are not listed. 

 If $G=B_\ell$, then $V_{nat}$ has (hyperbolic) basis $v_0,e_1,f_1,\dots,e_\ell,f_\ell$ on which root elements act by 
\begin{flalign*}
x_{\epsilon_i-\epsilon_j}(t) &: e_j\mapsto e_j+te_i,\quad f_i\mapsto f_i-tf_j, &\\
x_{-\epsilon_i+\epsilon_j}(t) &: e_i\mapsto e_i+te_j,\quad f_j\mapsto f_j-tf_i, &\\
x_{\epsilon_i+\epsilon_j}(t) &: f_j\mapsto f_j+te_i,\quad f_i\mapsto f_i-te_j, &\\
x_{-\epsilon_i-\epsilon_j}(t) &: e_j\mapsto e_j-tf_i,\quad e_i\mapsto e_i+tf_j, &\\
x_{\epsilon_i}(t) &: v_0\mapsto v_0+2te_i,\quad f_i\mapsto f_i-tv_0-t^2e_i, &\\
x_{-\epsilon_i}(t) &: v_0\mapsto v_0-2tf_i,\quad e_i\mapsto e_i+tv_0-t^2f_i, &
 \end{flalign*}
 while fixing the basis vectors that are not listed.

The following two results describe the structure of maximal connected subgroups of classical groups. Recall that by $Cl(V)$ we denote a classical group with natural module $V$.

\begin{theorem}\cite{subgroupstructure}\label{subgroupstructure}
Let $H$ be a closed connected subgroup of $G=Cl(V)$. Then one of the following holds:

\begin{enumerate}[label=(\roman*)]
\item $H\leq Stab_G(X)$ with $X\leq V$ a proper non-zero subspace which is either totally singular or non-degenerate, or $p=2$, $G=SO(V)$ and $X$ is non-singular of dimension $1$;
\item $V=V_1\otimes V_2$ and $H$ lies in a subgroup of the form $Cl(V_1)\otimes Cl(V_2)$ acting naturally on $V_1\otimes V_2$ with $\dim V_i\geq 2$ for $i=1,2$;
\item $H$ is a simple algebraic group acting irreducibly on $V$ and $V|_H$ is tensor indecomposable.
\end{enumerate}
\end{theorem}

The possibilities for the second case of Theorem~\ref{subgroupstructure} are given by the following proposition.

\begin{proposition}\cite[Prop~2.2]{subgroupstructure}\label{tensor decomps}
Suppose $V=V_1\otimes V_2$ and $f_i$ is a non-degenerate bilinear form on $V_i$.
\begin{enumerate}[label=(\roman*)]
\item There is a unique non-degenerate bilinear form $f=f_1\otimes f_2$ on $V$ such that\\ $f(u_1\otimes u_2,v_1\otimes v_2)=f_1(u_1,v_1)f_2(u_2,v_2)$ for all $u_i,v_i\in V_i$.
\item $f$ is symmetric if $f_1,f_2$ are both alternating or both symmetric, and $f$ is alternating otherwise.
\item $f$ is preserved by $I(V_1)\circ I(V_2)$ acting naturally on the tensor product, where $I(V_i)$ is the stabilizer in $GL(V_i)$ of $f_i$.
\item If $p=2$ then there is a unique quadratic form $Q$ on $V$, with associated bilinear form $f$, such that $Q(v_1\otimes v_2)=0$ for all $v_i\in V_i$ and $Q$ is preserved by $Sp(V_1)\otimes Sp(V_2)$.
\end{enumerate}
\end{proposition}

The following lemma gives the dimension of the symplectic and orthogonal Grassmannians we are acting on.

\begin{lemma}\label{dimension totally singular subspaces}
Let $V$ be either a symplectic or orthogonal geometry of dimension $d$ over an algebraically closed field. Then
\[ \dim \mathcal{S}_k(V)= kd-\frac{3k^2+\epsilon_V k }{2},\]
where $\epsilon_V$ is $1$ or $-1$
according as $V$ is orthogonal or symplectic.
\end{lemma}

\begin{proof}
If $V$ is orthogonal with $d=2\ell$ and $k=\frac{d}{2}-1$, we have $\dim \mathcal{S}_k(V)=\dim D_\ell-\dim P_{\ell-1,\ell}$. In all other cases the dimension is simply given by  $\dim Cl(V) / P_k=\dim Cl(V)-\dim P_k$ for $Cl(V)=B_{\frac{d-1}{2}},C_{\frac{d}{2}},$ or $D_{\frac{d}{2}}$ as appropriate.
\end{proof}

Recall that a $ts$-small quadruple is a $4$-tuple of the form $(G,\lambda,p,k)$, with $V=V_G(\lambda)$ a self-dual irreducible $G$-module, and $\dim G\geq \dim \mathcal{S}_k(V)$. Lemma~\ref{dimension totally singular subspaces} gives the following dimension bound.

\begin{lemma}\label{dimension bound}
Let $(G,\lambda,p,k)$ be a $ts$-small quadruple. Suppose that $V=V_G(\lambda)$ has dimension $d$.
Then \[ \dim G\geq  kd-\frac{3k^2+\epsilon_V k }{2},\]
where $\epsilon_V$ is $1$ or $-1$
according as $V$ is orthogonal or symplectic.
\end{lemma}
\begin{proof}
This follows directly from Lemma~\ref{dimension totally singular subspaces}.
\end{proof}

A self-dual module is either orthogonal or symplectic. The following lemma provides a useful criterion in odd characteristic.

\begin{lemma}\cite[Lemma~78-79]{steinberg}\label{froebenius schur}
Let $G$ be a simple simply connected algebraic group and $V=V_G(\lambda)$ a self-dual $G$-module in characteristic $p\neq 2$. Then if $Z(G)$ has no element of order $2$ the module $V$ is orthogonal. 

Otherwise let $z$ be the only element of order $2$ in $Z(G)$, except for the case $G=D_\ell$ with even $\ell$, where $z$ is the element of $Z(G)$ such that $G/\langle z \rangle \simeq SO_{2\ell}(k)$. Then the module $V$ is orthogonal if $\lambda(z) = 1$, and symplectic otherwise. 
\end{lemma}

Note that the value $\lambda(z)$ can be computed by \cite[Appendix~A.2]{lubeck}.
If $p=2$, for certain modules we can determine whether $V$ is symplectic or orthogonal thanks to \cite{mikko}. More explicit descriptions of some of these forms can be found in \cite{garibaldi_nakano_2016} and \cite{Babi}.

When dealing with maximal totally singular subspaces of an orthogonal module $V$ of dimension $2\ell$, we need to be able to distinguish between the two $D_\ell$-orbits on $\mathcal{S}_\ell(V)$. The following lemma provides an easy way to do so.
\begin{lemma}\label{lemma interesction maximal totally singular}\cite[\nopp 22.13]{Aschbacher_2000}
    Let $V$ be an orthogonal module of even dimension and $U$, $W$ two maximal totally singular subspaces of $V$. Then $U$ and $V$ are in the same $SO(V)$-orbit if and only if $\dim U - \dim U\cap V$ is even.
\end{lemma}

We conclude this section with a lemma concerning orthogonality of weight spaces in a self-dual irreducible module.

\begin{lemma}\label{weight spaces orthogonality lemma}
    Let $V = V_G(\lambda)$ be a self-dual irreducible $KG$-module. Let $(\cdot,\cdot)$ be an irreducible symmetric or alternating bilinear form on $V$ preserved by $G$.
    Then the following hold:
    \begin{enumerate}[label = (\alph*)]
        \item $V_{\mu}$ is totally singular for all non-zero weights $\mu$;
                \item for any two weights $\mu,\nu$ such that $\mu\neq\pm\nu$ the weight spaces $V_\mu$ and $V_{\nu}$ are orthogonal to each other;
        \item $V_0$ is non-degenerate.
    \end{enumerate}
\end{lemma}
\begin{proof}
    If $\mu\neq 0$, there exists $t\in T$ and $\kappa\neq \pm 1$ such that $t.v = \kappa v$ for all $v\in V_{\mu}$. Therefore for all $u,v\in V_{\mu}$ we have $(u,v) = (\kappa u,\kappa v) = \kappa^2 (u,v)$, implying that $(u,v) = 0$. Similarly, if $V$ is an orthogonal module we get $Q(v) = \kappa^2 Q(v)$, implying $Q(v) = 0$ for all $v\in V_{\mu}$. This proves $(a)$.

    Let $v,u$ be two weight vectors in distinct and non-opposite weight spaces. Then there exists $t\in T$ such that $t.v = \kappa_1 v$ and $t.u = \kappa_2 u$ with $\kappa_1\neq \kappa_2^{-1}$. Therefore $(u,v) = \kappa_1\kappa_2 (u,v)$, which implies $(u,v) = 0$, proving $(b)$.

    By $(b)$, the zero weight space $V_0$ is orthogonal to all non-zero weight spaces. Therefore any singular vector $v$ in the radical of $V_0$ is a singular vector in the radical of $V$, which implies $v=0$ since the form is non-degenerate. This proves that $V_0$ is itself non-degenerate.
\end{proof}

\subsection{Clifford theory}
Let $G\leq GL(V)$ be a subgroup acting completely reducibly and homogeneously on $V$. The following lemma shows that $G$ must preserve a tensor product structure on $V$.

\begin{lemma}\cite[Lemma~4.4.3]{KL}\label{tensor decomposition homogeneus}
    Let $V$ be an irreducible $KG$-module and suppose that $S\leq G$ acts completely reducibly and homogeneously on $V$, with $s\geq 2$ irreducible summands of dimension $r$. Then the following hold:
        \begin{enumerate}[label=(\roman*)]
        \item there is a tensor decomposition $V=V_1\otimes V_2$ (where $\dim V_1 = r,\, \dim V_2 = s$), such that $S\leq GL(V_1)\otimes 1$ and $C_{GL(V)}(S) = 1\otimes GL(V_2)$;
        \item $C_{GL(V)}(C_{GL(V)}(S))=GL(V_1)\otimes 1$;
        \item $N_{GL(V)}(S) = N_{GL(V_1)}(S)\otimes GL(V_2) $;
        \item the irreducible $KS$-submodules of $V$ are precisely the subspaces $V_1\otimes v$, where $0\neq v\in V_2$.
    \end{enumerate}
\end{lemma}

The following lemma shows that no cyclic extension of $G$ can act irreducibly on $V$.

\begin{lemma}\label{cyclic extension lemma}
    Let $V$ be an irreducible $KG$-module and suppose that $S\leq G$ acts completely reducibly and homogeneously on $V$, with $s\geq 2$ irreducible summands of dimension $r$. Let $S\langle \tau\rangle\leq G$ be a cyclic extension of $S$. Then $S\langle \tau \rangle$ does not act irreducibly on $V$.
\end{lemma}

\begin{proof}
We apply Lemma~\ref{tensor decomposition homogeneus} to get $V=V_1\otimes V_2$ with $\dim V_1 = r$, $S\leq GL(V_1)\otimes 1$ and $\tau =\tau_1\otimes \tau_2\in GL(V_1)\otimes GL(V_2)$. Since $\tau_2$ stabilises at least one $1$-space of $V_2$, by part $(iv)$ of Lemma~\ref{tensor decomposition homogeneus}, the element $\tau$ stabilises one irreducible $KS$-submodule of $V$.
\end{proof}

\begin{lemma}\label{cyclic extension V_1+V_2 lemma}
    Let $V$ be an irreducible $KG$-module and suppose that $S\leq G$ acts homogeneously on $V$ as the sum of $2$ irreducible summands. Let $S\langle \tau\rangle\leq G$ be a cyclic extension of $S$. Then $S\langle \tau \rangle$ fixes $1$, $2$, or all $KS$-submodules of $V$.
\end{lemma}

\begin{proof}
    From the proof of Lemma~\ref{cyclic extension lemma} we have $V=V_1\otimes V_2$, with $\dim V_2 =2$, $S\leq GL(V_1)\otimes 1$ and $\tau=\tau_1\otimes \tau_2\in GL(V_1)\otimes GL(V_2)$. Then $\tau_2$ stabilises either $1$, $2$ or all $1$-spaces of $V_2$, concluding.
\end{proof}

\subsection{Generic stabilizers}
In this section we gather some of the essential lemmas that will allow us to determine (semi-)generic stabilizers. As mentioned in the introduction, a (semi-)generic stabilizer realises the minimum dimension of any stabilizer. This follows from the next two results.

\begin{lemma}\cite[Lemma~3.7]{Newstead1978}\label{minimum dimension lemma}
    Let an algebraic group $G$ act on a quasi-projective variety $X$. For any $t\in\mathbb{N}\cup \{0\}$, the set $\{x\in X\lvert \dim G_x\geq t\}$ is closed.
\end{lemma}

\begin{corollary}\label{minimum dimension generic}
Let an algebraic group $G$ act on an irreducible quasi-projective variety $X$ with (semi-)generic stabilizer $S$.
Then for all $x\in X$ we have $\dim G_x\geq \dim S$.
\end{corollary}

\begin{proof}
By assumption there exists an open set $Y$ such that all elements of $Y$ have stabilizer isomorphic to $S$. By Lemma~\ref{minimum dimension lemma} we then find that $\dim G_x\geq \dim S$ for all $x\in X$.
\end{proof}

Let us consider some of the methods used in \cite{generic}.
In particular we are interested in the \textit{localization to a subvariety} approach \cite[§4.1]{generic}. Let $X$ be a variety on which a simple algebraic group $G$ acts. Let $Y$ be a subvariety of $X$ and $x\in X$. The \textit{transporter} in $G$ of $x$ into $Y$ is \[\mathrm{Tran}_G(x,Y)=\{g\in G:g.x\in Y\}.\]
Let $\phi:G\times X\rightarrow X$ be the orbit map. 
\begin{lemma}\cite[Lemma~4.1.1]{generic}\label{transporter dimension lemma}
For $y\in Y$ the following hold:
\begin{enumerate}[label=(\roman*)]
    \item $\dim \mathrm{Tran_G}(y,Y) = \dim \phi^{-1} (y)$;
    \item $\codim \mathrm{Tran_G}(y,Y)=\dim (\overline{G.y})-\dim (\overline{G.y\cap Y}).$
\end{enumerate}
    
\end{lemma}
If $Y$ is a subvariety of $X$, a point $y\in Y$ is called \textit{$Y$-exact} if \[\codim \mathrm{Tran_G}(y,Y)=\codim Y.\] 
\begin{lemma}\cite[Lemma~4.1.2]{generic}\label{loc to a subvariety open set lemma}
Let $\hat{Y}$ be a dense open subset of $Y$. Suppose that all points in $\hat{Y}$ are $Y$-exact. Then $\phi(G\times \hat{Y})$ contains a dense open subset of $X$.
\end{lemma}
Such a set $\hat{Y}$ is sufficiently representative of the $G$-action on $X$, which leads to the following lemma.
\begin{lemma}\cite[Lemma~4.1.3]{generic}\label{loc to a subvariety lemma}
Let $\hat{Y}$ be a dense open subset of $Y$. Let $C$ be a subgroup of $G$ containing $G_X$. Suppose that for each $y\in\hat{Y}$ the following is true:
\begin{enumerate}[label=(\roman*)]
    \item $y$ is $Y$-exact;
    \item $G_y$ is conjugate to $C$.
\end{enumerate}
Then $C/G_X$ is the generic stabilizer in the action of $G$ on $X$.
\end{lemma}
We similarly derive a criterion for proving that there does not exist a generic stabilizer.
\begin{lemma}\label{no generic stabilizer lemma}
Assume that $Y$ is not finite, and let $\hat{Y}$ be a dense open subset of $Y$. Suppose that for each $y\in\hat{Y}$ the following is true:
\begin{enumerate}[label=(\roman*)]
    \item $y$ is $Y$-exact;
    \item for all $y'\in \hat{Y}\setminus \{y\}$ we have that $G_y$ is not conjugate to $G_{y'}$.
\end{enumerate}
Then there is no generic stabilizer in the action of $G$ on $X$.
\end{lemma}

\begin{proof}
    By Lemma~\ref{loc to a subvariety open set lemma} there is a dense open subset $U_1$ of $X$, contained in $\phi(G\times \hat{Y})$ which by assumption is the union of $G$-orbits with pairwise non-conjugate stabilizers. Given any such orbit $G.y$ for some $y\in \hat{Y}$, we have \[\dim(\overline{G.y}) =\dim(\overline{G.y})-\dim (\overline{G.y\cap Y}) = \codim Y <\dim X, \] since $y$ is $Y$-exact combined with $\dim Y\geq 1$ and Lemma~\ref{transporter dimension lemma}.
    
    Assume that there is a generic stabilizer in the action of $G$ on $X$. Then there is a dense open subset $U_2$ of $X$ such that $G_{x_1}$ and $G_{x_2}$ are conjugate for all $x_1,x_2\in U_2$. Taking the intersection of $U_1$ and $U_2$ we get an open dense subset $U$ of $X$ with the same property. Therefore, given any two $x_1,x_2\in U$, we must have $G.x_1 = G.x_2$, i.e. $U$ consists of a single $G$-orbit. This implies that there is a dense orbit, contradicting $\dim(\overline{G.y}) < \dim X$.
\end{proof}

\subsection{Spin modules}\label{spin modules section}
We set up spin modules following \cite{popov}.
Let $\{e_1,\dots,e_n,e_{n+1},\dots, e_{2n}\}=\{e_1,\dots,e_n,f_{1},\dots, f_{n}\}$ be a standard basis for the $2n$-dimensional $K$-vector space $V=V_{2n}$ with quadratic form $Q$ and bilinear form $(\cdot,\cdot)$, such that $\{e_i,e_{n+i}\}=\{ e_i,f_i\}$ are hyperbolic pairs for $i\leq n$. 
Let $L,M$ be the totally singular subspaces $\langle e_1,\dots,e_n\rangle $ and $\langle f_1,\dots,f_n\rangle$ respectively.

We denote by $C$ the Clifford algebra of $(V,Q)$. This is an associative algebra over $K$ generated by $V$, in which $v^2=Q(v)$ for every $v\in V$. It has the structure of a graded module over $K$. 
Let $\phi':C\rightarrow C $, sending $x\mapsto x'$, be the involution of $C$ fixing every element of $V$, i.e. the anti-automorphism sending a product $\prod_{i=1}^k v_i\in C$ to $ \prod_{i=1}^k v_{n-i+1}$. We denote by $C^+$ and $C^-$ the sums of homogeneous submodules of $C$ of even and odd degrees respectively. Then $C=C^+\oplus C^-$. In particular, $C^+$ is a subalgebra of $C$ invariant under $\phi'$. 

The \textit{Clifford group} is $G^*=\{s\in C|s$ is invertible in $C$ and $sVs^{-1}=V\}$. The \textit{even Clifford group} is $(G^*)^+=G^*\cap C^+$. The \textit{spin group} $Spin_{2n}$ is $\{s\in (G^*)^+|ss'=1\}$. 

The \textit{vector representation} of the Clifford group $G^*$ is given by $\Theta: G^*\rightarrow Aut(V,Q)$, such that $\Theta(s)\cdot v=svs^{-1}$. The restriction of $\Theta$ to $Spin_{2n}$ is the natural representation of $Spin_{2n}$. The root subgroups of $Spin_{2n}$ are parametrised by pairs $(i,j)$ with $i+j\neq 2n+1$; the root subgroup parametrised  by the pairs $(i,j)$ consists of elements of the form $1+\lambda e_ie_j$, where $1+\lambda e_ie_j$ acts on a vector $v\in V$ by $v\mapsto v+\lambda(e_j,v)e_i-\lambda(e_i,v)e_j$.

Put $e_L=e_1e_2\dots e_n$ and $e_M=e_{n+1}e_{n+2}\dots e_{2n}$. We denote by $C_W$ the subalgebra of $C$ generated by the elements of a subspace $W\subset V_{2n}$. Then $Ce_M$ is a minimal left ideal in $C$, and the correspondence $x\mapsto xe_M$ generates an isomorphism $C_L\rightarrow Ce_M$ of vector spaces. So for any $s\in C,x\in C_L$ there exists a unique element $y\in C_L$ for which $sxe_M=ye_M$. Setting $\rho(s)\cdot x = s\cdot x=y$ gives us the spinor representation $\rho$ of the algebra $C$ in $C_L$. Let $X=C_L\cap C^+$. Then restricting $\rho$ to $Spin_{2n}$, we get the half-spin representation of $Spin_{2n}$ in $X$.

An element of $X$ is called a \textit{spinor}. The restriction to $B_{n-1}$ is the spin representation for $B_{n-1}$.

\subsection{List of $ts$-small quadruples}
The following result lists all $ts$-small quadruples. We will then be able to prove Theorem~\ref{existence theorem} and Theorem~\ref{complete list of cases theorem} by proceeding case-by-case. 
\begin{proposition}\label{proposition candidates small quadruples}
    Let $(G,\lambda,p,k)$ be a $ts$-small quadruple and $V = V_G(\lambda)$. Then precisely one of the following is true:
    \begin{enumerate}[label=(\roman*)]
        \item $G$ has finitely many orbits on $\mathcal{G}_k(V)$ and $(G,\lambda,p,k)$ is in Table~\ref{tab:candidates that did have finitely many orbits};
        \item $k=1$, and either $V$ is a composition factor of $\mathrm{Lie}(G)$ or $G=C_\ell$ and $\lambda = \lambda_2$, as in Table~\ref{tab:candidates infinite families};
        \item $(G,\lambda,p,k)$ is in Table~\ref{tab:candidates that did not have finitely many orbits on all subspaces}.
    \end{enumerate}
\end{proposition}

\begin{center}

\begin{longtable}{c c c c c c}
\caption{$ts$-small quadruples with finitely many orbits on $\mathcal{G}_k(V)$} \label{tab:candidates that did have finitely many orbits} \\

\hline \multicolumn{1}{c}{\textbf{$G$}} & \multicolumn{1}{c}{\textbf{$\lambda$}} & \multicolumn{1}{c}{\textbf{$\dim V$}}&  \multicolumn{1}{c}{\textbf{$p$}} & \multicolumn{1}{c}{\textbf{$k$}}& \multicolumn{1}{c}{Orthogonal?} \\ \hline 
\endhead
$A_1$ & $\lambda_1$ & $2$ &  any&$1$  & no\\*
$A_1$ & $3\lambda_1$ & $4$ &  $>3$&$1$  & no\\*
$A_1$ & $\lambda_1+p^i\lambda_1$ & $4$ &  $<\infty$&$1$  & yes\\*
$A_2$ & $\lambda_1+\lambda_2$ & $7$ &  $3$&$1$ & yes\\*
$A_5$ & $\lambda_3$ & $20$ &  $2$&$1$  & yes\\
$A_5$ & $\lambda_3$ & $21$ &  $\neq 2$&$1$  & no\\
\hline
$B_{\ell}$, $\ell\geq 2$ & $\lambda_1$ & $2\ell+1$ & $\neq 2$& any  & yes\\
$B_3$ & $\lambda_3$ & $8$ &  any&$1,2,3$  & yes\\
$B_4$ & $\lambda_4$ & $16$ &  any&$1$  &yes\\
$B_5$ & $\lambda_5$ & $32$ &  $2$&$1$  &yes\\
$B_5$ & $\lambda_5$ & $32$ &  $\neq 2$&$1$  &no\\
\hline
$C_{\ell}$, $\ell\geq 3$ & $\lambda_1$ & $2\ell$ &  any&any  & no \\
$C_3$ & $\lambda_2$ & $13$ &  $3$&$1$  &yes\\*
$C_3$ & $\lambda_3$ & $14$ &  $\neq 2$&$1$  &no\\*\hline
$D_{\ell}$, $\ell\geq 4$ & $\lambda_1$ & $2\ell$ &  any&any  & yes\\
$D_6$ & $\lambda_6$ & $32$ & $2$& $1$  &yes\\
$D_6$ & $\lambda_6$ & $32$ & $\neq 2$& $1$  &no\\
\hline
$G_2$ & $\lambda_1$ & $7$ &  $\neq 2$& $1,2$ &yes\\
$G_2$ & $\lambda_1$ & $6$ &  $2$&$1,2,3$  &no\\
$F_4$ & $\lambda_4$ & $25$ &  $3$&$1$  &yes\\
$E_7$ & $\lambda_7$ & $56$ &  $2$&$1$  &yes\\
$E_7$ & $\lambda_7$ & $56$ &  $\neq 2$&$1$  &no\\
\hline
\end{longtable}
\begin{longtable}{c c c c c c}
\caption{Infinite families of $ts$-small quadruples} \label{tab:candidates infinite families} \\

\hline \multicolumn{1}{c}{\textbf{$G$}} & \multicolumn{1}{c}{\textbf{$\lambda$}} & \multicolumn{1}{c}{\textbf{$\dim V$}}&  \multicolumn{1}{c}{\textbf{$p$}} & \multicolumn{1}{c}{\textbf{$k$}}& \multicolumn{1}{c}{Orthogonal?} \\ \hline 
\endhead
$A_\ell$, $\ell\geq 4$ & $\lambda_1+\lambda_\ell$& $\ell^2+2\ell-1$  & $\mid \ell+1$, $\neq 2$ & $1$& yes \\
$A_\ell$, $\ell\geq 2$ & $\lambda_1+\lambda_\ell$& $\ell^2+2\ell$  & $\nmid \ell+1$  &$1$& yes\\ 
$A_\ell$, $\ell\equiv 3\mod 4$   & $\lambda_1+\lambda_\ell$& $\ell^2+2\ell-1$  &$2$  &$1$& yes \\
$A_\ell$, $\ell\equiv 1\mod 4$   & $\lambda_1+\lambda_\ell$& $\ell^2+2\ell-1$  &$2$  &$1$& no \\ \hline 
$B_\ell$ & $\lambda_2$ & $2\ell^2+\ell$&$\neq 2$  &$1$ & yes\\\hline 
$C_\ell$ & $2\lambda_1$ & $2\ell^2+\ell$&$\neq 2$  & $1$ & yes \\
$C_\ell$& $ \lambda_2$& $ 2\ell^2-\ell-1$ &$\nmid \ell$, $\neq 2$  &$1$& yes \\
$C_\ell$, $\ell\geq 5$  & $ \lambda_2$&$ 2\ell^2-\ell-2$ &$\mid \ell$, $\neq 2$  &$1$ & yes \\
$C_\ell$, $\ell\not\equiv 2 \ (\textrm{mod}\ 4)$  & $ \lambda_2$& $ 2\ell^2-\ell-1-\gcd(\ell,2)$ & $2$ &$1$& yes\\
$C_\ell$, $\ell\equiv 2 \ (\textrm{mod}\ 4)$  & $ \lambda_2$& $ 2\ell^2-\ell-2$ & $2$ &$1$& no\\ \hline
$D_\ell$, $\ell\geq 4$          &   $\lambda_2$& $2\ell^2-\ell$ &$\neq 2$  &$1$& yes \\ 
$D_\ell$, $\ell$ odd, $\ell\geq 4$      & $\lambda_2$& $2\ell^2-\ell-1$  & $2$ &$1$& yes \\ 
$D_\ell$, $\ell\equiv 0\mod 4$         & $\lambda_2$ & $2\ell^2-\ell-2$& $2$ &$1$& yes \\
$D_\ell$, $\ell\equiv 2\mod 4$, $\ell\geq 4$        & $\lambda_2$ & $2\ell^2-\ell-2$& $2$ &$1$& no \\\hline  
$G_2$ & $\lambda_2$   & $14$        &$\neq 3$  &$1$& yes \\ 
$F_4$ & $\lambda_1 $   & $52$     &$\neq 2 $  &$1$& yes \\ 
$E_6$ & $\lambda_2 $  & $ 78-\delta_{p,3}$         &  &$1$& yes \\ 
$E_7$ & $ \lambda_1$ & $ 133$ & $\neq 2$ & $1$& yes\\ 
$E_7$ & $ \lambda_1$ & $ 132$ & $2$ & $1$& no\\
$E_8$& $\lambda_8 $& $248 $  &  &$1$ & yes\\ 
\hline
\end{longtable}

\begin{longtable}{c c c c c c}
\caption{Remaining $ts$-small quadruples} \label{tab:candidates that did not have finitely many orbits on all subspaces} \\

\hline \multicolumn{1}{c}{\textbf{$G$}} & \multicolumn{1}{c}{\textbf{$V$}} & \multicolumn{1}{c}{\textbf{$\dim V$}}&  \multicolumn{1}{c}{\textbf{$p$}} & \multicolumn{1}{c}{\textbf{$k$}}& \multicolumn{1}{c}{Orthogonal?} \\ \hline 
\endhead
$A_1$ & $\lambda_1+p^i\lambda_1$ & $4$ &  &$2$ & yes\\*
$A_1$ & $3\lambda_1$ & $4$ &$\neq 2,3$  & $2$& no\\*
$A_1$ & $4\lambda_1$ & $5$ &$\neq 2,3$  & $1,2$& yes\\*
$A_2$ & $\lambda_1+\lambda_2$ & $7$ &$3$  & $2,3$& yes\\*
$A_2$ & $\lambda_1+\lambda_2$ & $8$ &$\neq 3$  & $4$& yes\\
$A_5$ & $\lambda_3$ & $20$ &$\neq 2$ & $2$ & no\\
$A_5$ & $\lambda_3$ & $20$ &$2$  & $2$& yes\\ \hline
$B_2$ & $2\lambda_2$ & $10$ &  & $5$& yes\\
$B_3$ & $\lambda_3$ & $8$ & & $4$ & yes\\
$B_4$ & $\lambda_4$ & $16$ &  & $2,3,7,8$&yes\\ 
$B_6$ & $\lambda_6$ & $64$ & $2$ &$1$ & yes\\
$B_6$ & $\lambda_6$ & $64$ &$\neq 2$  & $1$& no\\\hline
$C_3$ & $\lambda_2$ & $13$ & $3$ & $2$ &yes\\*
$C_3$ & $\lambda_2$ & $14$ &$\neq 3$  & $2$ &yes \\
$C_3$ & $\lambda_2$ & $14$ &$\neq 3$  & $7$ &yes\\ \hline
$D_6$ & $\lambda_6$ & $32$ &$\neq 2$  & $2$ &no\\
$D_6$ & $\lambda_6$ & $32$ & $2$ & $2$ &yes\\ \hline
$G_2$ & $\lambda_1$ & $7$ &$\neq 2$  & $3$ &yes\\
$F_4$ & $\lambda_4$ & $25$ &$3$  & $2$ &yes\\
$F_4$ & $\lambda_4$ & $26$ &$\neq 3$  & $1,2$ &yes\\
$E_7$ & $\lambda_7$ & $56$ & $\neq 2$ &  $2$& no \\
$E_7$ & $\lambda_7$ & $56$ & $2$  & $2$ & yes\\
\hline
\end{longtable}
\end{center}

\begin{proof}
    In \cite{finite} we find a complete list of modules with finitely many orbits on $k$-spaces, of which every self-dual one leads to a $ts$-small quadruple. Now assume that $(G,\lambda,p,k)$ is a $ts$-small quadruple where $G$ does not have finitely many orbits on $\mathcal{G}_k(V)$. In \cite[Thm~3.1]{rizzoli} we have a complete list of such quadruples with $k=1$ and $V$ orthogonal, while in \cite[~Prop.~4.1]{Rizzoli2} we have a complete list of such quadruples for $k\geq 2$.
    
    The proof follows from combining the three lists of $ts$-small quadruples, with the addition of the self-dual small quadruples where $V$ is symplectic and $k=1$.
\end{proof}

\section{Quadruples with finitely many orbits on $\mathcal{G}_k(V)$}\label{section quadruples with finitely many orbits}
In this section we handle the $ts$-small quadruples where we already have finitely many orbits on all $k$-spaces, i.e. the ones listed in Table~\ref{tab:candidates that did have finitely many orbits}. It follows directly that there is a dense orbit for the action on totally singular $k$-spaces, and producing the generic stabilizer reduces to finding a stabilizer of appropriate dimension.

\begin{proposition}\label{natural modules prop}
    Let $G$ be one of $A_1$, $B_\ell$ with $p\neq 2$, $C_\ell$ or $D_\ell$. Let $\lambda=\lambda_1$. Then the $ts$-small quadruple $(G,\lambda,p,k)$ has generic $ts$-stabilizer $P_k$, unless $k=\ell-1$ and $G=D_\ell$, in which case it has generic $ts$-stabilizer $P_{l-1,l}$.
\end{proposition}

\begin{proof}
    Note that if $G=D_\ell$, by our convention the values of $k$ are $1,\dots,\ell-1,\ell',\ell''$.
    In all cases in the statement of the Proposition, the group $G$ is transitive on $\mathcal{S}_k(V)$. It is well known that the maximal parabolic subgroups of a classical group are stabilizers of totally singular subspaces, therefore these must be the generic stabilizers. Unless we are in type $D$ and $k=\ell-1$, the stabilizer of $y\in \mathcal{S}_k(V)$ is a conjugate of $P_k$. If $G=D_\ell$ and $k=\ell-1$, the stabilizer of $y\in\mathcal{S}_k(V)$ is a conjugate of the intersection of $P_{\ell-1,\ell}$.
\end{proof}

\begin{proposition}\label{prop k=1 symplectic}
The generic $ts$-stabilizers for the $ts$-small quadruples in Table~\ref{tab:symplectic finite orbit modules} are as given.
\end{proposition}

\begin{center}
\begin{longtable}{l l l l l l}
\caption{$ts$-small quadruples with $k=1$ and $V$ symplectic} \label{tab:symplectic finite orbit modules} \\

\hline \multicolumn{1}{c}{\textbf{$G$}} & \multicolumn{1}{c}{\textbf{$\lambda$}} &  \multicolumn{1}{c}{\textbf{$p$}} & \multicolumn{1}{c}{\textbf{$k$}}& \multicolumn{1}{c}{\textbf{$C_{\mathcal{S}_k(V)}$}}& \multicolumn{1}{c}{Reference} \\ \hline 
\endhead
$A_1$ & $3\lambda_1$ & $>3$ & $1$ &  $Sym(3)$ & \cite[Prop.~5.1.6]{generic}\\*
$A_5$ & $\lambda_3$ & $\neq 2$ & $1$ &  $A_2^2.\mathbb{Z}_2$ & \cite[Prop.~5.2.7]{generic}\\*
$C_3$ & $\lambda_3$ & $\neq 2$ & $1$ &  $A_2.\mathbb{Z}_2$ & \cite[Prop.~5.2.7]{generic}\\*
$B_5$ & $\lambda_5$ & $\neq 2$ & $1$ &  $A_4.\mathbb{Z}_2$ & \cite[Prop.~5.2.7]{generic}\\*
$D_6$ & $\lambda_6$ & $\neq 2$ & $1$ &  $A_5.\mathbb{Z}_2$ & \cite[Prop.~5.2.7]{generic}\\*
$G_2$ & $\lambda_1$  &  $2$&$1$ &  $U_5A_1T_1$ &\cite[Prop.~5.2.14]{generic}\\
$E_7$ & $\lambda_7$ &  $\neq 2$ & $1$ & $ E_6.\mathbb{Z}_2$ & \cite[Prop.~5.2.7]{generic}\\ 
\hline
\end{longtable}
\end{center}

\begin{proof}
    In each of these cases $k=1$ and the module $V$ is symplectic. Therefore $\mathcal{S}_k(V)=\mathcal{G}_k(V)$ and the result follows directly from \cite{generic}. In the last column of Table~\ref{tab:symplectic finite orbit modules} we give a reference for each individual case.
\end{proof}

\begin{proposition}\label{prop A_1 w1+p^iw1 k=1}
    Let $G=A_1$ and $\lambda = \lambda_1+p^i\lambda_1$ with $p>0$ and $i>0$. Then the quadruple $(G,\lambda,p,1)$ has generic $ts$-stabilizer $T_1$.
\end{proposition}

\begin{proof}
    Let $q=p^i$ and let $\sigma=\sigma_q$ be the standard Frobenius morphism acting on $K$ as $t\mapsto t^\sigma=t^q$ and on $G$ as $x_{\pm \alpha_1}(t)\mapsto x_{\pm \alpha_1}(t^q)$. Let $G=SL_2(K)$.
We can view $V$ as the space $M_{2\times 2}(K)$ of $2\times 2$ matrices on which $G$ acts by $g.v=gv(g^\sigma)^T$ for $v\in M_{2\times 2}(K)$ and $g\in G$. Since $G$ preserves the determinant of $v$ for all $v\in M_{2\times 2}(K)$, we can take the quadratic form $Q:V\rightarrow K$ as $Q(v)=\det v$.
The singular $1$-spaces of $V$ are therefore the $1$-spaces spanned by matrices with determinant $0$. Let $y$ be the singular $1$-space spanned by  $\left(\begin{matrix} 0 & 1\\ 0 & 0 \\ \end{matrix}\right)$. Then $\left(\begin{matrix} a & b\\ c & d \\ \end{matrix}\right)\in G_y$ if and only if $b=c=0$ and $d=a^{-1}$. Since $\dim G - \dim \mathcal{S}_1(V) = 1 =\dim G_y$, the element $y$ is in a dense orbit for the $G$-action on $\mathcal{S}_1(V)$. Therefore $C_{\mathcal{S}_1(V)} = T_1$.
\end{proof}

\begin{proposition}\label{prop A2 w1+w2 k=1 p=3}
    Let $G=A_2$ and $\lambda = \lambda_1+\lambda_2$ with $p=3$. Then the quadruple $(G,\lambda,p,1)$ has generic $ts$-stabilizer $U_2T_1$.
\end{proposition}

\begin{proof}
    By \cite[Lemma~5.5]{Rizzoli2} there is $y\in \mathcal{S}_1(V)$ with $G_y=U_2T_1$. Since $\dim G - \dim \mathcal{S}_1(V)=8-5 =\dim G_y$, the element $y$ is in a dense orbit for the $G$-action on $\mathcal{S}_1(V)$. Therefore $C_{\mathcal{S}_1(V)} = U_2T_1$.
\end{proof}

\begin{proposition}\label{prop A_5 w3 k=1 p=2}
    Let $G=A_5$ and $\lambda = \lambda_3$ with $p=2$. Then the quadruple $(G,\lambda,p,1)$ has generic $ts$-stabilizer $U_8A_2T_1$.
\end{proposition}

\begin{proof}
    By \cite[2.3.1(II)]{revoyTrivectors} there is $y\in \mathcal{S}_1(V)$ with $G_y=U_8A_2T_1$. Since $\dim G - \dim \mathcal{S}_1(V)=35-18 =\dim G_y$, the element $y$ is in a dense orbit for the $G$-action on $\mathcal{S}_1(V)$. Therefore $C_{\mathcal{S}_1(V)} = U_8A_2T_1$.
\end{proof}

\begin{proposition}\label{prop B_3 w3 k=1}
    Let $G=B_3$, $\lambda = \lambda_3$. Then the quadruple $(G,\lambda,p,1)$ has generic $ts$-stabilizer $U_6A_2T_1$.
\end{proposition}

\begin{proof}
    By \cite[Thm~B]{factorizations}, the group $G$ is transitive on $\mathcal{S}_1(V)$. The generic stabilizer is the $P_3$-parabolic, i.e. $C_{\mathcal{S}_1(V)} = U_6A_2T_1$.
\end{proof}

\begin{proposition}\label{prop B_3 w3 k=2,3}
    Let $G=B_3$, $\lambda = \lambda_3$ with $k=2$ or $k=3$. Then the quadruple $(G,\lambda,p,k)$ has generic $ts$-stabilizer $U_5A_1A_1T_1$ if $k=2$ and $U_3A_2T_1$ if $k=3$.
\end{proposition}

\begin{proof}
By \cite[Thm~B]{factorizations}, the group $G$ is transitive on $\mathcal{S}_4'(V)$. Let $W\in \mathcal{S}_4'(V)$. The group $G$ is the group of fixed points of a triality automorphism of $D_4 = Cl(V)$. Therefore $G_W$ is isomorphic to the generic stabilizer for the action on $\mathcal{S}_1(V)$, i.e. $U_6A_2T_1$. Then it is easy to see (\cite[Lemma~3.5]{finite}) that $G_W$ acts on $\mathcal{G}_2(W)$ with two orbits, one with stabilizer $U_5A_1A_1T_1$ and one with stabilizer $U_7A_1A_1T_1$. Since every totally singular $2$-space is contained in an element of $\mathcal{S}_4'(V)$, we conclude that there are at most two $G$-orbits on $\mathcal{S}_2(V)$. Since $\dim G-\dim \mathcal{S}_2(V)=12$, there must be a $12$-dimensional stabilizer for the $G$-action on $\mathcal{S}_2(V)$. The only possibility is therefore $C_{\mathcal{S}_2(V)} = U_5A_1A_1T_1$.

Similarly, $G_W$ acts on $\mathcal{G}_3(W)$ with two orbits, one with stabilizer $U_3A_2T_1$, and one with stabilizer $U_8A_1T_2$. Since $\dim G-\dim U_3A_2T_1 = \dim \mathcal{S}_3(V)$, we conclude that $C_{\mathcal{S}_3(V)} = U_3A_2T_1$.
\end{proof}

\begin{proposition}\label{prop B4 B5 w4 w5 k=1}
    Let $G=B_\ell$, $\lambda = \lambda_\ell$ with $\ell=4$ or $\ell=5$, with $p=2$ if $\ell=5$. Then the quadruple $(G,\lambda,p,1)$ has generic $ts$-stabilizer $U_7G_2T_1$ if $\ell=4$ and $U_{14}B_2T_1$ if $\ell=5$.
\end{proposition}

\begin{proof}
    By \cite[Prop.~5, Prop.~6]{igusa} (when $p\neq 2$) and \cite[Lemma~2.11]{finite} (when $p=2$), there is only one orbit on $\mathcal{G}_1(V)$ with a $22$-dimensional stabilizer when $\ell=4$, with structure $U_7G_2T_1$, and only one orbit on $\mathcal{G}_1(V)$ with $25$-dimensional stabilizer when $\ell=5$, with structure $U_{14}B_2T_1$. Since $\dim G-\dim\mathcal{S}_1(V) = 22$ when $\ell=4$ and $25$ when $\ell=4$, we conclude that $C_{\mathcal{S}_1(V)} = U_7G_2T_1$ when $\ell=4$ and $U_{14}B_2T_1$ when $\ell=5$.
\end{proof}

\begin{proposition}\label{prop C3 w2 k=1}
    Let $G=C_3$, $\lambda = \lambda_2$ with $p=3$. Then the quadruple $(G,\lambda,p,1)$ has generic $ts$-stabilizer $U_6A_1T_1$.
\end{proposition}

\begin{proof}
    By \cite[Lemma~5.15]{Rizzoli2} there is $y\in \mathcal{S}_1(V)$ with $G_y=U_6A_1T_1$. Since $\dim G - \dim \mathcal{S}_1(V)=21-11 =\dim G_y$, the element $y$ is in a dense orbit for the $G$-action on $\mathcal{S}_1(V)$. Therefore $C_{\mathcal{S}_1(V)} = U_6A_1T_1$.
\end{proof}

\begin{proposition}\label{D_6 w6 k=1 p=2}
    Let $G=D_6$, $\lambda = \lambda_6$ with $p=2$. Then the quadruple $(G,\lambda,p,1)$ has generic $ts$-stabilizer $U_{14}B_3T_1$.
\end{proposition}

\begin{proof}
    By the proof of \cite[Lemma~2.11]{finite} there is only one orbit on $\mathcal{G}_1(V)$ with a $36$-dimensional stabilizer, with structure $U_{14}B_3T_1$. Since $\dim G - \dim \mathcal{S}_1(V)=66-30 =\dim G_y$, the element $y$ is in a dense orbit for the $G$-action on $\mathcal{S}_1(V)$. Therefore $C_{\mathcal{S}_1(V)} = U_{14}B_3T_1$.
\end{proof}

\begin{proposition}\label{prop G_2 w1 p not 2 k=1}
    Let $G=G_2$, $\lambda = \lambda_1$ with $p\neq 2$.
    Then the quadruple $(G,\lambda,p,1)$ has generic $ts$-stabilizer $U_{5}A_1T_1$.
\end{proposition}

\begin{proof}
    By \cite[Thm.~A]{factorizations} $G$ is transitive on singular $1$-spaces of $V$. Therefore a representative can be taken to be the $1$-space spanned by the highest weight vector, with stabilizer $P_1=U_{5}A_1T_1$.
\end{proof}

\begin{proposition}\label{prop G_2 w1 k=2}
    Let $G=G_2$, $\lambda = \lambda_1$. Then the quadruple $(G,\lambda,p,2)$ has generic $ts$-stabilizer $U_{3}A_1T_1$.
\end{proposition}

\begin{proof}
    Let $\mu_1,\mu_2,\mu_3$ be the positive weights $2\alpha_1+\alpha_2,\alpha_1+\alpha_2,\alpha_1$, respectively. Given weight vectors $v_{\mu_1},v_{-\mu_2} $, let $y:=\langle v_{\mu_1},v_{-\mu_2}\rangle$, an element of $\mathcal{S}_2(V)$. Since $n_2n_1n_2n_1n_2.\langle v_{\mu_1} \rangle = \langle v_{-\mu_2}\rangle$, we can easily determine that $G_{\langle v_{\mu_1}\rangle} \cap G_{\langle v_{-\mu_2}\rangle} = P_1\cap P_1^{n_2n_1n_2n_1n_2} = U_3T_2$, by checking which root subgroups are in common between $P_1$ and  $P_1^{n_2n_1n_2n_1n_2}$. Therefore $G_y\leq U_3A_1T_1$. Since $\dim \mathcal{S}_2(V) = 7$ (in both cases $p\neq 2$ and $p=2$), the minimum dimension of the stabilizer of any totally singular $2$-space is $7$. Therefore $G_y = U_3A_1T_1 = C_{\mathcal{S}_2(V)}$.
\end{proof}

\begin{proposition}\label{G2 k=3 p=2}
    Let $G=G_2$, $\lambda = \lambda_1$ with $p=2$. Then the quadruple $(G,\lambda,p,3)$ has generic $ts$-stabilizer $A_2$.
\end{proposition}

\begin{proof}
    By the proof of \cite[Lemma~3.4]{finite} there is an $A_2$-subgroup which is the stabilizer of $y\in \mathcal{S}_3(V)$. Since $\dim G - \dim \mathcal{S}_3(V)=14-6 =\dim G_y$, the element $y$ is in a dense orbit for the $G$-action on $\mathcal{S}_3(V)$. Therefore $C_{\mathcal{S}_3(V)} = A_2$.
\end{proof}

\begin{proposition}\label{prop f4 w4 p=3 k=1}
    Let $G=F_4$, $\lambda = \lambda_4$ with $p=3$. Then the quadruple $(G,\lambda,p,1)$ has generic $ts$-stabilizer $U_{14}G_2T_1$.
\end{proposition}

\begin{proof}
    By \cite[Lemma~4.13]{brundan} there is $y\in \mathcal{S}_1(V)$ such that $G_y=U_{14}G_2T_1$. Since $\dim G - \dim \mathcal{S}_1(V)=29 =\dim G_y$, the element $y$ is in a dense orbit for the $G$-action on $\mathcal{S}_1(V)$. Therefore $C_{\mathcal{S}_1(V)} = U_{14}G_2T_1$.
\end{proof}

\begin{proposition}\label{E7 k=1 p=2 w7}
    Let $G=E_7$, $\lambda = \lambda_7$ with $p=2$. Then the quadruple $(G,\lambda,p,1)$ has generic $ts$-stabilizer $U_{26}F_4T_1$.
\end{proposition}

\begin{proof}
    By \cite[Lemma~4.3]{OrdersMaxSubgroupsExceptional} there is only one orbit on $\mathcal{G}_1(V)$ with a $79$-dimensional stabilizer, with structure $U_{26}F_4T_1$. Since $\dim G - \dim \mathcal{S}_1(V)=133-54 =\dim U_{26}F_4T_1$, we must have $C_{\mathcal{S}_1(V)} = U_{26}F_4T_1$.
\end{proof}

\section{Infinite families of quadruples}\label{infinite families section}
In this section we handle the cases that appear in Table~\ref{tab:candidates infinite families}. The following two propositions provide a reduction to a finite list of cases. 
\begin{proposition}\label{adjoint module general proposition}
    Let $V$ be a composition factor of $\mathrm{Lie}(G)$, with $p\neq 2$ if $G=B_\ell$ or $G=C_\ell$, and assume that the $0$-weight space $V_0$ is at least $3$-dimensional. Then $C_{\mathcal{S}_1(V)}=C_{\mathcal{G}_1(V)}$.
\end{proposition}

\begin{proof}
    The composition factors of $\mathrm{Lie}(G)$ are listed in \cite[Prop.~1.10]{LST3}. By assumption on $V$, we have $V=\mathrm{Lie}(G)/Z$, where $Z$ is the centre of $\mathrm{Lie}(G)$. If $p=2$, and $G=A_\ell$ with $\ell\equiv 1\mod 4$, or $G=D_\ell$ with $\ell\equiv 2\mod 4$, or $G=E_7$, then the module $V$ is symplectic and $\mathcal{S}_1(V) = \mathcal{G}_1(V)$. In all other cases the module $V$ is orthogonal.
    The proof closely mimics \cite[Lemma~4.2.1(ii)]{generic}, but in the interest of clarity it is fully reproduced with the appropriate changes. Note that our setup corresponds to the specific case $\Theta = 1$ in \cite[Lemma~4.2.1(ii)]{generic}, which in particular means $G=H$ in the proof of \cite[Lemma~4.2.1(ii)]{generic}. Let \[W^\ddag = \{ w\in W:\exists \xi\in K^*,\forall v\in V_0, w.v=\xi v\},\] and let $N^\ddag$ be the pre-image of $W^\ddag$ under the quotient map $N\rightarrow W$. Let $Y$ be $\mathcal{S}_1(V_0)$ and let $\hat{Y}_1$ be the subset of $Y$ consisting of $1$-spaces of $V_0$ spanned by images of regular semisimple elements in $\mathrm{Lie}(T)$. Since $Z$ does not contain regular semisimple elements, the set $\hat{Y}_1$ is non-empty, and thus dense in $Y$. Since $\dim V_0\geq 3$, the span of $Y$ is the full $V_0$. Therefore, any element in $N$ which fixes all $y\in Y$ must be in $N^\ddag$. Thus, given $w\in W\setminus W^\ddag$, take $n\in N$ with $nT=w$; the set of elements of $Y$ fixed by $n$ is a proper closed subvariety of $Y$. Let $\hat{Y}_2$ be the complement of the union of these subvarieties as $w$ runs over $W\setminus W^\ddag$. Set $\hat{Y} = \hat{Y}_1\cap\hat{Y}_2$, a dense open subset of $Y$.

    Let $y\in\hat{Y}$. By \cite[Lemma~2.1]{finite}, two elements of $\mathcal{G}_1(V_0)$ are in the same $G$-orbit if and only if they are in the same $W$-orbit. Therefore $G.y\cap Y$ is finite and $\dim(\overline{G.y\cap Y})=0$. Also, since $y$ is spanned by a regular semisimple element, we have $G_y^0 = T$, and therefore $\dim (\overline{G.y})=\dim G - \dim T $. Finally note that since $V = \mathrm{Lie(G)}/Z$, we have $(C_G(V_0))^0 = T$, and therefore $\dim V-\dim V_0=\dim G-\dim T$ (see the proof of \cite[Lemma~2.4]{finite}). Thus
    \begin{flalign*}
        \dim \mathcal{S}_1(V)-\dim (\overline{G.y}) &= \dim V- 2 -\dim (\overline{G.y}) =& \\
        &= \dim G-\dim T+\dim V_0 -2 -\dim (\overline{G.y}) =& \\
        &= \dim Y - \dim(\overline{G.y\cap Y}) = \dim Y + \codim \mathrm{Tran}_G(y,Y)-\dim (\overline{G.y}), & \\
    \end{flalign*}
where the last step uses Lemma~\ref{transporter dimension lemma}.
By definition, this proves that $y$ is $Y$-exact. The conditions of Lemma~\ref{loc to a subvariety lemma} hold and $C_{\mathcal{S}_1(V)} = T.W^\ddag = C_{\mathcal{G}_1(V)}$. 
\end{proof}

\begin{proposition}\label{C_l lambda_2 general proposition}
    Let $G=C_\ell$, $\lambda=\lambda_2$. Furthermore assume that the $0$-weight space $V_0$ is at least $3$-dimensional. Then $C_{\mathcal{S}_1}(V)=C_{\mathcal{G}_1}(V) = A_1^\ell$.
\end{proposition}

\begin{proof}
    If $p=2$ and $\ell \equiv 2\mod 4$, then the module $V$ is symplectic and $\mathcal{S}_1(V)=\mathcal{G}_1(V)$. Therefore assume that when $p=2$ we have $\ell \not\equiv 2\mod 4$, which implies that the module $V$ is orthogonal. We use the setup of \cite[Prop.~5.2.5]{generic} combined with the approach of \cite[Lemma~4.2.1(ii)]{generic}, which we saw in action in Proposition~\ref{adjoint module general proposition}. Inside $\bigwedge^2 V_{nat}$ we have submodules $X_1 = \{ \sum_{i<j}\rho_{ij}e_i\wedge e_j+\sum_{i<j}\sigma_{ij}f_i\wedge f_j+\sum_{i,j}\tau_{ij}e_i\wedge f_j:\sum_i\tau_{ii}=0\}$ and $X_2=\langle \sum_i e_i\wedge f_i \rangle$. If $p\nmid \ell$ then $V=X_1$, otherwise $X_2<X_1$ and $V=X_1/X_2$. In all cases $V=X_1/(X_1\cap X_2)$. Let $x_i = e_i\wedge f_i$. The $0$-weight space is $V_0 = \{ \sum a_ix_i+(X_1\cap X_2):\sum a_i = 0\}$. Then $G$ fixes a non-degenerate quadratic form on $V$ such that $Q( a_ix_i+(X_1\cap X_2)) = \sum a_i^2 +\sum_{i<j}a_ia_j$ (see \cite[8.1.2]{garibaldi_nakano_2016} when $p=2$). Let $Y = \mathcal{S}_1(V_0)$ and let \[\hat{Y}_1 = \{ \langle v \rangle\in Y , v=\sum a_ix_i+(X_1\cap X_2): a_i\neq a_j\text{ if }i\neq j\},\] a dense subset of $Y$. Then $y\in \hat{Y}_1$ is fixed by $C=A_1^\ell = \bigcap_{i} G_{\langle e_i,f_i \rangle}$. Any minimal connected overgroup of $C$ in $G$ is isomorphic to $C_2 A_1^{\ell-2}$, which does not fix $y\in \hat{Y}_1$ because of the condition on the coefficients. Therefore for any $y\in \hat{Y}_1$ we have $(G_y)^0 = C$. Let $N = N_G(C)/C$, a group isomorphic to $Sym(\ell)$. Let \[N^\ddag = \{ n\in N:\exists \xi\in K^*,\forall v\in V_0, n.v=\xi v\}.\] Now assume that $n\in N$ fixes all $y\in Y$. Since $\dim V_0\geq 3$, the span of $Y$ is the full $V_0$. Therefore, any element in $N$ which fixes all $y\in Y$ must be in $N^\ddag$. Thus, the set of elements fixed by $n\in N\setminus N^\ddag$ is a proper closed subvariety of $Y$. Let $\hat{Y}_2$ be the complement of the union of these subvarieties as $n$ runs over $N\setminus N^\ddag$. Set $\hat{Y} = \hat{Y}_1\cap\hat{Y}_2$, a dense open subset of $Y$. Finally note that the proof of \cite[Prop.~5.2.5]{generic} shows that $N^\ddag = 1$. Let $y\in \hat{Y}$. Two elements of $\mathcal{G}_1(V_0)$ are in the same $G$-orbit if and only if they are in the same $N$-orbit, and therefore $\dim(\overline{G.y\cap Y})=0$. Also, $\dim V-\dim V_0=\dim G-\dim C$, by the proof of \cite[Lemma~2.4]{finite}. Therefore, as in Proposition~\ref{adjoint module general proposition}, we get
    \begin{flalign*}
        \dim \mathcal{S}_1(V)-\dim (\overline{G.y})&= \dim V- 2 -\dim (\overline{G.y}) =&\\
        &= \dim G-\dim C+\dim V_0 -2 -\dim (\overline{G.y}) = &\\
        &=\dim Y - \dim(\overline{G.y\cap Y}),&
    \end{flalign*}
proving that $y$ is $Y$-exact. The conditions of Lemma~\ref{loc to a subvariety lemma} hold and $C_{\mathcal{S}_1(V)} = C.N^\ddag = C = A_1^\ell$.

\end{proof}

\begin{proposition}\label{special adjoint cases}
    Let $G=B_2$ and $\lambda = 2\lambda_2$ with $p\neq 2$, or $G=A_2$ and $\lambda = \lambda_1+\lambda_2$ with $p\neq 3$, or $G=A_3$ and $\lambda_1+\lambda_3$ with $p=2$, or $G=G_2$ and $\lambda = \lambda_2$ with $p\neq 3$. Then $C_{\mathcal{S}_1(V)}$ is respectively $T_2.\mathbb{Z}_4$, $T_2.\mathbb{Z}_3$, $T_3.Alt(4)$, $T_2.\mathbb{Z}_6$.
\end{proposition}

\begin{proof}
    In all of these cases the module $V$ is a composition factor of $\mathrm{Lie}(G)$, the $0$-weight space is $2$-dimensional, and we will show that the generic stabilizer is the stabilizer of one of the two singular $1$-spaces of $V_0$.
    
    Let $G=B_2$ and $\lambda = 2\lambda_2$ with $p\neq 2$. Then $V=\mathrm{Lie}(G)$, and we can take $V_0 = \{\mathrm{diag}(a,b,0,-b,-a) :a,b\in K\}$. Since $p\neq 2$, the group $G$ fixes the non-degenerate quadratic form $Q$ induced by the Killing form. Let $v=\mathrm{diag}(a,b,0,-b,-a)$ be a singular element of $V_0$. Then since $Q(v)=0$, we know that $a^2+b^2=0$. Since $v$ is regular semisimple, we must have $(G_{\langle v\rangle})^0 = T_2$. We then find that $W_{\langle v\rangle}=\langle w\rangle $, where $w$ is an element of order $4$ sending $\mathrm{diag}(a,b,0,-b,-a)\mapsto \mathrm{diag}(b,-a,0,a,-b) $. Since $\dim G-\dim \mathcal{S}_1(V) = 2 =\dim G_{\langle v\rangle}$, we conclude that $C_{\mathcal{S}_1(V)} = T_2.\mathbb{Z}_4$.
    
    Let $G=A_2$ and $\lambda =\lambda_1+\lambda_2$ with $p\neq 3$. Then $V=\mathrm{Lie}(G)$, and we can take $V_0 = \{\mathrm{diag}(a,b,-a-b):a,b\in K\}$. If $p\neq 2$, a non-degenerate symmetric bilinear form preserved by $G$ is given by the Killing form. If $p=2$ we find an explicit description of a non-degenerate quadratic form preserved by $G$ in \cite[§5.1]{Babi}.
    Let $v=\mathrm{diag}(a,b,-a-b)$ be a singular element of $V_0$, which implies that $a^2+b^2+ab = 0$. Since $v$ is regular semisimple, we must have $(G_{\langle v\rangle})^0 = T_2$. As $a^2+b^2+ab = 0$, we then find that $(W)_{\langle v\rangle}=\langle w\rangle $, where $w$ is a $3$-cycle in $W$. Since $\dim G-\dim \mathcal{S}_1(V) = 2 =\dim G_{\langle v\rangle}$, we conclude that $C_{\mathcal{S}_1(V)} = T_2.\mathbb{Z}_3$.

    Let $G=A_3$ and $\lambda =\lambda_1+\lambda_3$ with $p=2$. Then $V=\mathfrak{sl}_4/\langle I\rangle$, where $I$ is the identity $4\times 4$ matrix, and we can take $V_0 = \{\mathrm{diag}(a,b,a+b,0)+\langle I\rangle:a,b\in K\}$. Let $v=\mathrm{diag}(a,b,a+b,0)+\langle I \rangle$ be a singular element of $V_0$, which implies that $a^2+b^2+ab = 0$ (again see \cite[§5.1]{Babi}). Since $\mathrm{diag}(a,b,a+b,0)$ is regular semisimple, we must have $(G_{\langle v\rangle})^0 = T_3$. As $a^2+b^2+ab = 0$, we then find that $W_{\langle v\rangle}=\langle \tau_1,\tau_2,w\rangle $, where $w$ is a $3$-cycle in $W$ sending $\mathrm{diag}(a,b,a+b,0)\mapsto \mathrm{diag}(a+b,a,b,0)$, $\tau_1$ is an element of order $2$ sending $\mathrm{diag}(a,b,a+b,0)\mapsto \mathrm{diag}(0,a+b,b,a)$ and $\tau_2$ is an element of order $2$ sending $\mathrm{diag}(a,b,a+b,0)\mapsto \mathrm{diag}(a+b,0,a,b)$. Since $\dim G-\dim \mathcal{S}_1(V) = 3 =\dim G_{\langle v\rangle}$, we conclude that $C_{\mathcal{S}_1(V)} = T_3.Alt(4)$.

    Let $G=G_2$ and $\lambda = \lambda_2$ with $p\neq 3$. Then $V=\mathrm{Lie}(G)$. We view $G$ as a subgroup of $B_3$, so that we can take $V_0 =\{ \mathrm{diag}(a,b,-a-b,0,-a,-b,a+b):a,b\in K\}$. Let $v=\mathrm{diag}(a,b,-a-b,0,-a,-b,a+b)$ be a singular element of $V_0$, which implies that $a^2+b^2+ab = 0$ (see \cite[§5.1]{Babi}). Since $v$ is regular semisimple, we must have $(G_{\langle v\rangle})^0 = T_2$. As $a^2+b^2+ab = 0$, we then find that $(W)_{\langle v\rangle}=\langle \tau, w\rangle $, where $\tau$ is an element of order $2$ sending $v\mapsto -v$ and $w$ is an element of order $3$ sending $\mathrm{diag}(a,b,-a-b,0,a+b,-b,-a)\mapsto \mathrm{diag}(-a-b,a,b,0,-b,-a,a+b)$. Since $\dim G-\dim \mathcal{S}_1(V) = 2 =\dim G_{\langle v\rangle}$, we conclude that $C_{\mathcal{S}_1(V)} = T_2.\mathbb{Z}_6$.

\end{proof}

\begin{proposition}\label{special cases type C w2}
    Let $G=C_3$ and $\lambda = \lambda_2$ with $p\neq 3$, or $G=C_4$ and $\lambda = \lambda_2$ with $p=2$, or $G=D_4$ and $\lambda=\lambda_2$ with $p=2$. Then $C_{\mathcal{S}_1(V)}$ is respectively $A_1^3.\mathbb{Z}_3$, $A_1^4.Alt(4)$, $T_4.(2^3.Alt(4))$.
\end{proposition}

\begin{proof}
    In all of these cases the module $V$ is a composition factor of $\bigwedge^2 V_{nat}$, the $0$-weight space is $2$-dimensional, and the generic stabilizer is the stabilizer of one of the two singular $1$-spaces of $V_0$. Let $X_1,X_2$ be as in the proof of Proposition~\ref{C_l lambda_2 general proposition}, so that $V = X_1/X_2$.

    Let $G=C_3$ and $\lambda = \lambda_2$ with $p\neq 3$. Then $X_2 =0$, and we can take \[V_0 = \{a e_1\wedge f_1+b e_2\wedge f_2 -(a+b)e_3\wedge f_3 :a,b,\in K\}.\] Let $v=a e_1\wedge f_1+b e_2\wedge f_2 -(a+b)e_3\wedge f_3$ be a singular element of $V_0$. Then since $Q(v)=0$, we know that $a^2+b^2+ab=0$. Since $a\neq b$, we must have $(G_{\langle v\rangle})^0 = A_1^3$ as in the proof of Proposition~\ref{C_l lambda_2 general proposition}. We then find that $(W)_{\langle v\rangle}=\langle w\rangle $, where $w$ is a $3$-cycle in $W$. Since $\dim G-\dim \mathcal{S}_1(V) = 9 =\dim G_{\langle v\rangle}$, we conclude that $C_{\mathcal{S}_1(V)} = A_1^3.\mathbb{Z}_3$.

    The remaining two cases are entirely similar, with the result for $D_4$ being derived from $C_4$, since $V_{D_4}(\lambda_2) = V_{C_4}(\lambda_2)\downarrow D_4$ when $p=2$.
\end{proof}

\section{Remaining quadruples}\label{remainin quadruples section}
This section is where we deal with the remaining cases, i.e. all the possibilities listed in Table~\ref{tab:candidates that did not have finitely many orbits on all subspaces}.

\begin{proposition}\label{many cases from k=1 and k=2 proposition}
The generic $ts$-stabilizers for the $ts$-small quadruples in Table~\ref{tab:remaining small quadruples proposition} are as given.
\end{proposition}

\begin{center}
\begin{longtable}{l l l l l l}
\caption{Remaining $ts$-small quadruples with known generic $ts$-stabilizer} \label{tab:remaining small quadruples proposition} \\

\hline \multicolumn{1}{c}{\textbf{$G$}} & \multicolumn{1}{c}{\textbf{$\lambda$}} &  \multicolumn{1}{c}{\textbf{$p$}} & \multicolumn{1}{c}{\textbf{$k$}}& \multicolumn{1}{c}{\textbf{$C_{\mathcal{S}_k(V)}$}}& \multicolumn{1}{c}{Reference} \\ \hline 
\endhead
$A_1$ & $\lambda_1+p^i\lambda_1$ & $<\infty$ & $2$ &  $U_1T_1$ & \cite[Prop.~5.1]{Rizzoli2}\\*
$A_1$ & $3\lambda_1$ & $>3$ & $2$ &  $Alt(4)$ & \cite[Prop.~5.2]{Rizzoli2}\\*
$A_1$ & $4\lambda_1$ & $>3$ & $1$ &  $Alt(4)$ & \cite[Prop.~4.1]{rizzoli}\\*
$A_1$ & $4\lambda_1$ & $>3$ & $2$ &  $Sym(3)$ & \cite[Prop.~5.2]{Rizzoli2}\\*
$A_2$ & $\lambda_1+\lambda_2$ & $3$ & $2$ &  $U_1$ & \cite[Prop.~5.3]{Rizzoli2}\\*
$B_4$ & $\lambda_4$ &  $2$ & $2$ & $ U_5A_1A_1$ & \cite[Prop.~5.55]{Rizzoli2}\\ 
$B_4$ & $\lambda_4$ &  $\neq 2$ & $2$ &  $ A_1(A_2.\mathbb{Z}_2)$ & \cite[Prop.~5.55]{Rizzoli2}\\ 
$B_6$ & $\lambda_6$ &  $\neq 2$ & $1$ &  $A_2^2.\mathbb{Z}_2$ & \cite[Prop.~5.2.9]{generic}\\
$B_6$ & $\lambda_6$ &  $2$ & $1$ &  $(U_5A_1)^2.\mathbb{Z}_2$ & Prop.~\ref{extra prop B_4 p=2 k=1 lambda_4}\\ 
$C_3$ & $\lambda_2$ &  $3$ & $2$ & $U_1(T_1.\mathbb{Z}_2)$ & \cite[Prop.~5.14]{Rizzoli2}\\*
$F_4$ & $\lambda_4$ &  $3$ & $2$ &  $U_1(A_2.\mathbb{Z}_2)$ & \cite[Prop.~5.29]{Rizzoli2}\\
$F_4$ & $\lambda_4$ &  $\neq 3$ & $1$ &  $D_4.\mathbb{Z}_3$ &  \cite[Prop.~6.6]{rizzoli}\\
\hline
\end{longtable}
\end{center}
\begin{proof}
    In each of these cases, except for $(B_6,\lambda_6,p ,1)$ with $p\neq 2$, it was proven in \cite{rizzoli} or \cite{Rizzoli2} that $G$ has a dense orbit on $\mathcal{S}_k(V)$. In the case $(B_6,\lambda_6,p,1)$ with $p\neq 2$, the module $V$ is symplectic and we have $\mathcal{S}_k(V) = \mathcal{G}_k(V)$, with the result following from \cite{generic}. In Table~\ref{tab:remaining small quadruples proposition} we list the reference for each quadruple. 
\end{proof}

\begin{proposition}\label{extra prop B_4 p=2 k=1 lambda_4}
    Let $G= B_4$, $\lambda = \lambda_4$ with $p= 2$. Then the quadruple $(G,\lambda,p,1)$ has generic $ts$-stabilizer $(U_5A_1)^2.\mathbb{Z}_2$.  
\end{proposition}

\begin{proof}
    This result is already listed in \cite[Thm.~1]{rizzoli}, albeit without a full explanation. The quadruple $(D_5,\lambda_5,2,1)$ has generic stabilizer $(G_2G_2).\mathbb{Z}_2$ by \cite{generic}. By \cite[Lemma~5.17]{rizzoli} we deduce that $C_{\mathcal{S}_1(V)}$ is isomorphic to the generic stabilizer of the action of $(G_2G_2).\mathbb{Z}_2$ on $V_{G_2}(\lambda_1)\oplus V_{G_2}(\lambda_1)$. Since $G_2$ is transitive on non-zero vectors of $V_{G_2}(\lambda_1)$, the generic stabilizer for this action is easily seen to be $(U_5A_1)^2.\mathbb{Z}_2$.
\end{proof}

\begin{proposition}\label{G2 k=3 p not 2}
    Let $G= G_2$, $\lambda = \lambda_1$ with $p\neq 2$. Then the quadruple $(G,\lambda,p,3)$ has generic $ts$-stabilizer $A_2$.
\end{proposition}

\begin{proof}
    As in Proposition~\ref{G2 k=3 p=2} we know that there is $y\in \mathcal{S}_3(V)$ with $G_y = A_2$. Since $\dim G - \dim A_2 = 6 = \dim \mathcal{S}_3(V)$, we conclude that $C_{\mathcal{S}_3(V)} = A_2$.
\end{proof}

We now relax slightly the condition that the group acting should be simple: we allow a product of isomorphic simple groups, possibly extended by a graph automorphism. If the connected group is of the form $G_1G_2G_3$, we write $\lambda = \mu_1\otimes \mu_2\otimes \mu_3$ to mean $V_G(\lambda) = V_{G_1}(\mu_1)\otimes V_{G_2}(\mu_2)\otimes V_{G_3}(\mu_3)$, where each $\mu_i$ is a dominant weight for $G_i$.

\begin{lemma}\label{a1^3 preliminary lemma p=2}
    Let $G=A_1^3$, $\lambda=\lambda_1\otimes\lambda_1\otimes\lambda_1$ with $p=2$. Then there is an open dense subset $\hat{Y}$ of $\mathcal{S}_2(V)$ such that for all $y_1\neq y_2\in \hat{Y}$ we have $G_{y_1}\simeq G_{y_2} = U_1.\mathbb{Z}_2$ and $G_{y_1}$ is not conjugate to $G_{y_2}$.
\end{lemma}

\begin{proof}
    This is proved in \cite[Lemma~5.69]{Rizzoli2}.
\end{proof}

\begin{lemma}\label{a1^3 preliminary lemma p not 2}
    Let $G=A_1^3$, $\lambda=\lambda_1\otimes\lambda_1\otimes\lambda_1$ with $p\neq 2$. The quadruple $(G,\lambda,p,2)$ has generic $ts$-stabilizer $\mathbb{Z}_2\times \mathbb{Z}_2$.
\end{lemma}

\begin{proof}
    We can recover this result by a slight change of the proof of \cite[Prop.~6.1.7]{generic}. Let $G=SL_2(K)^3$ with basis $e,f$ for $V_{nat}$, so that $V=V_{nat}\otimes V_{nat}\otimes V_{nat}$. Then, like in the proof of \cite[Prop.~6.1.7]{generic}, given $\mathbf{a}=(a_1,a_2,a_3,a_4)$ let \[v^{(1)}=a_1e\otimes e\otimes e+ a_2e\otimes f\otimes f+a_3f\otimes e\otimes f+a_4f\otimes f\otimes e, \]
    \[v^{(2)}=a_1f\otimes f\otimes f+ a_2f\otimes e\otimes e+a_3e\otimes f\otimes e+a_4e\otimes e\otimes f. \] Then $y_{\mathbf{a}}:=\langle v^{(1)},v^{(2)} \rangle$ is totally singular if and only if $a_1^2+a_2^2+a_3^2+a_4^2 = 0$. In the proof of \cite[Prop.~6.1.7]{generic} they then proceed to define $\hat{Y}$ as a dense subset of $Y:=\{ y_{\mathbf{a}}:\mathbf{a}\neq (0,0,0,0)\}$ by requiring certain polynomials on the coefficients to be non-zero. One of these conditions is that $a_1^2+a_2^2+a_3^2+a_4^2 \neq  0$. However, this is later only used for the case $p=2$. Therefore we can modify the definition of $\hat{Y}$ in the proof of \cite[Prop.~6.1.7]{generic}, and still end up with an open dense subset of $\mathcal{G}_2(V)$ where all stabilizers are conjugate to $\mathbb{Z}_2.\mathbb{Z}_2$. The difference is that now this set will also contain totally singular $2$-spaces, and therefore $C_{\mathcal{S}_2(V)}=C_{\mathcal{G}_2(V)}=\mathbb{Z}_2.\mathbb{Z}_2$.
\end{proof}

\begin{proposition}\label{large family of cases}
Let $G=E_7$ and $\lambda=\lambda_7$, or $G=D_6$ and $\lambda=\lambda_6$, or $G=A_5$ and $\lambda=\lambda_3$. Then the quadruple $(G,\lambda,p,2)$ has no generic $ts$-stabilizer if $p=2$, but has a semi-generic $ts$-stabilizer. If $p\neq 2$ we have $C_{\mathcal{S}_2(V)}=C_{\mathcal{G}_2(V)}$. 
\end{proposition}

\begin{proof}
    We shall describe how to use the proof of \cite[Proposition~$6.2.20$]{generic} to reach the conclusion. In \cite[Proposition~$6.2.20$]{generic} the authors determine the generic stabilizer for the $G$-action on all $2$-spaces. They do so in the following manner. They define a certain $8$-space $V_{[0]}$ spanned by pairs of opposite weight vectors. This $8$-space is the fixed point space of a subgroup $A$ of $G$, where $A=D_4,A_1^3,T_2$ according to whether $G=E_7,D_6,A_5$ respectively. They define a dense subset $\hat{Y}_1$ of $Y:=\mathcal{G}_2(V_{[0]})$ with the property that for any $y\in Y$ we have $\mathrm{Tran}_G(y,Y)  = A A_1^3.Sym(3)$, where $A_1^3$ acts on $V_{[0]}$ as $\lambda_1\otimes\lambda_1\otimes\lambda_1$. The set $\hat{Y}_1$ is defined by requiring certain expressions in terms of the coefficients of the given $V_{[0]}$ basis to be non-zero. Here the key observation is that these conditions do not exclude all totally singular $2$-spaces of $V_{[0]}$, and therefore $\hat{Y}^{\mathcal{S}}_1:=\hat{Y}_1\cap \mathcal{S}_2(V_{[0]})$ is a dense subset of $Y^{\mathcal{S}}:=\mathcal{S}_2(V_{[0]})$. Given $y\in\hat{Y}^{\mathcal{S}}_1 $, since $\mathrm{Tran}_G(y,Y)  = A A_1^3.Sym(3)$, we also have $\mathrm{Tran}_G(y,Y^{\mathcal{S}})  = A A_1^3.Sym(3)$. Furthermore, they show that $G_y = A(A_1^3)_y$ for all $y\in \hat{Y}_1$. 

    Assume $p\neq 2$. Then by Lemma~\ref{a1^3 preliminary lemma p not 2} there exists a dense open subset $\hat{Y}^{\mathcal{S}}_2$ of $Y^{\mathcal{S}}$ such that every stabilizer is $A_1^3$-conjugate to $\mathbb{Z}_2.\mathbb{Z}_2$. Taking the intersection with $\hat{Y}^{\mathcal{S}}_1$ we get an open dense subset $\hat{Y}^{\mathcal{S}}$ of $Y^{\mathcal{S}}$. For all $y\in\hat{Y}^{\mathcal{S}}$ we know that $G_y$ is conjugate to $A.\mathbb{Z}_2.\mathbb{Z}_2$. In each case the codimension of the transporter of $y\in \hat{Y}^{\mathcal{S}}$ into $Y^{\mathcal{S}}$ is equal to the codimension of $Y^{\mathcal{S}}$ in $\mathcal{S}_2(V)$. Therefore every $y\in \hat{Y}^{\mathcal{S}}$ is $Y^{\mathcal{S}}$-exact. By Lemma~\ref{loc to a subvariety lemma} we conclude that $C_{\mathcal{S}_2(V)} $ is $D_4.\mathbb{Z}_2.\mathbb{Z}_2,A_1^3.\mathbb{Z}_2.\mathbb{Z}_2,T_2.\mathbb{Z}_2.\mathbb{Z}_2$ according to whether $G=E_7,D_6,A_5$ respectively. These generic stabilizers are the same as for the action on all $2$-spaces.

    Now assume that $p=2$. By Lemma~\ref{a1^3 preliminary lemma p=2} there exists a dense open subset $\hat{Y}^{\mathcal{S}}_2$ of $Y^{\mathcal{S}}$ such that every stabilizer has a $1$-dimensional connected component and stabilizers are pairwise non-conjugate. Taking the intersection with $\hat{Y}^{\mathcal{S}}_1$ we get a dense open subset $\hat{Y}^{\mathcal{S}}_3$ of $Y^{\mathcal{S}}$.
    Let  $y_1,y_2\in \hat{Y}^{\mathcal{S}}_3$ and assume that $x.y_1 = y_2$. Then $x\in N_G(A)$ and since $V_{{0}}$ is the fixed space of $A$, we must have $x\in \mathrm{Tran}_G(y_1,Y^{\mathcal{S}})  = A A_1^3.Sym(3)$. Therefore $y_1$ and $y_2$ must be in the same $A_1^3.Sym(3)$-orbit. However by construction $y_1$ and $y_2$ are not in the same $A_1^3$-orbit, and therefore there exists a dense open subset $\hat{Y}^{\mathcal{S}}$ of $Y^{\mathcal{S}}$, contained in $\hat{Y}^{\mathcal{S}}_3$, such that any two distinct elements have non-conjugate stabilizers. In each case the codimension of the transporter of $y\in \hat{Y}^{\mathcal{S}}$ into $Y^{\mathcal{S}}$ is equal to the codimension of $Y^{\mathcal{S}}$ in $\mathcal{S}_2(V)$. Therefore every $y\in \hat{Y}^{\mathcal{S}}$ is $Y^{\mathcal{S}}$-exact. By Lemma~\ref{no generic stabilizer lemma} there is no generic stabilizer in the action of $G$ on $\mathcal{S}_2(V)$. In fact Lemma~\ref{loc to a subvariety open set lemma} shows that we have semi-generic stabilizers $D_4.U_1.\mathbb{Z}_2,A_1^3.U_1.\mathbb{Z}_2,T_2.U_1.\mathbb{Z}_2$ according to whether $G=E_7,D_6,A_5$ respectively.
\end{proof}

\begin{lemma}\label{a2 preliminary lemma k=2}
    Let $G=A_2$, $\lambda=\lambda_1+\lambda_2$ with $p\neq 3$. Let $\tau$ be a graph automorphism of $G$. Then the quadruple $(G\langle\tau\rangle , \lambda,p,2)$ has generic $ts$-stabilizer $Sym(3)$.
\end{lemma}

\begin{proof}
    Take $G = SL_3(K)$ acting on $V = \mathfrak{sl}_3(K)$ by conjugation. Here $Z(G) = \langle \mathrm{diag} (\omega,\omega,\omega) \rangle $ where $\omega$ is a non-trivial third-root of unity. Note that $Z(G)$ acts trivially on $V$. Let $\tau$ be the graph automorphism acting on $G$ as $g\mapsto g^{-T}$ and on $V$ as $v\mapsto -v^T$.
    We have that $G$ fixes a non-degenerate quadratic form on $V$ given by \[Q\left( (m_{ij})_{ij} \right) = m_{11}^2+m_{22}^2+m_{11}m_{22}+\sum_{i<j}m_{ij}m_{ji}.\] For $1\leq i,j\leq 3$, let  $e_{ij}$ denote a $3\times 3$ matrix with a $1$ in position $(i,j)$ and zeroes everywhere else. There are three $G$-orbits on $\mathcal{S}_1(V)$, which we label as $\Delta_1$, $\Delta_2$ and $\Delta_3$, respectively with representatives $\langle e_{13}\rangle$, $\langle e_{12}+e_{23}\rangle$, $\langle e_{11}+\omega e_{22}+\omega^2 e_{33}\rangle$. This follows directly from considering the Jordan Canonical Form of elements in $V$. The stabilizers are respectively $B = U_3T_2$, $Z(G).U_2T_1$ and $T_2.\mathbb{Z}_3$. 
    
    Let
\[ u_{bc} =  \left(\begin{matrix}  & 1 &  \\  &  & b \\c &  & \\  \end{matrix}\right),\quad 
 v_{ad} =  \left(\begin{matrix}  &  & 1 \\ a &  &  \\ & d & \\  \end{matrix}\right); \]
    \[ Y =  \left\lbrace \langle u_{bc}, v_{ad}\rangle : a+c+bd = 0 \right\rbrace.\]
The set $Y$ is a $3$-dimensional subvariety of $\mathcal{S}_2(V)$. Let 
\[
\hat{Y} = \lbrace \langle u_{bc}, v_{ad}\rangle : abcd \neq 0,\, \frac{(bd-c)^2}{bcd}\neq 0,-\frac{3}{2},-3 \rbrace,
\]
where we disregard the expression $\frac{(bd-c)^2}{bcd}\neq -\frac{3}{2}$ if $p = 2$.
Then $\hat{Y}$ is a dense subset of $Y$. Let $y = \langle u_{bc}, v_{ad}\rangle \in \hat{Y}$. Then $\langle u_{bc} \rangle $ and $\langle v_{ad}\rangle$ are in $\Delta_3$, since they are spanned by rank-$3$ matrices. Now consider $v = u_{bc}+\lambda v_{ad}$. We have $\det v = bc+ad \lambda^3$, which implies that there are precisely three $1$-spaces of $y$ not belonging to $\Delta_3$. It is clear that none of these have rank $1$, and therefore all three of these $1$-spaces belong to $\Delta_2$. Let $\lambda_1,\lambda_2,\lambda_3$ be the three distinct roots of $q(x) = bc+ad x^ 3 $, so that $\langle u_{bc}+\lambda_i v_{ad}\rangle \in \Delta_2$. Then \[G_y\leq (G_{\langle u_{bc}+\lambda_1 v_{ad}\rangle }\cap G_{\langle u_{bc}+\lambda_2 v_{ad}\rangle }).Sym(3). \]
Let $g^* = \diag(1,\omega,\omega^2)\in G$, where again $\omega$ is a non-trivial third-root of unity. Then $g^*. u_{bc} = \omega^2 u_{bc}$ and $g^*. v_{ad} = \omega v_{ad}$, implying $g^*\in G_y$ and $Z(G)\langle g^* \rangle\simeq Z(G).\mathbb{Z}_3\leq G_y$. Take $\mu,\,\nu \in K$ with $\mu^3 = \frac{1}{ac}$ and $\nu^3 = \frac{ab}{d}$, and let \[g^\ddag =\diag (\mu,\nu,{(\mu\nu)^{-1}})\tau
,\] an element of $N_{G\langle\tau\rangle}(\langle g^* \rangle)$ fixing $y$. All elements of the form $h\tau$ with $h\in T$ are conjugate under $T$. Therefore for any $y_1,y_2\in \hat{Y}$ we know that $(G\langle\tau\rangle)_{y_1}$ and $(G\langle\tau\rangle)_{y_2}$ contain a subgroup conjugate to $Z(G).\langle g^*\rangle \langle g^\ddag \rangle \simeq Z(G).Sym(3)$. 

We now proceed in the following way. We show that the stabilizer in $G_{\langle u_{bc}+\lambda_1 v_{ad}\rangle}$ of $y$ is $Z(G)$. This in turn implies that $G_y = Z(G).\mathbb{Z}_3$. Since $\langle u_{bc}+\lambda_1 v_{ad} \rangle \in \Delta_2$, we are able to find an element of $G$ sending $u_{bc}+\lambda_1 v_{ad} \mapsto e_{12}+e_{23}$. This is achieved by a scalar multiple of \[ R = \left(\begin{matrix} 0 & 0 & 1 \\
c &\lambda_1 d & 0 \\
-\frac{bc}{\lambda_1} & c & -\lambda_1 a\\  \end{matrix}\right).\] We now have \[R.y = \left\langle \left(\begin{matrix} a_1 & 0 & 1 \\ a_2 & a_1 & 0 \\a_3 & -a_2 & -2 a_1\\  \end{matrix}\right) ,  e_{12}+e_{23}\right\rangle,\]
where $a_1=\lambda_1a,\,a_2 = \frac{a}{d}(bd-c) = \frac{c^2}{d}-b^2d$, and $a_3 = \frac{3abc}{\lambda_1 d} = -3a_1^2$.
 By assumption $a_1a_2a_3\neq 0$. Now let $g\in G_{\langle e_{12}+e_{23} \rangle}$. Multiplying by an element of $Z(G)$, we can assume that $g=ns$ where $s = \mathrm{diag}(\frac{1}{t},1,t)$ and $n = \left(\begin{matrix} 1 & n_1 & n_2 \\ 0 & 1 & n_1 \\0 & 0 & 1\\  \end{matrix}\right).$
 Then \begin{flalign*}
     (m_{ij})_{ij}&:=g. \left(\begin{matrix} a_1 & 0 & 1 \\ a_2 & a_1 & 0 \\-3a_1^2 & -a_2 & -2 a_1\\  \end{matrix}\right) =& \\
     &= \left(\begin{matrix} 
 a_1 + a_2 t n_1 - 3 a_1^2 t^2n_2 & m_{12} & m_{13} \\ 
 t (a_2 - 3 a_1^2 t n_1) & a_1 - 2 a_2 t n_1 + 3 a_1^2 t^2 n_1^2 & m_{23} \\
 -3 a_1^2 t^2 & -t (a_2 - 3 a_1^2 t n_1) & -2 a_1 + a_2 t n_1 - 3 a_1^2t^2 (n_1^2- n_2)\\
 \end{matrix}\right),&
 \end{flalign*}
 where \begin{flalign*}
     m_{12}&=t (a_2 (-n_1^2 - n_2) + 3 a_1^2 t n_1 n_2),&\\
 m_{13}&= \frac{1}{t^2} + a_2 t n_1^3 - 3 a_1n_2 - 3 a_1^2 t^2 n_2 (n_1^2 - n_2),&\\
m_{23}&= -3 a_1 n_1 - 3 a_1^2 t^2 (n_1^2-n_2) + a_2 t (n_1^2 - n_2).&
 \end{flalign*}
Assume that $g$ fixes $R.y$. Since $m_{31}=-3a_1^2t^2$, we must have $(m_{ij})_{ij} = t^2 R.y + \alpha (e_{12}+e_{23})$ for some $\alpha\in K$. Therefore $m_{21} = t^2 a_2$ which implies $n_1 = \frac{a_2(1-t)}{3 a_1^2 t}.$ Similarly, we must have $m_{11}=m_{22}$ which implies $n_2 = \frac{a_2 n_1}{a_1^2 t }-n_1^2$. Then $m_{11} = t^2 a_1$ implies that either $t=\pm 1$ or $a_2^2= 3a_1^3$. 
Assume that $a_2^2= 3a_1^3$. Since $a_2=\frac{a}{d}(bd-c)$ and $a_1=\lambda_1 a$, the equation $a_2^2= 3a_1^3$ implies $(bd-c)^2 = -3bcd$, contrary to the definition of $\hat{Y}$. Therefore $t=\pm 1$. If $t=1$ we immediately get $n_1=n_2=0$, concluding. Assume therefore that $t=-1$ and $p\neq 2$. Then $m_{12}=m_{23}$ forces $2a_2^2 = 3 a_1^3$, which implies $ 2(bd-c)^2 = -3bcd$, which is impossible by assumption on $\hat{Y}$.

 This concludes the proof that $G\langle\tau\rangle_y = Z(G)\langle g^*\rangle\langle g^\ddag \rangle$.
 Now for any $y\in \hat{Y}$, since $G_y= Z(G)\langle g^*\rangle$, any element in $\tran _G(y,Y)$ must be in $N( \langle g^*\rangle)$. We know that $N( \langle g^*\rangle) = T_2.\mathbb{Z}_3$, and it is easy to check that $T_2\in \tran_G(y,Y)$. Therefore $\dim \tran_G(y,Y) =2 $ and then since $\dim  \mathcal{S}_2(V)-\dim G = 1$, the set $\hat{Y}$ is $Y$-exact. By Lemma~\ref{loc to a subvariety lemma} we conclude that the quadruple $(G\langle\tau\rangle , \lambda,p,2)$ has generic $ts$-stabilizer $Sym(3)$.

\end{proof}

\begin{proposition}\label{F4 and C3 p not 3 k=2}
    Let $G=F_4$ and $\lambda=\lambda_4$, or $G=C_3$ and $\lambda=\lambda_2$, with $p\neq 3$. Then the quadruple $(G,\lambda,p,2)$ has generic $ts$-stabilizer $A_2.Sym(3)$ or $T_1.Sym(3)$ respectively.
\end{proposition}

\begin{proof}
    This is entirely similar to the proof of Prop~\ref{large family of cases}, and relies on the construction used in \cite[Prop.~6.2.18]{generic}. All we have to observe is that the set $\hat{Y}_1$ defined in the proof of \cite[Prop.~6.2.18]{generic} does indeed contain totally singular $2$-spaces. This follows from the observation in the proof of \cite[Prop.~6.2.17]{generic}, where the authors need to show that the set $\hat{Y}_1$ is non-empty. They do so by saying that $v^{(1)}=a_{33}e_{\gamma_{33}}+a_{12}e_{\gamma_{12}}+a_{21}e_{\gamma_{21}}$, $v^{(2)}=b_{11}e_{\gamma_{11}}+b_{23}e_{\gamma_{23}}+b_{32}e_{\gamma_{32}}$, and $v^{(3)}=e_{\gamma_{22}}+e_{\gamma_{31}}+e_{\gamma_{13}}$ span a $3$-space in $\hat{Y}_1$ if $(a_{12}b_{23}-a_{33}b_{11})(a_{21}b_{32}-a_{12}b_{23})(a_{33}b_{11}-a_{21}b_{32})\neq 0$. 
    Clearly there are totally singular $2$-spaces $\langle v^{(1)},v^{(2)}\rangle$ with coefficients satisfying this condition, and therefore the set $\hat{Y}_1$ defined in the proof of \cite[Prop.~6.2.18]{generic} does contain totally singular $2$-spaces.

    Once we understand this, the generic stabilizer is respectively $A_2.X$, $T_1.X$, where $X$ is the generic stabilizer for the action of $A_2.\mathbb{Z}_2$ on $\mathcal{S}_2(\lambda_1+\lambda_2)$. By Lemma~\ref{a2 preliminary lemma k=2} we conclude.
\end{proof}

\begin{proposition}\label{proposition a2 p=3 k=3}
Let $G=A_2$ and $\lambda=\lambda_1+\lambda_2$ with $p=3$. Then $C_{\mathcal{S}_3}(V) = T_2.\mathbb{Z}_3$.
\end{proposition}

\begin{proof}
Let $\alpha_1,\alpha_2$ be the fundamental roots for $A_2$ and let $\alpha_3=\alpha_1+\alpha_2$. The adjoint module $\mathrm{Lie}(G)$ has the Chevalley basis $e_{\alpha_3},e_{\alpha_2},e_{\alpha_1},h_{\alpha_1},h_{\alpha_2},e_{-\alpha_1},e_{-\alpha_2},e_{-\alpha_3}$. We write $v_1v_2$ for the Lie product of vectors $v_1,v_2\in \mathrm{Lie}(G)$. We assume that the structure constants are as described by the matrix 

\[\left(\begin{matrix} 0 & 0 & 0 &-1 & 1 & 0 \\ 0 & 0 & -1 &0 & 0 & 1 \\0 & 1 & 0 &0 & 0 & -1\\  1 & 0 & 0 &0 & -1 & 0 \\ -1 & 0 & 0 &1 & 0 & 0\\ 0 & -1 & 1 &0 & 0 & 0\\ \end{matrix}\right),\]

where the rows and columns are in the order $\alpha_3,\alpha_2,\alpha_1,-\alpha_1-\alpha_2,-\alpha_3$. By \cite[Lemma~5.4]{Rizzoli2} we can explicitly construct our highest weight irreducible module as:
$$V_{G}(\lambda_1+\lambda_2)= \mathrm{Lie}(G)/\langle h_{\alpha_1}-h_{\alpha_2} \rangle.$$

In a slight abuse of notation we omit writing the quotient, so that $v$ actually stands for $v+\langle h_{\alpha_1}-h_{\alpha_2} \rangle$. We  order the basis for $V_{G}(\lambda_1+\lambda_2 )$ as $e_{\alpha_3},e_{\alpha_2},e_{\alpha_1},h_{\alpha_1},e_{-\alpha_1},e_{-\alpha_2},e_{-\alpha_3}$. With respect to this ordering, using standard formulas found in \cite[§4.4]{cartersimple}, we find the matrices denoting the transformations $x_{\pm \alpha_1}(t),x_{\pm \alpha_2}(t),x_{\pm \alpha_3}(t)$, as well as $h_{\alpha_1}(\kappa)$ and $h_{\alpha_2}(\kappa)$. These are straightforward calculations and we therefore only state the results.

\begin{flalign*}
    x_{\alpha_1}(t) & =
    \begin{pmatrix}
1& t& 0& 0& 0& 0& 0\\
 0& 1& 0& 0& 0& 0& 0\\
 0& 0& 1& t& -t^2& 0& 0\\
 0& 0& 0& 1& t& 0& 0\\
 0& 0& 0& 0& 1& 0& 0\\
 0& 0& 0& 0& 0& 1& -t\\
 0& 0& 0& 0& 0& 0& 1
    \end{pmatrix}, & x_{-\alpha_1}(t) & =
    \begin{pmatrix}
1& 0& 0& 0& 0& 0& 0\\
 t& 1& 0& 0& 0& 0& 0\\
 0& 0& 1& 0& 0& 0& 0\\
 0& 0& -t& 1& 0& 0& 0\\
 0& 0& -t^2& -t& 1& 0& 0\\
 0& 0& 0& 0& 0& 1& 0\\
 0& 0& 0& 0& 0& -t& 1
    \end{pmatrix}, & \\
    x_{\alpha_2}(t) & =
    \begin{pmatrix}
1& 0& -t& 0& 0& 0& 0\\
 0& 1& 0& t& 0& -t^2& 0\\
 0& 0& 1& 0& 0& 0& 0\\
 0& 0& 0& 1& 0& t& 0\\
 0& 0& 0& 0& 1& 0& t\\
 0& 0& 0& 0& 0& 1& 0\\
 0& 0& 0& 0& 0& 0& 1
    \end{pmatrix}, & x_{-\alpha_2}(t) & =
    \begin{pmatrix}
1& 0& 0& 0& 0& 0& 0\\
 0& 1& 0& 0& 0& 0& 0\\
 -t& 0& 1& 0& 0& 0& 0\\
 0& -t& 0& 1& 0& 0& 0\\
 0& 0& 0& 0& 1& 0& 0\\
 0& -t^2& 0& -t& 0& 1& 0\\
 0& 0& 0& 0& t& 0& 1
    \end{pmatrix}, & \\
    x_{\alpha_3}(t) & =
    \begin{pmatrix}
1& 0& 0& -t& 0& 0& -t^2\\
 0& 1& 0& 0& -t& 0& 0\\
 0& 0& 1& 0& 0& t& 0\\
 0& 0& 0& 1& 0& 0& -t\\
 0& 0& 0& 0& 1& 0& 0\\
 0& 0& 0& 0& 0& 1& 0\\
 0& 0& 0& 0& 0& 0& 1
    \end{pmatrix}, & x_{-\alpha_3}(t) & =
    \begin{pmatrix}
1& 0& 0& 0& 0& 0& 0\\
 0& 1& 0& 0& 0& 0& 0\\
 0& 0& 1& 0& 0& 0& 0\\
 t& 0& 0& 1& 0& 0& 0\\
 0& -t& 0& 0& 1& 0& 0\\
 0& 0& t& 0& 0& 1& 0\\
 -t^2& 0& 0& t& 0& 0& 1
    \end{pmatrix}, & 
\end{flalign*}

$h_{\alpha_1}(\kappa)= \mathrm{diag}(\kappa,\kappa^{-1},\kappa^2,1,\kappa^{-2},\kappa,\kappa^{-1})$ and  $h_{\alpha_2}(\kappa)= \mathrm{diag}(\kappa,\kappa^2,\kappa^{-1},1,\kappa,\kappa^{-2},\kappa^{-1})$.

Let $(\cdot,\cdot):V\times V \rightarrow K$ be the non-degenerate symmetric bilinear form given by $(e_{\alpha_i},e_{-\alpha_j})=\delta_{ij}$, $(e_{\alpha_i},e_{\alpha_j})=(e_{-\alpha_i},e_{-\alpha_j})=0$,  $(h_{\alpha_1},e_{\pm\alpha_i})=0$ and $(h_{\alpha_1},h_{\alpha_1})=-1$, where $1\leq i,j\leq 3$. Then $G$ fixes this form, as can be seen by just checking the action of the generators. 

We need some information about the action of $G$ on singular $1$-spaces. Let $T=\langle h_{\alpha_1}(\kappa),h_{\alpha_2}(\kappa):\kappa\in K^*\rangle$ be the standard maximal torus and $B=\langle T,x_{\alpha_1}(t),x_{\alpha_2}(t):t\in K\rangle$ a Borel subgroup.
Then by \cite[Lemma~5.5]{Rizzoli2} the group $G$ has $2$ orbits on singular vectors in $V$, with representatives $x=e_{\alpha_3}$ and $y=e_{\alpha_1}+e_{\alpha_2}$. Furthermore $G_x=U_3T_1$, $G_y=U_2$, $G_{\langle x \rangle}=U_3T_2=B$ and $G_{\langle y \rangle}=U_2T_1\leq B$.

We now define a totally singular $3$-space that we will show has a $2$-dimensional stabilizer, therefore belonging to a dense orbit. Let \[W_3 = \langle e_{\alpha_1}, e_{\alpha_2},e_{-\alpha_3}\rangle.\]
Let $\Delta$ be the $G$-orbit with representative $\langle x\rangle$, where $x = e_{\alpha_3}$.
We start by observing that \[\mathcal{G}_1(W_3)\cap \Delta = \{\langle e_{\alpha_1} \rangle, \langle e_{\alpha_2} \rangle,\langle e_{-\alpha_3}\rangle\}.\]
It suffices to show that every other $1$-space of $W_3$ is in the same orbit as $\langle y\rangle$, where  $y = e_{\alpha_1}+ e_{\alpha_2}$. This is clear for $\langle e_{\alpha_1}+ \lambda e_{\alpha_2}\rangle$ when ${\lambda\neq 0} $. The element $x_{-\alpha_1}(t)x_{-\alpha_2}(-t)x_{-\alpha_3}(t^2)$ sends $y\mapsto y+t^3e_{-\alpha_3}$, therefore $\langle e_{\alpha_1}+\lambda e_{\alpha_2}+\gamma e_{-\alpha_3}\rangle $ is in the orbit of $\langle y \rangle$ for $\lambda,\gamma\neq 0$. Finally $n_{\alpha_2}(1)n_{\alpha_1}(1). \langle y\rangle =\langle e_{\alpha_1}+e_{-\alpha_3} \rangle $ and $n_{\alpha_1}(-1)n_{\alpha_2}(-1). \langle y\rangle =\langle e_{\alpha_2}+e_{-\alpha_3} \rangle $, concluding.

Since $\dim \mathcal{S}_3(V) = 6$, it suffices to prove that $G_{W_3} = T.\mathbb{Z}_3$. Since $y$ only contains three $1$-spaces that are in the orbit $\Delta$, the connected component of $G_{W_3}$ is simply \[G_{\langle e_{\alpha_1}\rangle}\cap G_{\langle e_{\alpha_2}\rangle}\cap G_{\langle e_{-\alpha_3}\rangle}, \] which is easily seen to be the maximal torus $T$. Finally one checks that the elements of order $2$ of $W=N_G(T)/T$ do not fix $W_3$, while the elements of order $3$ do. Therefore $G_{W_3} \simeq T_2.\mathbb{Z}_3$, and indeed $C_{\mathcal{S}_3}(V) = T_2.\mathbb{Z}_3$.
\end{proof}

\begin{proposition}\label{proposition a2 k=4}
Let $G=A_2$ and $\lambda = \lambda_1+\lambda_2$ with $p\neq 3$. Then $C_{\mathcal{S}'_4(V)}=C_{\mathcal{S}''_4(V)}=T_2.\mathbb{Z}_3$.
\end{proposition}

\begin{proof}
The group $G$ is stable under a triality automorphism of $D_4$, therefore the $(G,P_3)$-double cosets in $D_4$ are in bijection with the $(G,P_1)$-double cosets in $D_4$, as are the $(G,P_4)$-double cosets. By Proposition~\ref{special adjoint cases} the group $G$ acts on $\mathcal{S}_1(V)$ with generic stabilizer $T_2.\mathbb{Z}_3$, concluding.
\end{proof}

\begin{proposition}\label{proposition b3 k=4}
Let $G=B_3$ and $\lambda = \lambda_3$. Then $C_{\mathcal{S}'_4(V)}=C_{\mathcal{S}''_4(V)}=C_{\mathcal{S}_1(V)}=U_6A_2T_1$.
\end{proposition}

\begin{proof}
Similarly to Proposition~\ref{proposition a2 k=4}, the $(G,P_3)$-double cosets in $D_4$ are in bijection with the $(G,P_1)$-double cosets in $D_4$, as are the $(G,P_4)$-double cosets. Therefore $C_{\mathcal{S}'_4(V)}=C_{\mathcal{S}''_4(V)}=C_{\mathcal{S}_1(V)}=P_3(G)$, as claimed.
\end{proof}

\begin{proposition}\label{proposition b4 k=3}
Let $G=B_4$ and $\lambda=\lambda_4$. Then $C_{\mathcal{S}_3(V)} = A_1$.
\end{proposition}

\begin{proof}
 By \cite[Prop. 6.2.13]{generic} we already know that $D_5$ has a dense orbit on all $3$-spaces of $V_{D_5}(\lambda_5)$. We construct the module $V_{D_5}(\lambda_5)$ in the same way as in \cite[Prop. 6.2.13]{generic} and then consider the restriction to $G$. Let $\beta_1,\dots,\beta_6$ be the simple roots of a group of type $E_6$ and let $D_5<E_6$ have simple roots $\alpha_1=\beta_1, \alpha_2=\beta_3, \alpha_3=\beta_4, \alpha_4=\beta_5, \alpha_5=\beta_2$. Then we may take \[V_{D_5}(\lambda_5) = \langle e_{\alpha}:\alpha = \sum m_i\beta_i, m_6 = 1\rangle < \mathrm{Lie}(E_6).\] With this notation \cite[Prop. 6.2.13]{generic} shows that if we write \[\gamma_2 = e_{101111},\, \gamma_3 = e_{011111},\,\gamma_4 =  e_{111111},\,\gamma_5=e_{011211},\,\gamma_6= e_{111211},\,\gamma_7 =e_{011221} ,\] then
 \[W_3:= \langle e_{\gamma_2}+e_{\gamma_3}, e_{\gamma_4}+e_{\gamma_5}, e_{\gamma_6}+e_{\gamma_7}\rangle \] has stabilizer $A_1A_1$ in $D_5/Z(D_5)$. The generators for the first $A_1$ are simply $X_{\pm \rho}$, where $\rho$ is the longest root of $D_5$, while the second $A_1$ is generated by 
 \begin{flalign*}
     x(t) & = x_{\beta_1}(-t)x_{\beta_4}(2t)x_{\beta_5}(t)x_{\beta_2}(3t)x_{\beta_4+\beta_5}(-t^2)x_{\beta_2+\beta_4}(3t^2)x_{\beta_2+\beta_4+\beta_5}(4t^3) \text{ for } t\in K; & \\
     T_1& =\{h_{\beta_1}(\kappa)h_{\beta_2}(\kappa^3)h_{\beta_4}(\kappa^4)h_{\beta_5}(\kappa^3):\kappa\in K^*\};&\\
     n & = n_{\beta_1}n_{\beta_4}n_{\beta_2\beta_4\beta_5}^{-1}. &
 \end{flalign*}
 Let $\{e_1,\dots,e_5,f_5,\dots, f_1\}$ be the standard basis for the natural $D_5$-module, and let $G=(D_5)_{\langle e_2+f_2 \rangle}$. We will show that $W_3$ is totally singular and has stabilizer $A_1$ in $G/Z(G)$. Let $T=\langle h_{\beta_i}(\kappa_i): 1\leq i\leq 5, \kappa_i\in K^* \rangle $ be the standard maximal torus of $D_5$. Then $T_G:=T\cap G= \langle h_{\beta_i}(\kappa_i): 1\leq i\leq 5, \kappa_i\in K^* , \kappa_1=\kappa_3 \rangle$, since $\alpha_1=\beta_1$ and $\alpha_2=\beta_3$. Then the $T_G$-weights on $\langle e_{\gamma_i}\rangle_{2\leq i\leq 7}$ are respectively given by $\frac{\kappa_1}{\kappa_2},\frac{\kappa_2}{\kappa_4},\frac{\kappa_1\kappa_2}{\kappa_4},\frac{\kappa_4}{\kappa_1\kappa_5},\frac{\kappa_4}{\kappa_5},\frac{\kappa_5}{\kappa_1}$. No two such weights form a pair of opposite weights, and therefore by Lemma~\ref{weight spaces orthogonality lemma} the subspace $W_3$ is totally singular.

 Finally, we find that the diagonal subgroup of $A_1A_1$ with positive root subgroup $x(t)x_{\rho}(t)$ fixes $\langle e_2+f_2\rangle$. By maximality of this diagonal $A_1$ in $A_1A_1$, this means that $(A_1A_1)\cap G =A_1$. Therefore $\dim G - \dim \mathcal{S}_3(V) = 3 = \dim G_{W_3}$, which implies that $C_{\mathcal{S}_3(V)} = A_1$.
\end{proof}

In what follows we refer the reader back to Section~\ref{spin modules section} for the relevant notation on spin modules. 
For the following propositions, let $G=B_4 = (D_5)_{e_5-f_5}$, $\lambda=\lambda_4$ and order the basis of $V$ as 
\begin{flalign*}
 v_1&= 1, & v_5&=e_1e_5, & v_9&=e_1e_2e_3e_5, & v_{13}&=e_1e_4,\\
 v_2&=e_1e_2, & v_6&=e_2 e_5, &v_{10}&=e_1e_2e_4e_5, & v_{14}&=e_2e_4,\\
 v_3&=e_1e_3, &v_7&=e_3e_5, & v_{11}&=e_1e_3e_4e_5,&  v_{15}&=e_3e_4,\\
  v_4&=e_2e_3, & v_8&=e_4e_5, & v_{12}&=e_2e_3e_4e_5,& v_{16}&=e_1e_2e_3e_4.\\
 \end{flalign*}
\begin{lemma}\label{quadratic form spin description}
The quadratic form given by the matrix\\ $\antidiag (1,-1,1,-1,1,-1,1,-1,0,0,0,0,0,0,0,0)$ defines a non-degenerate quadratic form on $V_{B_4}(\lambda_4)$ fixed by the $B_4$-action. 
\end{lemma}

\begin{proof}
The pairs $v_i,v_{17-i}$ are pairs of opposite weight vectors. We can assume that $Q(v_1+v_{16})=1$ and use the $B_4$-action to determine whether $Q(v_i+v_{17-i})$ is $1$ or $-1$. Let $g=1+e_1e_2\in B_4$. Then $g.v_1 = 1+e_1e_2 = v_1+v_2$ and $g.v_{15}=e_3e_4+e_1e_2e_3e_4=v_{15}+v_{16}$. Therefore $0=Q(v_1+v_{15}) = Q(v_1+v_2+v_{15}+v_{16})=1+Q(v_2+v_{15})$, as claimed. The same approach shows that $Q(v_3+v_{14})=1$ and  $Q(v_4+v_{13})=-1$. To conclude let $g= (1+e_1f_5)(1+e_1e_5)$ which is an element of $B_4$ since it fixes $e_5-f_5$. Then $0=Q(v_1+v_{12}) = Q(v_1+v_5+v_{12}-v_{16})=-1+Q(v_5+v_{12})$. The remaining cases follow similarly. 
\end{proof}

\begin{lemma}\label{spinors to root subgroups}
For $\alpha\in \Phi(B_4)$, the root elements $x_{\alpha}(t)$ are written in terms of spinors as follows:
\begin{flalign*}
x_{\epsilon_i-\epsilon_j}(t) &= 1+t e_i f_j, &\\
x_{-\epsilon_i+\epsilon_j}(t) &= 1+t e_j f_i, &\\
x_{\epsilon_i+\epsilon_j}(t) &= 1+t e_i e_j, &\\
x_{-\epsilon_i-\epsilon_j}(t) &= 1-t f_i f_j, &\\
x_{\epsilon_i}(t) &= (1+t e_i e_5)(1+t e_i f_5), &\\
x_{-\epsilon_i}(t) &= (1-t f_i e_5)(1-t f_i f_5). &
 \end{flalign*}
\end{lemma}
\begin{proof}
    The action on $V_{nat}$ of the elements on the left hand sides of each equation is as described at the beginning of Section~\ref{bilinear forms section}. The action on $V_{nat}$ of the elements on the right hand sides of each equation is described in Section~\ref{spin modules section}. The result follows by comparing the two actions.
\end{proof}

\begin{proposition}\label{proposition b4 P_7}
Let $G=B_4$, $\lambda=\lambda_4$. Then $C_{\mathcal{S}'_8(V)}=A_2.\mathbb{Z}_2$.
\end{proposition}

\begin{proof}
Suppose that $p\neq 3$. Let $\beta_1 = \alpha_1$, $\beta_2 = \alpha_2$, $\beta_3 = \alpha_3$, $\beta_4 = -\alpha_0$, where $\alpha_0$ is the longest root in $\Phi^+(B_4)$. Then $\{\beta_i\}_i$ is the base of a root system of type $D_4$. Let $D$ be the corresponding $D_4$-subgroup of $B_4$. Then $V\downarrow D = \lambda_3+\lambda_4=V_8+V_8'$. Let $\omega$ be a non-trivial third-root of unity. Let $\tau$ be the composition of $h_{\beta_2}(\omega )$ with the triality automorphism of $D$ sending $x_{\beta_i}(t)\mapsto x_{\beta_{\sigma.i}}(t)$ for $i=1,3,4$ and $\sigma = (134)$. Then the fixed points in $D$ under the triality automorphism $\tau$ form an irreducible $A_2$-subgroup of $D$. Using the structure constants inherited from $B_4$, let $A$ be the irreducible $A_2$-subgroup of $D$ given by 
 \begin{flalign*}
     A =\langle &x_{\beta_1}(t)x_{\beta_3}(\omega^2 t)x_{\beta_4}(\omega t), &\\
 &x_{-\beta_1}(t)x_{-\beta_3}(\omega t)x_{-\beta_4}(\omega^2 t),&\\
&x_{\beta_1+\beta_2}(t)x_{\beta_2+\beta_3}(-\omega t)x_{\beta_2+\beta_4}(\omega^2 t),&\\
&x_{-\beta_1-\beta_2}(t)x_{-\beta_2-\beta_3}(-\omega^2 t)x_{-\beta_2-\beta_4}(\omega t) : t \in K\rangle.&
 \end{flalign*}
 By Lemma~\ref{spinors to root subgroups} this is the same as
\begin{flalign*}
     A =\langle &(1+te_1f_2)(1+\omega^2 t e_3f_4)(1-\omega t f_1f_2), &\\
 &(1+te_2f_1)(1+\omega t e_4f_3)(1+\omega^2 t e_1e_2),&\\
&(1+te_1f_3)(1-\omega t e_2f_4)(1-\omega^2 t f_1f_3),&\\
&(1+te_3f_1)(1-\omega^2 t e_4f_2)(1+\omega t e_1e_3) : t \in K\rangle.&
 \end{flalign*}
With this setup we have $V_8\downarrow A \simeq V_8'\downarrow A\simeq V_{A_2}(\lambda_1+\lambda_2)$. Then $A$ fixes all $8$-spaces of the form $\{v+\lambda\phi(v):v\in V_8\}$ where $\phi$ is an $A$-module isomorphism between $V_8$ and $V_8'$. Given our explicit generators for $A$, it is easy to verify that we can take $\phi$ acting as:
\begin{align*}
e_1e_2e_3e_4 &\mapsto e_2e_3e_4e_5,  & e_2e_3 &\mapsto e_4e_5, & e_3e_4 &\mapsto \omega^2 e_1e_3e_4e_5,\\
e_1e_2 &\mapsto \omega e_2e_5, & e_1e_4 &\mapsto e_1e_2e_3e_5, & 1 &\mapsto e_1e_5,\\
e_1e_3 &\mapsto \omega^2 e_3e_5, & e_2e_4 &\mapsto \omega e_1e_2e_4e_5. &  &
\end{align*}

When $\lambda\neq 0$, the group $A$ must be the connected component of the stabilizer of $\{v+\lambda\phi(v):v\in V_8\}$, since the only minimal connected overgroup of $A$ in $G$ is $D$, which only fixes the $8$-spaces $V_8$ and $V_8'$. Also, $N_G(A)=N_{D.\mathbb{Z}_2}(A) = Z(G).A.\mathbb{Z}_2 = Z(G).A\langle \tau_2\rangle$, where $\tau_2$ acts as a graph automorphism on $D_4$ and $A_2$, swapping $V_8$ and $V_8 '$. Explicit calculations show that we can take $\tau_2 = h_{\alpha_1}(-1) n $, where $n=n_1n_2n_1n_3n_4n_3n_2n_1$, for $n_i = n_{\alpha_i}$. One then checks that $\tau_2$ fixes $\{v+\lambda\phi(v):v\in V_8\}$ when $\lambda^2 = -1$, i.e. when $\{v+\lambda\phi(v):v\in V_8\}$ is totally singular. Since $ \dim G -\dim \mathcal{S}_8'(V) = 8 = \dim A$, we conclude that $C_{\mathcal{S}_8'(V)} = A_2.\mathbb{Z}_2$.

Now assume that $p=3$. This time let $A$ be the $A_2$-subgroup of $G$ obtained via $V_{A_2}(\lambda_1)\otimes V_{A_2}(\lambda_2)$. The subgroup $A$ acts indecomposably on the natural module for $G$, as $1/7/1$. More concretely, we can realise $A$ as the subgroup generated by
\begin{flalign*}
 x_{\gamma_1}(t)&:=x_{\alpha_2}(t)x_{\alpha_4}(-t)= (1+t e_2f_3) (1-t e_4e_5)(1-t e_4f_5),&\\
 x_{-\gamma_1}(t)&:=x_{-\alpha_2}(t)x_{-\alpha_4}(-t)= (1+t e_3f_2) (1+t f_4e_5)(1+t f_4f_5),&\\
 x_{\gamma_2}(t)&:=x_{\alpha_2+\alpha_3}(t)x_{\alpha_3+\alpha_4}(t)x_{\alpha_0-\alpha_2}(t)= (1+t e_2f_4) (1+t e_3e_5)(1+t e_3f_5)(1+t e_1e_3),&\\
 x_{-\gamma_2}(t)&:=x_{-\alpha_2-\alpha_3}(t)x_{-\alpha_3-\alpha_4}(t)x_{\alpha_1+\alpha_2}(-t)=  (1+t e_4f_2) (1-t f_3e_5)(1-t f_3f_5)(1-t e_1f_3),&
 \end{flalign*}
 as $t$ varies over $K$. Now let $W$ be the $8$-space spanned by vectors
\[ \begin{array}{lll}%
e_3e_5,  &  e_2e_5-e_3e_4,  &  e_2e_4,\\
1-e_1e_5,  &  e_1e_2e_3e_4+e_2e_3e_4e_5, &  e_2e_3+e_1e_2e_3e_5,\\
e_1e_2+e_1e_3e_4e_5,  &  e_1e_4-e_4e_5.  &  
\end{array}\]

By Lemma~\ref{quadratic form spin description} the subspace $W$ is totally singular. Simple calculations show that $A\leq G_{W}$. Furthermore let $i$ be a square root of $-1$ and \[\tau = h_{\alpha_1}(-1)h_{\alpha_2}(-1)h_{\alpha_3}(-1)h_{\alpha_4}(i)n_{0122}.\] One checks that $\tau\in N_{G}(A_2)$ and $\tau \in G_W$. Now let $M$ be a minimal connected overgroup of $A$ that fixes $W$. Since there are no irreducible subgroups of $G$ containing $A$, we must have $A\leq M\leq P_1 = G_{\langle e_1 \rangle} = U_7B_3T_1$. Since the projection of $A$ on $B_3$ is an irreducible $A_2<B_3$, the projection of $M$ on $B_3$ is either $A_2$ or $G_2$. As there is only one conjugacy class of $G_2$'s in $P_1$, corresponding to the $G_2\leq B_3$, the last case is not possible. Therefore $M\leq U_7A_2T_1$, and as the $A_2$ is acting irreducibly on the $U_7$, we must have $M = U_7A_2$. It is however straightforward to check that $ U_7\not\leq G_W$. This proves that $A=(G_W)^0$. The final step is to show that $C_G(A) = Z(G)$. One way to do this is to consider the centralizer $C_1$ of $T\cap A = \langle h_{\alpha_2}(\kappa)h_{\alpha_4}(\kappa),h_{\alpha_3}(\kappa)\rangle_{\kappa\in K^*}$, a maximal torus of $A$. We find that $C_1=\langle T,X_{\pm 1111},n_0\rangle\simeq A_1T_3.\mathbb{Z}_2$, where $n_0$ is an element of $N_G(T)$ sending each root to its negative. The centralizer of $A$ must be contained in $P_1$, as $\langle e_1 \rangle$ is the only $1$-space stabilised by $A$. We have $C_1\cap P_1 = \langle T,X_{1111}\rangle$, and at this point it is easy to see that $C_G(A)=C_{C_1\cap P_1}(A) = Z(G)$. Therefore $G_W=Z(G).A.\mathbb{Z}_2$ and by dimensional considerations, $A_2.\mathbb{Z}_2\leq G/Z(G)$ is the generic stabilizer for the $G$-action on $\mathcal{S}_8'(V)$.
\end{proof}

\begin{proposition}\label{proposition b4 k = 7}
Let $G=B_4$, $\lambda=\lambda_4$. Then $C_{\mathcal{S}_7(V)}=T_2.\mathbb{Z}_2$.
\end{proposition}

\begin{proof}
    By Proposition~\ref{proposition b4 P_7} there is a dense $G$-orbit on $\mathcal{S}'_8(V)$, with stabilizer $A_2.\mathbb{Z}_2$. Fix an $8$-space $y$ in this orbit. Then $G_y = A_2.\mathbb{Z}_2 = A_2\langle\tau\rangle$ acts on $y$ as on $\mathrm{Lie}(A_2)$, with $\tau$ acting on $\mathrm{Lie}(A_2)$ by transposition. When $p\neq 3$ the quadruple $(A_2.\mathbb{Z}_2,\lambda_1+\lambda_2,p,1)$ has generic stabilizer $T_2.\mathbb{Z}_2$, as the open set for the $A_2$-action is constructed like in Lemma~\ref{adjoint module general proposition} starting from elements in $\mathrm{Lie}(T)$ which are fixed by $\tau$. The same is actually true also when $p=3$. In this case the $A_2\langle \tau \rangle$ module $y$ is not irreducible, but we can still build an open dense subset of $ \mathcal{G}_1(y)$ consisting of orbits of regular semisimple elements, such that all stabilizers are conjugate to $T_2.\mathbb{Z}_2$. Since the action on $\mathcal{G}_1(y)$ is isomorphic to the action on $\mathcal{G}_7(y)$, there is an open dense subset $\hat{X}$ of $X:=\mathcal{G}_7(y)$ such that for all $x\in \hat{X}$ the stabilizer $(A_2.\mathbb{Z}_2)_x$ is $A_2.\mathbb{Z}_2$-conjugate to $T_2.\mathbb{Z}_2$. Note that $X\subset \mathcal{S}_7(V)$, and since every element of $\mathcal{S}_7(V)$ is a subspace of precisely one element of $\mathcal{S}'_8(V)$ we must have $\mathrm{Tran}_G(x,X) = G_y$ for all $x\in X$. As $\dim G - \dim \mathrm{Tran}_G(x,X) = 28 = \dim \mathcal{S}_7(V) - \dim X$, the set $\hat{X}$ is $X$-exact. By Lemma~\ref{loc to a subvariety lemma} we conclude that $C_{\mathcal{S}_7(V)}=T_2.\mathbb{Z}_2$.
\end{proof}

\begin{proposition}\label{proposition b4 P_8 p not 2}
Let $G=B_4$, $\lambda = \lambda_4$ with $p\neq 2$. Then $C_{\mathcal{S}''_8(V)}=A_1^3$.
\end{proposition}

\begin{proof}
    Let $V_{nat}$ be the natural module for $G$ and let $V_1\oplus V_2\oplus V_3$ be an orthogonal decomposition of $V_{nat}$ such that $\dim V_i= 3$.
    Let $S=A_1^3$ be the the connected component of the stabilizer of this orthogonal decomposition. Then $N_G(S)$ is a maximal subgroup of $G$ isomorphic to $(\mathbb{Z}_2^2\times S).Sym(3)$. The group $S$ acts homogeneously on $V$ as the sum of two copies of the $8$-dimensional irreducible $S$-module  $\lambda_1\otimes \lambda_1\otimes \lambda_1$. Let $Y$ be the $1$-dimensional variety of non-trivial proper $S$-submodules of $V$, i.e. the set of all $8$-dimensional $S$-submodules of $V$. Since $p\neq 2$, any such $8$-space must be totally singular. We will now show that no element of $N_G(S)/S$ acts trivially on $Y$. Let $\tau$ be a pre-image under the the canonical projection $N_G(S) \rightarrow S$ of one of the $3$ non-trivial reflections in $(\mathbb{Z}_2^2\times S)/S$. Then $\tau$ lies in $N_G(A_1D_3)$, acting as a graph automorphism on $D_3$. Now, $A_1D_3$ acts on $V$ as $(\lambda_1\otimes \lambda_2) \oplus (\lambda_1\otimes \lambda_3)$ and therefore $\tau$ swaps these two $8$-spaces. Similarly, a $2$-cycle $\tau\in (S.Sym(3))/S$ corresponds to an element in the $D_3$ component of $A_1D_3$ acting as a graph automorphism of $A_1^2\leq D_3$. Here $\tau$ does fix $(\lambda_1\otimes \lambda_2)$ and $ (\lambda_1\otimes \lambda_3)$, although they are not isomorphic $S\langle \tau \rangle$-modules. A $3$-cycle $\tau\in (S.Sym(3))/S$ also acts non-trivially on $Y$, since it is a product of two $2$-cycles that do not have the same fixed points. Finally, no product of a reflection with a transposition can act trivially, again because they do not fix the same points. Now by Lemma~\ref{cyclic extension V_1+V_2 lemma} there is a dense subset $\hat{Y}$ of $Y$ on which no element of $N_G(S)/S$ has fixed points. Furthermore, there are only three proper connected subgroups of $G$ that properly contain $S$, all isomorphic to $A_1D_3$ acting on $V$ as $(\lambda_1\otimes \lambda_2) \oplus (\lambda_1\otimes \lambda_3)$. Any such $A_1D_3$ only fixes two $8$-spaces, which are not contained in $Y$ since they are the fixed points of the $2$-cycles in $Sym(3)$. Let $y\in \hat{Y}$.
    We have shown that $G_y = S$ and therefore $\mathrm{Tran}_G(y,Y) = N_G(S)$. We then get $\dim G - \dim \mathrm{Tran}_G(y,Y) = 27 = \dim\mathcal{S}''_8(V) -\dim Y$. Thus, the set $\hat{Y}$ is $Y$-exact, and by Lemma~\ref{loc to a subvariety lemma} we conclude that $C_{\mathcal{S}''_8(V)}=A_1^3$.

\end{proof}

\begin{proposition}\label{proposition b4 P_8 p 2}
Let $G = B_4$, $\lambda = \lambda_4$ with $p=2$. Then the quadruple $(G,\lambda,p,8'')$ has no generic $ts$-stabilizer, but has a semi-generic $ts$-stabilizer $A_1^3$.
\end{proposition}

\begin{proof}
    Given the standard parabolic $P_1=U_7B_3T_1 = U_7L$, let $X\leq L'$ be a subgroup isomorphic to an $A_1^3$ acting as $2\perp 2\perp 2$ on $V_{B_3}(\lambda_1)$. Here $L$ acts on the abelian unipotent radical $U_7$ by fixing the longest short-root subgroup $X_{1111}$ and as $V_{B_3}(\lambda_1)$ on $U_7/X_{1111}$. Then  $X$ has a $3$-dimensional $1$-cohomology on $U_7$, corresponding to the conjugacy classes of $A_1^3$-subgroups of $U_7X$. We can parametrise this by pairing the root subgroups generating $X$ with the highest and lowest weight vectors for the action on $U_7/X_{1111}$. More precisely, take \[X = \langle X_{\pm 0111},X_{\pm 0011},X_{\pm 0001}\rangle\] and define \begin{flalign*}
     A_1^{(1)}(\lambda)&:=\langle x_{0111}(t)x_{1222}(\lambda t),x_{-0111}(t)x_{1000}(\lambda t) \rangle_{t\in K},&\\
A_1^{(2)}(\lambda)&:=\langle x_{0011}(t)x_{1122}(\lambda t),x_{-0011}(t)x_{1100}(\lambda t) \rangle_{t\in K},&\\
 A_1^{(3)}(\lambda)&:=\langle x_{0001}(t)x_{1112}(\lambda t),x_{-0001}(t)x_{1110}(\lambda t) \rangle_{t\in K},&\\
    X_{abc} &:= \langle A_1^{(1)}(a),A_1^{(2)}(b),A_1^{(3)}(c) \rangle.&
 \end{flalign*}
 Each $A_1^{(i)}(\lambda)$ is a connected subgroup of $P_1$ of type $A_1$. Furthermore, $A_1^{(i)}(\lambda)$ and $A_1^{(j)}(\mu)$ commute if $i\neq j$, which means that $X_{abc}$ is isomorphic to $A_1^3$.
 We can write the given generators for $A_1^{(i)}(\lambda)$ in a nice compact form in the Clifford algebra, namely \[ A_1^{(i-1)}(\lambda) = \langle (1+t e_ie_5) (1+t e_if_5)(1+\lambda t e_ie_1),(1+t f_ie_5) (1+t f_if_5)(1+\lambda t f_ie_1)\rangle_{t\in K},\] where $i\in \{2,3,4\}$.
Then $\mathcal{C}:=\{X_{abc}\}_{a,b,c\in K}$ is a set of representatives for the conjugacy classes of $A_1^3$-subgroups of $U_7X$. 

Now consider an arbitrary $X_{abc}$. We proceed to show that $X_{abc}$ acts homogeneously on $V$ as a sum of two irreducible $8$-spaces. Since $P_1$ fixes the (totally singular) $8$-space  
\[V_1:= \left\langle \begin{array}{lll}%
e_1e_2, &e_1e_3, & e_1e_4,\\ e_1e_5,& e_1e_2e_3e_5,&e_1e_2e_4e_5,\\e_1e_3e_4e_5,&e_1e_2e_3e_4 
\end{array}\right\rangle,\]
so does $X_{abc}$. Secondly, let \[W_{abc}:= \left\langle \begin{array}{lll}%
e_1e_3e_4e_5+e_3e_4,& (a+1)e_1e_2e_3e_4+e_2e_3e_4e_5,&e_2e_3+(a+b+1)e_1e_2e_3e_5,\\e_3e_5+(1+b)e_1e_3,& (1+a+c)e_1e_2e_4e_5+e_2e_4,&(c+1)e_1e_4+e_4e_5,\\1+(b+c+1)e_1e_5, &e_2e_5+(1+a+b+c)e_1e_2
\end{array}\right\rangle.\]
 A simple check using the generators of $X_{abc}$ shows that $W_{abc}$ is fixed by $X_{abc}$ and it is isomorphic to $\lambda_1\otimes\lambda_1\otimes\lambda_1$ as an $A_1^3$-module. Since $p=2$, it is not guaranteed that $W_{abc}$ is totally singular. Indeed, $W_{abc}$ is totally singular if and only if $a+b+c=0$, by a direct check using Lemma~\ref{quadratic form spin description}. We now consider the subset $\mathcal{C}^*$ of $\mathcal{C}$ given by triples $(a,b,c)$ with $a+b+c=0$ such that $a,b,c$ are all distinct. Under these conditions on $(a,b,c)$, we know that an element of $\mathcal{C}^*$ acts homogeneously on $V$ as a sum of two totally singular $8$-spaces, and acts indecomposably on $V_{B_4}(\lambda_1)$ as $1/(2\perp 2\perp 2)/1$. 
 
 We will now show that $X_{abc}\in \mathcal{C}^*$ is the connected component of the stabilizer in $G$ of $W_{abc}$. Since $X$ acts indecomposably as $(2\perp 2 \perp 2) /1$ on $U_7$, so does $X_{abc}$. In particular recall that $X_{1111}$ is fixed by $X_{abc}$. Therefore if $u$ is a non-trivial element in $U_7$, we must have $X_{1111} \cap \langle u,X_{abc} \rangle\neq 1$. A direct check shows that no non-trivial element of $X_{1111}$ stabilises $W_{abc}$, implying $(U_7)_{W_{abc}} = 1$.
 Let $M$ be a minimal connected overgroup of $X_{abc}$, such that $M\leq G_{W_{abc}}$. If $M$ has a larger projection $\overline{M}$ onto $L'$ than $X_{abc}$, it means that either $\overline{M}=A_1B_2$ or $\overline{M}=L'$. In the latter case $M=U_7L'$, which is absurd, therefore assume that $\overline{M}=A_1B_2$. Without loss of generality take $\overline{M} = \langle X_{\pm 0111},X_{\pm 0010},X_{\pm 0001}\rangle.$ Let $u\in U_7$ and $\kappa\in K^*$ such that $ux_{0010}(1)h_{\alpha_1}(\kappa)\in M$. Since $M\cap U_7 = 1$, we must have $[u,X_{0010}] =[u,h_{\alpha_1}(\kappa)] = 1$. We cannot have $u\in X_{1222}$, since a direct check shows that $ux_{0010}(1)h_{\alpha_1}(\kappa)$ does not fix $W_{abc}$; so we must have $\kappa = 1$ and $u\in \langle  X_{1000},X_{1110},X_{1111},X_{1122},X_{1222}\rangle$. 
 Similarly, since \[[X_{0010},A_1^{(3)}(c)] =[X_{0010},A_1^{(1)}(a)]=1,\] we also get $[u,X_{\pm 0011}] =[u,X_{\pm 0111}] =1$. Thus, $u\in \langle X_{1110},X_{1111}\rangle $. Now assume that $x:=x_{1110}(t_1)x_{1111}(t_2)x_{0010}(1)$ fixes $W_{abc}$. Since $x.(e_1e_3e_4e_5+e_3e_4)\in W_{abc}$ we find that $t_1=t_2=0$. Therefore $x = x_{0010}(1)$. Since $x.(e_4e_5+(c+1)e_1e_4)\in W_{abc}$ we find that $b=c$, which is absurd by our choice of $(a,b,c)$. This completes the proof that $(G_{W_{abc}})^0=X_{abc}$ when $a,b,c$ are all distinct.

 The radical of the Levi $B_3T_1$ acts by scalar multiplication on $(a,b,c)$. Therefore the subset $\mathcal{C}^{**}$ of $\mathcal{C}^*$ defined by the further condition $a=1$, contains $A_1^3$-subgroups which are pairwise non-conjugate in $U_7XT_1$. Let $Y = \{W_{abc}:X_{abc}\in \mathcal{C}^{**} \}$, a $1$-dimensional variety of totally singular $8$-spaces. By \cite[Prop.~3.5.2~(D)]{stewart2013reductive} $G$-fusion of elements of $\mathcal{C}^{**}$ is controlled by $N_{L'}(X)/X\simeq Sym(3)$. We can be even more precise, and like in \cite[Lemma~4.1.3]{stewart2013reductive} deduce that if $X_{abc}^g = X_{a'b'c'}$ then $g\in U_7N_{B_3T_1}(X)$. Therefore $X_{abc}$ is $G$-conjugate to $X_{a'b'c'}$ if and only if $(a',b',c') = t(\pi (a),\pi(b),\pi(c))$ for some $t\in K$ and $\pi\in Sym(\{a,b,c\})$. Also, if $a,b,c$ are pairwise distinct and are not of the form $a,\mu a,\mu^2 a$ where $\mu$ is a root of $x^2+x+1$, we must have $N_G(X_{abc}) = U_1X_{abc}$. 
 
 Therefore there is a dense subset $\hat{Y}$ of $Y$ such that any two distinct elements in $\hat{Y}$ have non-conjugate stabilizers in $G$, isomorphic to $A_1^3$. Let $y\in\hat{Y}$. Then by construction $A_1^3\leq \mathrm{Tran_G}(y,Y) \leq A_1^3.Sym(3)$, and by dimensional considerations $\hat{Y}$ is $Y$-exact. By Lemma~\ref{no generic stabilizer lemma} we conclude that there is no generic stabilizer and by Lemma~\ref{loc to a subvariety open set lemma} we conclude that there is a semi-generic stabilizer isomorphic to $A_1^3$.
\end{proof}

There are now two cases left in order to complete the proof of Theorem~\ref{complete list of cases theorem}. These are given by the $ts$-small quadruples $(C_2,2\lambda_1,p,5)$ ($p\neq 2$) and $(C_3,\lambda_2,p,7)$ ($p\neq 3$). These two cases present considerable challenges and similarities to each other. They are the subject of the next two sections.
\subsection{The case $(C_2, 2\lambda_1,p,5)$}
In this section we handle the case of $C_2$ acting on maximal totally singular subspaces of its adjoint module. We shall prove that this action has a dense orbit, with finite generic stabilizer. We resort to making extensive use of computational methods in Magma, with the relevant code being listed in Appendix~\ref{magma code appendix} as well as being made available on the author's GitHub \cite{Aluna_Magma_code_used}.

Suppose that $p\neq 2$.
    Let $G=Sp_4(K)$, with fundamental roots $\alpha_1,\alpha_2$, where $\alpha_1$ is short. Let $\alpha_3 = \alpha_1+\alpha_2$ and $\alpha_4 = 2\alpha_1+\alpha_2$. Order the standard basis of the natural module $V_{nat}$ as $(e_1,e_2,f_2,f_1)$ and let $V = \mathrm{Lie}(G)\leq \mathfrak{sl}_4(K)$, on which $G$ is acting by conjugation. Let $e_{\pm\alpha_1},e_{\pm \alpha_2},e_{\pm\alpha_3},e_{\pm\alpha_4},h_{\alpha_1},h_{\alpha_2}$ be the corresponding Chevalley basis, where $e_{\alpha_1},e_{ \alpha_2},e_{\alpha_3},e_{\alpha_4}$ are respectively the matrices
    \[ \left(\begin{matrix} 
&1 &  & \\
& &0 & \\
& & & -1\\
& & & \\
\end{matrix}\right),\, \left(\begin{matrix} 
&0 &  & \\
& &1 & \\
& & & 0\\
& & & \\
\end{matrix}\right),\,\left(\begin{matrix} 
& & 1 & \\
& & &1 \\
& & & \\
& & & \\
\end{matrix}\right),\,\left(\begin{matrix} 
& &  & 1\\
& & & \\
& & & \\
& & & \\
\end{matrix}\right),\]
the elements $e_{-\alpha_1},e_{-\alpha_2},e_{-\alpha_3},e_{-\alpha_4}$ are respectively their transposes, and \[h_{\alpha_1} = \diag(1,-1,1,-1)\, ,h_{\alpha_2} = \diag(0,1,-1,0), h_{\alpha_4} = \diag(1,0,0,-1).\]
    Let $T$ be the standard maximal torus of $G$. The module $V$ is orthogonal, with quadratic form given by \[Q(v)=\mathrm{Trace}(v^2).\]
For $i$ a square root of $-1$, and $\zeta$ a square root of $-2$, let $W_{(i,\zeta)}$ be the totally singular $5$-space of $V$ spanned by 
\begin{align*}
v^{(0)} &= h_{\alpha_4}+ih_{\alpha_2}, \\
v^{(1)} &=e_{\alpha_1}+\zeta e_{\alpha_2}, \\
v^{(2)} &= e_{\alpha_3}+\zeta e_{-\alpha_4}, \\
v^{(3)} &= e_{-\alpha_3}+\zeta e_{\alpha_4}, \\
v^{(4)} &= e_{-\alpha_1}+\zeta e_{-\alpha_2}. 
 \end{align*}
Let $W_{5}$ be the $5$-space of $V$ spanned by
\begin{align*}
u^{(0)} &= h_{\alpha_4}+2h_{\alpha_2}, \\
u^{(1)} &=e_{\alpha_1}+3 e_{\alpha_2}, \\
u^{(2)} &= e_{\alpha_2}+3 e_{-\alpha_4}, \\
u^{(3)} &= e_{-\alpha_3}, \\
u^{(4)} &= e_{-\alpha_1}+e_{-\alpha_2}+3e_{\alpha_4}, 
 \end{align*}
 a totally singular subspace if $p=5$.
Let
\[\tau = \left(\begin{matrix} 
0& 1& 0 & 0\\
 0& 0& 0& -1\\
 1 & 0& 0& 0\\
 0& 0 & 1 & 0\\
\end{matrix}\right), \text{ and }
    x= 
\begin{cases}
\diag (\omega^{-1},\omega^{-2},\omega^2,\omega) \text{ with } \omega^5 = 1,\omega\neq 1,& \text{if } p\neq 5;\\
\begin{pmatrix} 
1&2 & 1 & 1\\
 & 1&1 &4 \\
  & & 1&3 \\
 &  &  & 1\\
\end{pmatrix},& \text{if } p = 5.
\end{cases}
\]
Furthermore, let 
\[   \tau^*= 
\begin{cases}
    \tau,& \text{if } p\neq 5;\\
    \diag (\alpha,2\alpha,-\alpha,3\alpha),\text{ where } \alpha^2 = 2, & \text{if } p = 5.
\end{cases}
\]
Finally, let  
\[
S^* = \langle x,\tau^* \rangle, \text{ and }
W^*= 
\begin{cases}
W_{(i,\zeta)},& \text{if } p\neq 5;\\
W_5,& \text{if } p = 5.
\end{cases}
\]
With this setup, it is easy to check that $S^* \leq G_{W^*}$.

\begin{proposition}\label{C2 prop p=5 k=5}
    Let $G=C_2$, $\lambda = 2\lambda_1$ with $p = 5$. Then $C_{\mathcal{S}_5'(V)}=\mathbb{Z}_4$.
\end{proposition}

\begin{proof}
    We use the setup of \cite[Lemma~4.3.1(i)]{generic} and its proof.
    Let $h_0 = \diag (-1,-2,2,1)$, a regular semisimple element of $\mathrm{Lie}(T)$, and set $\mathfrak{G} = \langle h_0 \rangle$. For a subspace $U$ of $\mathrm{Lie}(G)$, write $\mathrm{Ann}_{\mathrm{Lie}(G)}(U)$ for the subspace $\{v\in \mathrm{Lie}(G):[v,U]\leq U\}.$
    A straightforward calculation shows that $\mathrm{Ann}_{\mathrm{Lie}(G)}(W_{(3,\zeta)}) = \mathfrak{G}$. Let $S = G_{W_{(3,\zeta)}}$ and take $g\in S$. We have $\mathfrak{G} = \mathrm{Ann}_{\mathrm{Lie}(G)}(g.W_{(3,\zeta)})=g.\mathfrak{G}$. Therefore $g.\mathfrak{G} = \mathfrak{G}$, which is easily seen to imply $g\in  T.\langle \tau \rangle.$ A direct calculation shows that $T\cap S = \pm 1$, which implies that $S = \langle \tau \rangle = Z(G).\mathbb{Z}_4$. Since $\dim G - \dim S = \dim G = \dim \mathcal{S}_5'(V)$, we conclude that $C_{\mathcal{S}_5'(V)}=\mathbb{Z}_4$.
\end{proof}

\begin{remark}
    Note that the subspace $W_{2,\zeta}$ does not belong to the same $D_5$-orbit as $W_{3,\zeta}$ by Lemma~\ref{lemma interesction maximal totally singular}, however it also does not have a finite stabilizer. Indeed it is not difficult to see that it has a stabilizer isomorphic to $U_3T_2$.
\end{remark}

The following lemmas describe the subgroup structure of $Sp_4(q)$.

\begin{lemma}\cite[§8.2]{MaximalSubgroupsColva}\label{lemma subgroup structure of sp4(q)}
Assume that $p<\infty$ and let $q = p^e$ for some $e\in\mathbb{Z}_{\geq 1}$. Then the maximal subgroups of $Sp_4(q)$ ($q$ odd) are as in Table~\ref{tab:maximal subgroups of sp4(q)}, and the maximal subgroups of $SL_2(q)$ are as in Table~\ref{tab:maximal subgroups of sl2(q)}. In both cases see \cite{MaximalSubgroupsColva} for more details, including the precise notation.
\end{lemma}

\begin{center}
\begin{longtable}{l l l l}
\caption{Maximal subgroups of $Sp_4(q)$ ($q$ odd)} \label{tab:maximal subgroups of sp4(q)} \\

\hline  \multicolumn{1}{l}{Class} & \multicolumn{1}{l}{\textbf{$M_q$}}& \multicolumn{1}{l}{Notes} & \multicolumn{1}{l}{\# conjugacy classes} \\ \hline 
\endhead
$\mathscr{C}_1$ & $q^{1+2}.((q-1)\times Sp_2(q))$ &  & $1$ \\
$\mathscr{C}_1$ & $q^3.GL_2(q)$ &  & $1$ \\
$\mathscr{C}_2$ & $Sp_2(q)^2.\mathbb{Z}_2$ &  & $1$ \\
$\mathscr{C}_2$ & $GL_2(q).\mathbb{Z}_2$ & $q\geq 5$ & $1$ \\
$\mathscr{C}_3$ & $Sp_2(q^2).\mathbb{Z}_2$ &  & $1$ \\
$\mathscr{C}_3$ & $GU_2(q).\mathbb{Z}_2$ & $q\geq 5$ & $1$ \\
$\mathscr{C}_5$ & $Sp_4(q_0).(2,r)$ & $q=q_0^r$, $r$ prime & $(2,r)$ \\
$\mathscr{C}_6$ & $\mathbb{Z}_2.\mathbb{Z}_2^4.Sym(5)$ & $q=p\equiv \pm 1 \mod 8$ & 2 \\
$\mathscr{C}_6$ & $\mathbb{Z}_2.\mathbb{Z}_2^4.Alt(5)$ & $q=p\equiv \pm 3 \mod 8$ & 1 \\
$\mathscr{S}$ & $\mathbb{Z}_2.Alt(6)$ & $q=p\equiv \pm 5 \mod 12$, $q\neq 7$ & 1 \\
$\mathscr{S}$ & $\mathbb{Z}_2.Sym(6)$ & $q=p\equiv \pm 1 \mod 12$ & 2 \\
$\mathscr{S}$ & $\mathbb{Z}_2.Alt(7)$ & $q=7$ & 1 \\
$\mathscr{S}$ & $SL_2(q)$ & $p\geq 5$, $q\geq 7$ & 1 \\
\hline
\end{longtable}
\end{center}

\begin{center}
\begin{longtable}{l l l l}
\caption{Maximal subgroups of $SL_2(q)$} \label{tab:maximal subgroups of sl2(q)} \\

\hline \multicolumn{1}{l}{Class} & \multicolumn{1}{l}{\textbf{$M_q$}}& \multicolumn{1}{l}{Notes} & \multicolumn{1}{l}{\# conjugacy classes} \\ \hline 
\endhead
$\mathscr{C}_1$ & $q.(q-1)$ &  & $1$ \\
$\mathscr{C}_2$ & $Q_{2(q-1)}$ & $q\neq 5$; $q$ odd  & $1$ \\
$\mathscr{C}_2$ & $D_{2(q-1)}$ & $q$ even  & $1$ \\
$\mathscr{C}_3$ & $Q_{2(q+1)}$ & $q$ odd  & $1$ \\
$\mathscr{C}_3$ & $D_{2(q+1)}$ & $q$ even  & $1$ \\
$\mathscr{C}_5$ & $SL_2(q_0).(2,r)$ & $q=q_0^r$, $r$ prime, $q$ odd & $(2,r)$ \\
$\mathscr{C}_5$ & $PSL_2(q_0)$ & $q=q_0^r$, $r$ prime, $q_0\neq 2$, $q$ even & $1$ \\
$\mathscr{C}_6$ & $\mathbb{Z}_2.\mathbb{Z}_2^2.Sym(3)$ & $q=p\equiv \pm 1 \mod 8$ & 2 \\
$\mathscr{C}_6$ & $\mathbb{Z}_2.\mathbb{Z}_2^2.\mathbb{Z}_3$ & $q=p\equiv \pm 3,5,\pm 11,\pm 13,\pm 19 \mod 40$ & 1 \\
$\mathscr{S}$ & $SL(2,5)$ & $q=p\equiv \pm 1 \mod 10$ & 2 \\
$\mathscr{S}$ & $SL(2,5)$ & $q=p^2, p\equiv \pm 3 \mod 10$ & 2 \\
\hline
\end{longtable}
\end{center}

\begin{lemma}\label{lemma c2 w5 p5 U2T1 stab}
    Assume $p=5$. Let $H = G_{\langle u^{(1)} \rangle}$. Then $H_{W^*} = S^* \simeq Z(G).\mathbb{Z}_5.\mathbb{Z}_4$.
\end{lemma}

\begin{proof}
    The element $u^{(1)}$ is a regular nilpotent element, and a simple calculation shows that $H = U_2T_1$ where 
    \[U_2 = \left\{\left(\begin{matrix} 
1&-b &-b^2  & a \\
& 1& 2b  &b^2 \\
&  & 1&b \\
& & & 1\\
\end{matrix}\right): a,b\in K\right\},\,T_1 = \langle \diag(\kappa^3,\kappa,\kappa^{-1},\kappa^{-3}):\kappa \in K^*\rangle\] Let $g =\left(\begin{smallmatrix} 
1&-b &-b^2  & a \\
& 1& 2b  &b^2 \\
&  & 1&b \\
& & & 1\\
\end{smallmatrix}\right)\left(\begin{smallmatrix} 
\kappa^3&&  &  \\
& \kappa&   & \\
&  &\kappa^{-1} & \\
& & & \kappa^{-3}\\
\end{smallmatrix}\right)\in H_{W^*}$. We have $g u^{(0)} g^{-1} = u^{(0)}-b u^{(1)} +(b^3+3a)e_{\alpha_4}$. This forces $a = 3b^3$. Also, $gu^{(3)}g^{-1} = \frac{1}{\kappa^4}(u^{(3)}-b^2 u^{(0)}+2b^3 u^{(1)}+2b u^{(4)}+(b^5-b)e_{\alpha_4
})$, forcing $b=b^5$. Finally $gu^{(2)}g^{-1}=  \frac{1}{\kappa^6} (-b^3 u^{(0)}+b(1-\kappa^8)e_{\alpha_3} +u^{(2)}+3b u^{(3)}+3b^2 u^{(4)}-b^4 u^{(1)}+(\kappa^8-1)e_{\alpha_2}+b^2(\kappa^8-1)e_{\alpha_4})$, implying $\kappa^8 = 1$. This allows us to conclude that $H_{W^*} = \langle x , \tau^* \rangle =S^*$, as claimed.
\end{proof}

\begin{lemma}\label{lemma c2 normalizer stab}
    Let $H = N_G(\langle x \rangle)$. Then $H_{W^*}= S^* \simeq Z(G).\mathbb{Z}_5.\mathbb{Z}_4$.
\end{lemma}
\begin{proof}
    Assume $p\neq 5$. Since $x$ is a regular semisimple element, it is easy to see that $H= T.\langle \tau \rangle$, and one quickly finds that $T_{W^*} = \pm\langle x\rangle$. Since $\tau\in G_{W^*}$, we conclude that $H_{W^*} = \langle x,\tau \rangle = Z(G).\langle x \rangle.\mathbb{Z}_4$. If $p=5$, we have $H\leq \langle C_G(x),T \rangle$. It is easy to see that $C_G(x)$ is the unipotent radical of $G_{\langle u^{(1)}\rangle}$ and that \[N_T(\langle x \rangle) = \langle \tau^* \rangle \leq \diag(\kappa^3,\kappa,\kappa^{-1},\kappa^{-3}):\kappa \in K^*\rangle.\] Therefore $H\leq G_{\langle u^{(1)} \rangle}$ and we can conclude by Lemma~\ref{lemma c2 w5 p5 U2T1 stab}.
\end{proof}

\begin{lemma}\label{lemma c2 reducible k=5}
    Suppose that $H\leq G$ is a reducible subgroup of $G$ containing $\langle x \rangle$. Then $H_{W^*}\leq  S^*$. 
\end{lemma}

\begin{proof}
    If $p\neq 5$, the only $1$-spaces of $V_{nat}$ stabilised by the semisimple element $x$ are spanned by a standard basis vector, therefore $H\leq G_U$ where $U\leq V_{nat}$ is either a $1$-space or a totally singular $2$-space or a non-degenerate $2$-space, spanned by standard basis vectors. On the other hand, if $p=5$, then $H\leq G_U$ where $U = \langle e_1\rangle$ or $U=\langle e_1,e_2\rangle$, as $x$ is a regular unipotent element. In all the cases where $U$ is totally singular, i.e. $G_U\simeq P_1$ or $G_U\simeq P_2$, it is easily seen that $G_U$ stabilises a unique $6$-space of $V$. We intersect such $6$-space with $W^*$, identifying a $1$-space spanned by a regular nilpotent element $v$ that must be stabilised by $H_{W^*}$. This then reduces the problem to computing the stabilizer of $W^*$ within a $U_2T_1$. If $p=5$, we find that $\langle v \rangle = \langle u^{(1)} \rangle$, concluding by Lemma~\ref{lemma c2 w5 p5 U2T1 stab}. 
    
    Therefore from now on assume that $p\neq 5$. If $H = G_{\langle e_1 \rangle}$, then $G$ stabilises $\langle e_{\pm\alpha_2},e_{\alpha_1},e_{\alpha_3},e_{\alpha_4
    },h_{\alpha_2}\rangle$, which intersects $W^*$ in $\langle v^{(1)} \rangle$. We find that $G_{\langle v^{(1)} \rangle} = U_2T_1$ where
    \[U_2 = \left\{\left(\begin{matrix} 
1&b &\frac{b^2\zeta}{2}  & a \\
& 1& \zeta b  &-\frac{b^2\zeta}{2} \\
&  & 1&-b \\
& & & 1\\
\end{matrix}\right): a,b\in K\right\},\,T_1 = \langle \diag(\kappa^3,\kappa,\kappa^{-1},\kappa^{-3}):\kappa \in K^*\rangle.\] A direct calculation like in the proof of Lemma~\ref{lemma c2 w5 p5 U2T1 stab} then shows that $(U_2T_1)_{W^*}= \pm \langle x \rangle$. The other cases with $G_U\simeq P_1$ or $G_U\simeq P_2$ are similarly dealt with. Here we report just the intersection $\langle v \rangle$. If $U = \langle e_2 \rangle$, then $\langle v \rangle = \langle v^{(2)}\rangle$; if $U = \langle f_1 \rangle$, then $\langle v \rangle = \langle v^{(4)}\rangle$; if $U = \langle f_2 \rangle$, then $\langle v \rangle = \langle v^{(3)}\rangle$; if $U = \langle e_1,e_2 \rangle$, then $\langle v \rangle = \langle v^{(1)}\rangle$; if $U = \langle e_1,f_2 \rangle$, then $\langle v \rangle = \langle v^{(3)}\rangle$; if $U = \langle e_2,f_1 \rangle$, then $\langle v \rangle = \langle v^{(4)}\rangle$.

It remains to consider the case $H\leq G_{\langle e_1,f_1 \rangle}$. Here $G$ fixes the subspace $\langle h_{\alpha_1},h_{\alpha_2},e_{\pm \alpha_2},e_{\pm\alpha_4
} \rangle$, which intersects $W^*$ in $\langle v^{(0)}\rangle$. Therefore $H_{W^*}\leq N_G(T)$, and we conclude by Lemma~\ref{lemma c2 normalizer stab}.
\end{proof}

\begin{lemma}\label{lemma c2 case dot 2 reduction}
        Suppose that $\langle x \rangle \leq H.\mathbb{Z}_2 < G$, where $H$ is an arbitrary subgroup of $G$. Then $(H.\mathbb{Z}_2)_{W^*}\leq  S^*$ if and only if  $H_{W^*}\leq  S^*$.
\end{lemma}
\begin{proof}
    The forward direction is trivial. Suppose that $H_{W^*}\leq  \langle x , \tau^* \rangle = S^*$. Since $x$ has order $5$, we must have $\langle x\rangle\leq H$.  Since $H_{W^*}\leq  S^*$, the subgroup $\langle x \rangle$ is the unique subgroup of order $5$ in $H_{W^*}$. Therefore, since $H_{W^*} \lhd (H.\mathbb{Z}_2)_{W^*}$, the subgroup $\langle x \rangle$ is normal in $(H.\mathbb{Z}_2)_{W^*}$. By Lemma~\ref{lemma c2 normalizer stab} we know that the stabilizer of $W^*$ in $N_G(\langle x\rangle)$ is $S^*$, concluding.
\end{proof}

\begin{proposition}\label{C_2 k=5 prop}
    Let $G=C_2$, $\lambda = 2\lambda_1$ with $p\neq 2$. Then $C_{\mathcal{S}_5'(V)}=\mathbb{Z}_5.\mathbb{Z}_4$ and $C_{\mathcal{S}_5''(V)} = \mathbb{Z}_{5/(p,5)}.\mathbb{Z}_4$.
\end{proposition}

\begin{proof}
Let $S = G_{W^*}$. We shall prove that $S\leq N_G(\langle x \rangle)$. This will conclude the proof of the proposition as follows. By Lemma~\ref{lemma c2 normalizer stab} we have $(N_G(\langle x \rangle)_{W^*} = S^*$, which then implies $S=S^*$ and that $W^*$ is in a dense $G$-orbit on one of the two $D_5$-orbits on $\mathcal{S}_5(V)$. If $p\neq 5$, then $W_{(i,\zeta)}$ and $W_{(-i,\zeta)}$ intersect in a $4$-dimensional subspace, and therefore by Lemma~\ref{lemma interesction maximal totally singular}, they belong to distinct $D_5$-orbits on $\mathcal{S}_5(V)$. They each have stabilizer $S^*$, concluding the $p\neq 5$ case. If $p=5$ the subspace $W^*$ intersects $W_{(3,\zeta)}$ trivially, which by Lemma~\ref{lemma interesction maximal totally singular} implies that $W^*$ and $W_{(3,\zeta)}$ belong to distinct $D_5$-orbits on $\mathcal{S}_5(V)$. Again $G_{W^*} = S^*$, concluding.

In order to prove that $S\leq N_G(\langle x \rangle)$,
we show that for all $p<\infty$ and $e\in \mathbb{Z}_{\geq 1}$, if \[\langle x \rangle \leq R \leq Sp_4(p^e)=Sp_4(q)<G, \text{ with }R\not\leq N_G(\langle x \rangle),\] then $R$ does not stabilise $W^*$. Note that this is indeed sufficient, since if $g\in S\setminus N_G(\langle x \rangle)$, then there must exist $e\in \mathbb{Z}_{\geq 1}$ such that $g \in Sp_4(p^e)$, with $R = \langle x,g\rangle$ satisfying the condition above. We shall make extensive use of maximal subgroups of $Sp_4(q)$, often combined with exhaustive computations in Magma. The commented code is made available both in the Appendix as well as on the author's Github \cite{Aluna_Magma_code_used}. The $p=\infty$ case then follows from the $p<\infty$ case.

Suppose that \[\langle x \rangle \leq R \leq M_q < Sp_4(q)<G, \text{ with }R\not\leq N_G(\langle x \rangle),\] where $M_q$ is a maximal subgroup of $Sp_4(q)$, as listed in Table~\ref{tab:maximal subgroups of sp4(q)}. The goal is to prove that $R$ does not stabilize $W$.

If $M_q$ is as in one of the first $6$ rows of Table~\ref{tab:maximal subgroups of sp4(q)}, then by Lemma~\ref{lemma c2 case dot 2 reduction} we can assume that $R$ is reducible, and Lemma~\ref{lemma c2 reducible k=5} implies that $R$ does not stabilize $W$. If $M_q = Sp_4(q_0).(2,r)$ where $r$ is prime and $q = q_0^r$, then Lemma~\ref{lemma c2 case dot 2 reduction} allows us to reduce to one of the other cases.

Suppose that $M_q$ is the double cover of $Alt(6)$, $Sym(6)$ or $Alt(7)$, in which case $q = p$.
An exhaustive search using Magma shows that $Z(G).\langle x \rangle \leq R^* \leq R$ where $R^* = Z(G).Alt(5)$, the double cover of $Alt(5)$, isomorphic to $SL(2,5)$. The general strategy adopted for this exhaustive search is the following. We set up $R$ as an abstract group. Then for all conjugacy classes of elements of order $5$ of $R$, we take a representative $g_5$ and go through all subgroups of $R$ that contain $g_5$, determining which ones do not normalise $\langle g_5\rangle$.
By Lemma~\ref{lemma c2 reducible k=5} we can assume that $R^*$ is irreducible in $G$. If $p\neq 5$, since $p\neq 2$ by assumption and $p\neq 3$ by choice of $M_q$, we can use ordinary character theory to show that $R^*$ does not fix any $5$-space of $V$. The subgroup $R^*$ must be embedded in $G$ via its unique irreducible symplectic character $\chi$ of degree $4$. We then verify that $S^2(\chi) = \chi_4+\chi_5+\chi_7$, where $\chi_4,\chi_5,\chi_7$ are irreducible characters of degrees $3,3,4$. Therefore $R^*$ fixes no $5$-space of $V$. See Listing~\ref{verb1} for the corresponding Magma code. If $p = 5$, we can use a direct construction of $2.Alt(6)\leq Sp_4(5)$ to check that $R^*$ acts on $V$ with composition factors of dimensions $3,3,3,1$. Again, this means that $R$ fixes no $5$-space of $V$. See Listing~\ref{verb3} for the Magma code proving this.

Suppose that $M_q = Z(G).\mathbb{Z}_2^4.Sym(5)$, the normalizer of an extraspecial subgroup of $G$ of minus type. Similarly to the previous case, an exhaustive search shows that $R$ must contain $R^* = Z(G).\mathbb{Z}_2^4.\langle x \rangle$ or $R^* = Z(G).Alt(5)$. By Lemma~\ref{lemma c2 reducible k=5} we can assume that $R^*$ is irreducible in $G$. In the second case we have already seen that $R^*$ does not fix a $5$-space of $V$ when $p\neq 3$. If $p=3$ the same holds, which can be checked directly in $Sp_4(3)$ by taking an explicit construction of $Z(G).Alt(5)$. Therefore assume that $R^* = Z(G).\mathbb{Z}_2^4.\langle x \rangle$, a group with GAP Id $(160,199)$. If $p\neq 5$, we can use the ordinary characters of $R^*$ to check that $R^*$ does not fix any totally singular $5$-spaces of $V$. The subgroup $R^*$ must be embedded in $G$ via its unique irreducible symplectic character $\chi$ of degree $4$. One then finds that $S^2(\chi) = \psi_1+\psi_2$, where $\psi_1$ and $\psi_2$ are distinct self-dual irreducible characters of degree $5$. Therefore $R^*$ stabilises exactly two non-degenerate $5$-spaces of $V$. See Listing~\ref{verb2} for the Magma code. If $p=5$, we can use a direct construction of $\mathbb{Z}_2.\mathbb{Z}_2^4.Alt(5)\leq Sp_4(5)$ that the $KR^*$-module $V\downarrow R^*$ has two self-dual non-isomorphic composition factors, implying that $R^*$ does not stabilise a totally singular $5$-space. See details of the computations in Listing~\ref{verb4}.

It remains to consider the case $M_q = SL_2(q)$ with $p\geq 5$ and $q\geq 7$, as in the last row of Table~\ref{tab:maximal subgroups of sp4(q)}. In order to handle this case, we consider the subgroup structure of the maximal subgroups of $SL_2(q)$, as classified in Table~\ref{tab:maximal subgroups of sl2(q)}. First note that $SL_2(q)$ does not fix a $5$-space of $V$, as if $p = 5$ it acts on $V$ with composition factors of dimensions $4,3,3$, while if $p>5$ it acts on $V$ with composition factors of dimension $7$ and $3$. Therefore assume that $R\leq M_q^* < SL_2(q) $ where $M_q^*$ is a maximal subgroup of $SL_2(q)$, as described by Table~\ref{tab:maximal subgroups of sl2(q)}. If $M_q^*$ is as in one of the first five rows of Table~\ref{tab:maximal subgroups of sl2(q)}, then by Lemma~\ref{lemma c2 case dot 2 reduction} we can assume that $R$ is reducible, and Lemma~\ref{lemma c2 reducible k=5} implies that $R$ does not fix $W^*$. If $M_q^* = SL_2(q_0).(2,r)$ where $r$ is prime and $q = q_0^r$, then Lemma~\ref{lemma c2 case dot 2 reduction} allows us to reduce to one of the other cases. Since $p\neq 2$, the case $M_q^* = PSL_2(q_0) $ is excluded, while since $x$ has order $5$, the cases $M_q^* = \mathbb{Z}_2.\mathbb{Z}_2^2.Sym(3)$ and $M_q^* = \mathbb{Z}_2.\mathbb{Z}_2^2.\mathbb{Z}_3$ are not possible.
The only other possibility is $M_q^* = SL(2,5)$ with $p\neq 5$, as per the last two rows of Table~\ref{tab:maximal subgroups of sl2(q)}. In this case we must have $R = M_q^* = SL(2,5)$, which we have already dealt with. This completes the case-by-case analysis.
\end{proof}

\subsection{The case $(C_3, \lambda_2,p,7)$}
In this section we handle the last remaining case needed to complete the proof of Theorem~\ref{complete list of cases theorem}.
In particular we shall prove that the $ts$-small quadruples $(C_3, \lambda_2,p,7')$ and $(C_3, \lambda_2,p,7'')$ have a finite generic stabilizer.
The strategy is entirely similar to the one used for the $(C_2, 2\lambda_1,p,5)$ case. We shall however make even greater use of computational methods, sometimes resorting to solving large systems of equations using Magma. Again, the code can be found in Appendix~\ref{magma code appendix}, as well as on the author's GitHub \cite{Aluna_Magma_code_used}. 

    Suppose that $p\neq 3$.
    Let $G=Sp_6(K)$ and order the standard basis of the natural module $V_{nat}$ as $(e_1,e_2,e_3,f_3,f_2,f_1)$. Let $V$ be the submodule of $\bigwedge^2 V_{nat}$
    defined by \[V=\langle e_i\wedge e_j,f_i\wedge f_j,e_i\wedge f_j, \sum \alpha_{i}e_i\wedge f_i : i\neq j, \sum \alpha_i=0 \rangle. \] Then $V=V_{G}(\lambda_2)$. Let $\omega$ be a primitive cube root of unity and let $(v_1,\dots ,v_{14})$ be the ordered basis of $V$ given by 
\begin{flalign*}
 v_1&= e_1\wedge e_2, & v_5&=e_2\wedge f_3, & v_{10}&=e_3\wedge f_2, &\\
 v_2&=e_1\wedge e_3, & v_6&=e_1 \wedge f_2, &v_{11}&=e_3\wedge f_1, &\\
 v_3&=e_2\wedge e_3, &v_7&=e_1\wedge f_1+\omega e_2\wedge f_2 +\omega^2 e_3\wedge f_3, & v_{12}&=f_2\wedge f_3,& \\
  v_4&=e_1\wedge f_3, & v_8&=e_1\wedge f_1+\omega^2 e_2\wedge f_2 +\omega e_3\wedge f_3, & v_{13}&=f_1\wedge f_3,& \\
  & & v_9&=e_2\wedge f_1, & v_{14}&=f_1\wedge f_2. \\
 \end{flalign*}

 Then it is easy to check that $G$ fixes a non-degenerate quadratic form on $V$, given by \[Q\left(\sum_1^{14}\alpha_i v_i\right)= \sum_1^7 \alpha_i \alpha_{15-i}.\] Let $T$ be the standard maximal torus of $G$. For $\omega$ a primitive cube root of unity, and $i$ a fourth root of unity, let $W_{(\omega,i)}$ be the totally singular $7$-space of $V$ spanned by 
\begin{align*}
v^{(0)} &= e_1\wedge f_1+\omega e_2\wedge f_2 +\omega^2 e_3\wedge f_3, \\
v^{(1)} &=e_2\wedge f_3+i e_1\wedge f_2, \\
v^{(2)} &= e_1\wedge f_3-i e_2\wedge e_3, \\
v^{(3)} &= f_1\wedge f_2+i e_1\wedge e_3, \\
v^{(4)} &= e_1\wedge e_2+i f_1\wedge f_3, \\
v^{(5)} &= e_3\wedge f_1-i f_2\wedge f_3, \\
v^{(6)} &= e_3\wedge f_2+i e_2\wedge f_1, \\
 \end{align*}
Let $W_{7}$ be the $7$-space of $V$ spanned by
\begin{align*}
u^{(0)} &= e_1\wedge f_1+4 e_2\wedge f_2 +2 e_3\wedge f_3, \\
u^{(1)} &=e_1\wedge e_2 +3 e_3\wedge f_1+3 f_2\wedge f_3, \\
u^{(2)} &= e_2\wedge e_3+4 f_1\wedge f_2, \\
u^{(3)} &= e_1\wedge f_3 +4 f_1\wedge f_2, \\
u^{(4)} &= e_2\wedge f_3+2 e_1\wedge f_2,\\
u^{(5)} &= e_2\wedge f_1 +5 e_3 \wedge f_2, \\
u^{(6)} &= f_1\wedge f_3, 
 \end{align*}
 a totally singular subspace if $p=7$. Let
\[\tau = \left(\begin{matrix} 
0& 0& 0 & 0 & -1 & 0\\
0& 0& 1 & 0 & 0 & 0\\
1& 0& 0 & 0 & 0 & 0\\
0& 0& 0 & 0 & 0 & 1\\
0& 0& 0 & 1 & 0 & 0\\
0& 1& 0 & 0 & 0 & 0\\
\end{matrix}\right), \text{ and }
    x= 
\begin{cases}
\diag (\omega^{-1},\omega^{-2},\omega^{-3},\omega^{3},\omega^2,\omega) \text{ with } \omega^7 = 1,\omega\neq 1,& \text{if } p\neq 7;\\
\begin{pmatrix} 
1& 1& 4 & 6 & 2 & 1\\
& 1& 1 & 4 & 1 & 5\\
& & 1 & 1 & 3 & 6\\
& &  & 1 & 6 & 4\\
& &  &  & 1 & 6\\
& &  &  &  & 1\\
\end{pmatrix},& \text{if } p = 7.
\end{cases}
\]
Furthermore, let 
\[   \tau^*= 
\begin{cases}
    \tau,& \text{if } p\neq 7;\\
    \diag (\alpha,5\alpha,4\alpha,6\alpha,2\alpha,3\alpha),\text{ where } \alpha^2 = 5, & \text{if } p = 7.
\end{cases}
\]
Finally, let  
\[
S^\dag = \langle x,(\tau^*)^4 \rangle,\,  S^* = \langle x,\tau^* \rangle, \text{ and }
W^*= 
\begin{cases}
W_{(\omega,i)},& \text{if } p\neq 7;\\
W_7,& \text{if } p = 7.
\end{cases}
\]
With this setup, it is easy to check that $S^\dag \leq S^* \leq G_{W^*}$.
\begin{proposition}\label{C3 p=7 k=7 prop}
    Let $G=C_3$, $\lambda = \lambda_2$ with $p = 7$. Then $C_{\mathcal{S}_7'(V)}=\mathbb{Z}_6$.
\end{proposition}
\begin{proof}
We use the setup of \cite[Lemma~4.3.1(i)]{generic} and its proof.
    Let $h_0 = \diag (-1,-2,-3,3,2,1)$, a regular semisimple element of $\mathrm{Lie}(T)$, and set $\mathfrak{G} = \langle h_0 \rangle$. For a subspace $U$ of $\mathrm{Lie}(G)$, write $\mathrm{Ann}_{\mathrm{Lie}(G)}(U)$ for the subspace $\{v\in \mathrm{Lie}(G):[v,U]\leq U\}.$
    A straightforward calculation shows that $\mathrm{Ann}_{\mathrm{Lie}(G)}(W_{(2,i)}) = \mathfrak{G}$. Let $S = G_{W_{(2,i)}}$ and take $g\in S$. We have $\mathfrak{G} = \mathrm{Ann}_{\mathrm{Lie}(G)}(g.W_{(2,i)})=g.\mathfrak{G}$. Therefore $g.\mathfrak{G} = \mathfrak{G}$, which is easily seen to imply $g\in  T.\langle \tau \rangle.$ A direct calculation shows that $T\cap S = \pm 1$, which implies that $S = \langle \tau \rangle = Z(G).\mathbb{Z}_6$. Since $\dim G - \dim S = \dim G = \dim \mathcal{S}_7'(V)$, we conclude that $C_{\mathcal{S}_7'(V)}=\mathbb{Z}_6$.
\end{proof}

\begin{lemma}\label{lemma c3 reducible k=7}
    Suppose that $H\leq G$ is a reducible subgroup of $G$ containing $S^\dag$. Then $H_{W^*}\leq  S^*$. 
\end{lemma}

\begin{proof}
Suppose that $p=7$. Since $x\in S^\dag$ is a regular unipotent element contained in the standard Borel subgroup $B$, we have that $H$ is contained in $G_{\langle e_1 \rangle}$, $G_{\langle e_1,e_2 \rangle}$ or $G_{\langle e_1,e_2,e_3 \rangle}$. We consider each of these cases and deduce that $H_{W^*}\leq B$. We then use Magma to directly show that $B_{W^*} = S^*$, concluding as required. Suppose that $H\leq G_{\langle e_1,e_2 \rangle}$. The group $G_{\langle e_1,e_2 \rangle}$ stabilises $U_6 = \langle e_1\wedge e_2,e_1\wedge e_3,e_2\wedge e_3,e_1\wedge f_3,e_2\wedge f_3, e_1\wedge f_1+e_2\wedge f_2 -2 e_3\wedge f_3\rangle $, and therefore $H_{W^*}$ must stabilise $W^* \cap U_6 = \langle e_1\wedge f_3-e_2\wedge e_3\rangle$ as well as $W^*\cap (U_6)^\perp = \langle e_1\wedge f_3-e_2\wedge e_3,e_2\wedge f_3 +2 e_1 \wedge f_2\rangle$. The latter implies that $H_{W^*}$ stabilises $\langle e_1,e_2,e_3,f_2,f_3 \rangle$ and therefore also its radical $\langle e_1 \rangle$. Let $g\in H_{W^*}$. Since $g$ stabilises $\langle e_1\rangle$, $\langle e_1,e_2\rangle$, $\langle e_1,e_2 \rangle ^\perp $ and $\langle e_1\wedge f_3-e_2\wedge e_3 \rangle$, it is easy to see that $g.e_3 \in \langle e_1,e_2,e_3 \rangle$. Therefore $g\in B$ and $H_{W^*}\leq B$. Now consider the case $H\leq G_{\langle e_1 \rangle}$. Similarly to the previous case, we find that $H_{W^*}$ stabilises $\langle e_1\wedge f_3-e_2\wedge e_3,e_2\wedge f_3 +2 e_1 \wedge f_2\rangle$. Let $g = (a_{ij})_{ij} \in H_{W^*}$. We have $g.(e_1\wedge f_3-e_2\wedge e_3) = a_{11}e_1\wedge (a_{14}e_1+a_{24}e_2+a_{34}e_3+a_{44}f_3+a_{54}f_2)-(a_{12}e_1+a_{22}e_2+a_{32}e_3+a_{42}f_3+a_{52}f_2)\wedge (a_{13}e_1+a_{23}e_2+a_{33}e_3+a_{43}f_3+a_{53}f_2)$. Since $g.(e_1\wedge f_3-e_2\wedge e_3) \in \langle e_1\wedge f_3-e_2\wedge e_3,e_2\wedge f_3 +2 e_1 \wedge f_2\rangle$, we must have \[\det \begin{pmatrix}
    a_{22} & a_{23} \\
    a_{52} & a_{53}
\end{pmatrix} = \det \begin{pmatrix}
    a_{32} & a_{33} \\
    a_{42} & a_{43}
\end{pmatrix} = \det \begin{pmatrix}
    a_{32} & a_{33} \\
    a_{52} & a_{53}
\end{pmatrix} =\det \begin{pmatrix}
    a_{42} & a_{43} \\
    a_{52} & a_{53}
\end{pmatrix} = 0.\] If $(a_{52},a_{53}) \neq (0,0)$, we get \[ \det \begin{pmatrix}
    a_{22} & a_{23} \\
    a_{32} & a_{33}
\end{pmatrix} = \det \begin{pmatrix}
    a_{22} & a_{23} \\
    a_{42} & a_{43}
\end{pmatrix}, \] which implies $g.(e_1\wedge f_3-e_2\wedge e_3) =0$, a contradiction. Therefore $a_{52} = a_{53}=0$. Considering the image of the second basis vector $e_2\wedge f_3 +2 e_1 \wedge f_2$ we similarly find that $a_{52} = a_{54} = 0$. Therefore $g$ stabilises $\langle e_1,e_2,e_3,f_3 \rangle$, and therefore also its radical $ \langle e_1,e_2 \rangle$, reducing to the case $H\leq G_{\langle e_1,e_2 \rangle}$. The case $H\leq G_{\langle e_1,e_2,e_3 \rangle}$ follows similarly. This proves that $H_{W^*}\leq B$. It remains to show that $B_{W^*}=S^*$. Given $g\in B$, we can write it as \[ g =h_{\alpha_1}(t_1)h_{\alpha_2}(t_2)h_{\alpha_3}(t_3)x_{100}(a_1)x_{110}(a_2)x_{010}(a_3)x_{221}(a_4)x_{121}(a_5)x_{111}(a_6)x_{021}(a_7)x_{011}(a_8)x_{001}(a_9),\] where $t_1,t_2,t_3\in K^*$ and $a_i\in K$ for $1\leq i \leq 9$. Let $U$ be the subspace of $V$ with basis given by \[ u_1,\dots, u_7 = e_1\wedge e_3,\, e_1\wedge f_2,\,e_1\wedge f_1-e_2\wedge f_2,\,e_3\wedge f_2,\,e_3\wedge f_1,\,f_2\wedge f_3,\,f_1\wedge f_2.\] Then $V = W^* \oplus U$. For each basis vector $u^{(i)}$ of $W^*$, write $g.u^{(i)}$ as $w^*+ \sum_{j=1}^7 f_{ij}u_i$, where $w^* \in W^*$ and we view $f_{ij}$ as an element of $\mathbb{F}_7[t_1^{\pm 1},t_2^{\pm 1}, t_3^{\pm 1},a_1,\dots ,a_9]$. Then to determine $B_{W^*}$ it suffices to determine the zero locus of the ideal
$I\otimes K \leq K[t_1^{\pm 1},t_2^{\pm 1}, t_3^{\pm 1},a_1,\dots ,a_9]$, where $I \leq \mathbb{F}_7[t_1^{\pm 1},t_2^{\pm 1}, t_3^{\pm 1},a_1,\dots ,a_9]$ is the ideal generated by $\{f_{ij}:1\leq i,j\leq 7\}$. We do this by first determining a Groebner basis for $I$ using Magma. See Listing~\ref{verb5} for the code that does this. We find that $I$ is generated by \begin{gather*}
   a_1 + 6a_9,\,
    a_2 + 3a_9^2,\,
    a_3 + 6a_9,\,
    a_4 + a_9^5,\,
    a_5 + 6a_9^4,\,
    a_6 + a_9^3,\,
    a_7 + 2a_9^3,\,
    a_8 + 3a_9^2,\,
    a_9^7 + 6a_9,\, \\ 
    t_1 + 6t_2^{-2}t_3^{-3},\,
    t_2 + 6t_2^{-2},\,
    t_3 + 6t_3^{-3},\,
    t_1^{-1} + 6(t_2t_3)^{-1},\,
    t_2^{-3} + 6,\,
    t_3^{-4} + 6.
\end{gather*}

It is now easy to find the zero locus of $I\otimes K$, to determine that \begin{gather*}
   h_{\alpha_1}(t_1)h_{\alpha_2}(t_2)h_{\alpha_3}(t_3)x_{100}(a_1)x_{110}(a_2)x_{010}(a_3)x_{221}(a_4)x_{121}(a_5)x_{111}(a_6)x_{021}(a_7)x_{011}(a_8)x_{001}(a_9) \in B_{W^*} \\
   \iff a_9\in \mathbb{F}_7,\, a_1 = a_9,\, a_2 = 4a_1^2,\, a_3 = a_1,\, a_4 = 6a_1^5,\, a_5 = a_1^4, \, a_6 = -a_1^3,\, a_7 = 5a_1^3,\, a_8 = 4a_1^2,\, \\
   t_2^3 = t_3^4 = 1, t_1 = t_2t_3. 
\end{gather*} Taking $a_1=t_1=t_2=t_3 = 1$ we get the element $x$, while taking $a_1 = 0$ we get the subgroup generated by $\tau^*$. This completes the proof that $B_{W^*} = S^*$.

We now turn to the case $p\neq 7$. The only subspaces of $V$ stabilised by $S^\dag$ are $\langle e_1,e_2,f_3 \rangle$ and $\langle e_3,f_1,f_2 \rangle$. Suppose that $H\leq G_{\langle e_1,e_2,f_3 \rangle}$. Then it is easy to see that $H_{W^*}$ stabilises the subspace $W^\ddag = \langle v^{(0)},v^{(1)},v^{(2)},v^{(4)} \rangle$. Let \[g = u^- n t u,\]
where \begin{align*}
u &= x_{100}(a_1)x_{110}(a_2)x_{010}(a_3)x_{221}(a_4)x_{121}(a_5)x_{111}(a_6)x_{021}(a_7)x_{011}(a_8)x_{001}(a_9), \\
t &=h_{\alpha_1}(t_1)h_{\alpha_2}(t_2)h_{\alpha_3}(t_3), \\
n &\in \{1,n_1,n_3n_2n_3,n_3n_1n_2n_3,n_3n_2n_3n_1,n_3n_1n_2n_3n_1\}, \\
u^- &=x_{100}(b_1)x_{110}(b_2)x_{010}(b_3)x_{221}(b_4)x_{121}(b_5)x_{111}(b_6)x_{021}(b_7)x_{011}(b_8)x_{001}(b_9),\\
u^- &\in \langle X_\alpha : n.\alpha \in \Phi^-,\alpha\in \Phi^+ \rangle,
 \end{align*}

 for $a_1,\dots,a_9,b_1,\dots,b_9\in K$ and $t_1,t_2,t_3\in K^*$. Then $g$ is an arbitrary element of $G_{\langle e_1,e_2,f_3 \rangle}$, written in terms of its Bruhat decomposition. For each possible $n$, we use Magma to solve the system of equations corresponding to $g \in G_{W^\ddag}$, similarly to how we did for $p=7$. More care is now required for the setup of the computations in Magma, since the characteristic is arbitrary. What we do is find a Groebner basis over $\mathbb{Q}$, and also output the list of primes that the algorithm divided by in its various steps. If $p$ is not in such list we can use the Groebner basis to easily solve the system, otherwise we simply run the Groebner basis algorithm again over $\mathbb{F}_p$. This is done as per Listing~\ref{verb6} and Listing~\ref{verb7}. We find that $g$ does not stabilise $W^\ddag$ when $n \in \{n_1,n_3n_2n_3,n_3n_1n_2n_3n_1\}$. When $n\in \{1,n_3n_1n_2n_3,n_3n_2n_3n_1\}$, we find that $g$ stabilises $W^\ddag$ if and only if $u = u^- = 1$ and $t_1,t_2,t_3$ satisfy the following:
\begin{align*}
t_1 &= \frac{1}{t_3}, & t_2 &= \frac{1}{t_3^{10}}, & t_3^{14} &= 1, &\text{ when } n &\in  \{1,n_3n_2n_3n_1\}\\
t_1 &= -\frac{1}{t_3}, & t_2 &= -\frac{1}{t_3^{10}}, & t_3^{14} &= 1, &\text{ when } n &= n_3n_1n_2n_3.\\
 \end{align*}
 This is easily seen to be equivalent to $g\in S^*$, concluding the proof of this Lemma.
\end{proof}

\begin{lemma}\label{lemma c3 normalizer stab}
    Let $H = N_G(S^\dag)$. Then $H_{W^*}= S^*$.
\end{lemma}
\begin{proof}
    Assume $p\neq 7$. Since $x$ is a regular semisimple element, we have $H\leq N_G(T)$, from which it is easy to check that $H \leq T.\langle \tau^* \rangle$. A quick calculation shows that $T_{W^*} = \pm \langle x \rangle$, concluding as required.

    If $p = 7$ we have $H\leq B$, where $B$ is the standard Borel subgroup of $G$. Then we conclude by Lemma~\ref{lemma c3 reducible k=7}.
\end{proof}

\begin{lemma}\label{lemma c3 case dot 2 reduction}
        Suppose that $S^\dag \leq H.\mathbb{Z}_2 < G$, where $H$ is an arbitrary subgroup of $G$. Then $(H.\mathbb{Z}_2)_{W^*}\leq  S^*$ if and only if  $H_{W^*}\leq  S^*$.
\end{lemma}
\begin{proof}
    The forward direction is trivial. Suppose that $H_{W^*}\leq  S^*$.
    Since $S^\dag$ has order $21$, we must have $S^\dag\leq H$.  Since $H_{W^*}\leq  S^*$, the subgroup $S^\dag$ is the unique subgroup of order $21$ in $H_{W^*}$. Therefore $H_{W^*} \lhd (H.\mathbb{Z}_2)_{W^*}$ implies that $S^\dag$ is normal in $(H.\mathbb{Z}_2)_{W^*}$. By Lemma~\ref{lemma c3 normalizer stab} we know that the stabilizer of $W^*$ in $N_G(S^\dag)$ is $S^*$, concluding.
\end{proof}

\begin{proposition}\label{C_3 k=7 prop}
    Let $G=C_3$, $\lambda = \lambda_2$ with $p\neq 3$. Then $C_{\mathcal{S}_7'(V)}=\mathbb{Z}_7.\mathbb{Z}_6$ and $C_{\mathcal{S}_7''(V)} = \mathbb{Z}_{7/(p,7)}.\mathbb{Z}_6$.
\end{proposition}

\begin{proof}
Let $S = G_{W^*}$. We shall prove that $S\leq N_G(S^\dag)$. This will conclude the proof of the proposition as follows. By Lemma~\ref{lemma c3 normalizer stab} we have $(N_G(S^\dag)_{W^*} = S^*$, which then implies $S=S^*$ and that $W^*$ is in a dense $G$-orbit on one of the two $D_7$-orbits on $\mathcal{S}_7(V)$. If $p\neq 7$, then $W_{(\omega,i)}$ and $W_{(\omega^2,i)}$ intersect in a $6$-dimensional subspace, and therefore by Lemma~\ref{lemma interesction maximal totally singular}, they belong to distinct $D_7$-orbits on $\mathcal{S}_7(V)$. They each have stabilizer $S^*$, concluding the $p\neq 7$ case. If $p=7$ the subspace $W^*$ intersects $W_{(2,i)}$ trivially, which by Lemma~\ref{lemma interesction maximal totally singular} implies that $W^*$ and $W_{(4,\zeta)}$ belong to distinct $D_7$-orbits on $\mathcal{S}_7(V)$. Again $G_{W^*} = S^*$, concluding.

In order to prove that $S\leq N_G(\langle x \rangle)$,
we show that for all $p<\infty$ and $e\in \mathbb{Z}_{\geq 1}$, if \[S^\dag \leq R \leq Sp_6(p^e)=Sp_6(q)<G, \text{ with }R\not\leq N_G(S^\dag),\] then $R$ does not stabilise $W^*$. Note that this is indeed sufficient, since if $g\in S\setminus N_G (S^\dag)$, then there must exist $e\in \mathbb{Z}_{\geq 1}$ such that $g \in Sp_6(p^e)$, with $R = \langle S,g\rangle$ satisfying the condition above. We shall make extensive use of maximal subgroups of $Sp_6(q)$, combined with computations in Magma. The $p=\infty$ case then simply follows from the $p<\infty$ case. 

Suppose that \[ S^\dag \leq R \leq M_q < Sp_6(q)<G, \text{ with }R\not\leq N_G(S^\dag),\] where $M_q$ is a maximal subgroup of $Sp_6(q)$, as listed in \cite[Table~8.28,Table~8.29]{MaximalSubgroupsColva}. By Lemma~\ref{lemma c3 reducible k=7} we can assume that $M_q$ is an irreducible subgroup, with order divisible by $|S^\dag| = 21$, and by Lemma~\ref{lemma c3 case dot 2 reduction} we can assume that $M_q$ does not contain a reducible subgroup of index $2$. The goal is to prove that $R$ does not stabilize $W^*$. Going through \cite[Table~8.28,Table~8.29] {MaximalSubgroupsColva} for the maximal subgroups of $Sp_6(q)$, and \cite[Table~8.30,Table~8.31,Table~8.32,Table~8.33,Table~8.34] {MaximalSubgroupsColva} for the maximal subgroups of $SO_6^+(q)$, $SO_6^-(q)$ and $G_2(q)$ in even characteristic, we reduce to having to consider the cases in the following table:
\begin{longtable}{l l l}
\hline  \multicolumn{1}{l}{Class} & \multicolumn{1}{l}{\textbf{$M_q$}}& \multicolumn{1}{l}{Notes} \\ \hline 
\endhead
$\mathscr{C}_2$ & $(Sp_2(q)^3).Sym(3)$ &  \\
$\mathscr{C}_4$ & $Sp_2(q)\otimes GO_3(q)$ & $p\neq 2$ \\
$\mathscr{C}_5$ & $Sp_6(q_0).\mathbb{Z}_2$ & $q=q_0^2$\\
$\mathscr{S}$ & $Z(G).PSL_2(7).\mathbb{Z}_2$ & \\
$\mathscr{S}$ & $Z(G).PSL_2(13)$ & \\
$\mathscr{S}$ & $(Z(G)\times U_3(3)).\mathbb{Z}_2$ & \\
$\mathscr{S}$ & $Z(G).J_2$ & \\
$\mathscr{S}$ & $SL_2(q)$ & $p \geq 7$ \\
\hline
\end{longtable}
If $M_q$ is $Sp_6(q_0).\mathbb{Z}_2$, then by Lemma~\ref{lemma c3 case dot 2 reduction} we reduce to one of the other cases.

Suppose that $M_q = (Sp_2(q)^3).Sym(3)$. By Lemma~\ref{lemma c3 case dot 2 reduction} we can assume that $R\leq (Sp_2(q)^3).\mathbb{Z}_3$. Since $x$ has order $7$, it must be contained in $Sp_2(q)^3$. Therefore $p\neq 7$ since $Sp_2(q)^3$ does not contain a regular unipotent element. Then $Sp_2(q)^3$ must be the stabilizer of the orthogonal sum $\langle e_1 ,f_1 \rangle \perp \langle e_2 ,f_2 \rangle \perp \langle e_3 ,f_3 \rangle$ in $Sp_6(q)$. This implies that $(Sp_2(q)^3).\mathbb{Z}_3 = (Sp_2(q)^3).\langle \tau^4 \rangle$. Since $\tau^4 \in S^\dag$, it remains to determine the stabilizer of $W^*$ in $(Sp_2(q)^3)$, which is easily seen to be $\langle x,\tau^3 \rangle \leq S^*$, concluding.

Suppose that $M_q = Sp_2(q)\otimes GO_3(q)$ with $p\neq 2$. Then $M_q$ does not contain a regular unipotent element, and therefore $p \neq 7$. Comparing the actions of $M_q$ and $S^\dag$ on $V$, we also see that $M_q$ does not contain $S^\dag$.

Suppose that $M_q = Z(G).PSL_2(7).\mathbb{Z}_2$. By Lemma~\ref{lemma c3 case dot 2 reduction} we can assume that $R\leq Z(G).PSL_2(7)$. Here we use Magma to determine that the only possibility for $R$ is $Z(G).PSL_2(7)$ itself, which acts on $V$ as a sum of a $6$-dimensional irreducible and an $8$-dimensional irreducible when $p\neq 7$, and with composition factors of dimension $5,1,3,5$ when $p=7$, therefore not stabilising $W^*$. The Magma code used here and for the next cases can be found in Listing~\ref{verb8}.

Suppose that $M_q = Z(G).PSL_2(13)$. Then $M_q$ does not contain a subgroup of order $21$, a contradiction. Suppose that $M_q = (Z(G)\times U_3(3)).\mathbb{Z}_2$. Clearly we can reduce to the case $R\leq U_3(3)$. Here we find that the only possibility for $R$ is $PSL_2(7)$, concluding like for the $M_q = Z(G).PSL_2(7).\mathbb{Z}_2$ case. 

Suppose that $M_q = Z(G).J_2$. Then we find that all the possibilities for $R$ contain one of $PSL_2(7)$, $SL_2(7)$, $U_3(3)$, concluding via the previous analysis. Finally if $M_q = SL_2(q)$, a similar treatment using the maximal subgroups of $SL_2(q)$ allows us to conclude.
\end{proof}

\section{Proof of Theorem~\ref{maximal semisimple theorem}}\label{section proof of theorem maximal semisimple theorem}
In this section we shall prove Theorem~\ref{maximal semisimple theorem}. Unlike with previous work, we shall not be interested in determining the exact structure of the (semi-)generic stabilizers, if they exist. Instead we will often resort to finding a dense open subset of the variety where the stabilizers have a certain minimal dimension, in order to exclude the possibility of a dense orbit, in a very similar fashion to the work we have done to determine (semi-)generic stabilizers. The following lemma is a crucial tool for narrowing down the cases we will have to consider.
\begin{lemma}\label{dense orbit equivalence semisimple case}
    Let $H= Cl(V_1)\otimes Cl(V_2)\leq Cl(V_1\otimes V_2)=Cl(V)$. Assume that $\dim V_2 = k\dim V_1$ for some $k\geq 1$. Let $G = Cl(V_1)\otimes Cl(V_2')\leq Cl(V_1\otimes V_2')=Cl(V')$ with $\dim V_2'\geq \dim V_2$ and $Cl(V),Cl(V')= SO(V),SO(V')$ or $Cl(V),Cl(V')= Sp(V),Sp(V')$.
    Then $H$ has a dense orbit on $\mathcal{S}_k(V)$ if and only if $G$ has a dense orbit on $\mathcal{S}_k(V')$.
\end{lemma}

\begin{proof}
Write $\dim V_1 = d_1$, $\dim V_2 = d_2 = kd_1$, $\dim V_2' = d_2'$; then $\dim V = d_1d_2$, $\dim V' = d_1d_2'$. We may assume $V_2\leq V_2'$; let $V_2''$ be the orthogonal complement to $V_2$ in $V_2'$, so that $\dim V_2'' = d_2'-d_2$. Let $\epsilon_V$ be $1$ or $-1$ according as $V$ and $V'$ are both orthogonal or both symplectic, and similarly $\epsilon_{V_2}$ be $1$ or $-1$ according a $V_2$, $V_2'$ and $V_2''$ are all orthogonal or all symplectic. Then 
\begin{align*} 
\dim G - \dim \mathcal{S}_k(V') &=  \dim Cl(V_1)+\dim Cl(V_2')-kd_1d_2'+\frac{3k^2+\epsilon_V k}{2}, \\ 
\dim H - \dim \mathcal{S}_k(V) &=  \dim Cl(V_1)+\dim Cl(V_2)-kd_1d_2+\frac{3k^2+\epsilon_V k}{2};
\end{align*}
thus 
\begin{align*} 
&\quad (\dim G - \dim \mathcal{S}_k(V'))-(\dim H-\dim \mathcal{S}_k(V))  \\ 
= &\quad \dim Cl(V_2')-\dim Cl(V_2)-kd_1(d_2'-d_2)\\
= &\quad \frac{1}{2}d_2'(d_2'-\epsilon_{V_2})-\frac{1}{2}d_2(d_2-\epsilon_{V_2})-d_2(d_2'-d_2)\\
= & \quad\frac{1}{2}(d_2'-d_2)(d_2'-d_2-\epsilon_{V_2})\\
= &\quad \dim Cl(V_2'').
\end{align*}
Now let $v_1,\dots ,v_{d_1}$ be a fixed basis of $V_1$. Given $y\in\mathcal{S}_k(V')$, choose a basis $x_1,\dots ,x_k$ of $y$ and write each $x_j$ uniquely as $\sum_{i=1}^{d_1} v_i\otimes u_{ij}$ with each $u_{ij}\in V_2'$; set $\mathrm{supp}_2(y) = \langle u_{ij}:1\leq i\leq d_1,\, 1\leq j\leq k\rangle\leq V_2'$. Define
\[ Y = \{y\in \mathcal{S}_k(V'):\mathrm{supp}_2(y)\text{ is non-degenerate of dimension }d_2\}.\]
The set $Y$ is dense in $\mathcal{S}_k(V')$, because the set $\{ y\in \mathcal{S}_k(V'): \dim \mathrm{supp}_2(y) <d_2\}$ is a proper closed subvariety of $\mathcal{S}_k(V')$, and non-degenerate $d_2$-spaces are dense in the variety of all $d_2$-spaces in $V_2'$; likewise $Y\cap \mathcal{S}_k(V)$ is dense in $\mathcal{S}_k(V)$. Moreover all non-degenerate $d_2$-spaces in $V_2'$ lie in a single $Cl(V_2')$-orbit; thus given $y\in Y$, by applying an element of $G$ we may assume that $\mathrm{supp}_2(y) = V_2$, and then the stabilizer of $y$ in $G$ must fix $V_2$ and hence $V_2''$, whence $G_y = H_y \times Cl(V_2'')$, so that $\dim G_y-\dim H_y = \dim Cl(V_2'')$ and hence \[ \dim \mathcal{S}_k(V')-(\dim G-\dim G_y) = \dim \mathcal{S}_k(V)-(\dim H-\dim H_y).\] Now if $H$ has a dense orbit on $\mathcal{S}_k(V)$, it must meet $Y\cap\mathcal{S}_k(V)$; thus it contains some $y\in Y\cap \mathcal{S}_k(V)$ for which the right side of the above equation is zero, whence the left side is also zero and so $G$ has a dense orbit on $\mathcal{S}_k(V')$. Conversely if $G$ has a dense orbit on $\mathcal{S}_k(V')$, it must meet $Y$ and therefore $Y\cap \mathcal{S}_k(V)$; thus it contains some $y\in Y\cap \mathcal{S}_k(V)$ for which the left side of the above equation is zero, whence the right side is also zero and $H$ has a dense orbit on $\mathcal{S}_k(V)$.
\end{proof}

\begin{lemma}\label{semisimple reduction both orthogonal}
    Let $G=SO(V_1)\otimes SO(V_2)$ with $3\leq \dim V_1\leq \dim V_2$. Then $G$ does not have a dense orbit on $\mathcal{S}_k(V_1\otimes V_2)$ for all $1\leq k\leq \frac{1}{2}\dim V_1\dim V_2$. 
\end{lemma}

\begin{proof}
    Let $m=\dim V_1$, $n=\dim V_2$ and assume that $G$ does have a dense orbit on $\mathcal{S}_k(V_1\otimes V_2)$. Then \[\dim G = \frac{1}{2}(m^2+n^2-m-n)\geq mnk-\frac{3k^2+k}{2} = \dim\mathcal{S}_k(V_1\otimes V_2). \] We first show that $mk\leq n$. If $k=1$ this is immediate, so assume $k>1$. Write $n= am$ for some $a\geq 1$. Then if we define $g:\mathbb{R}\rightarrow\mathbb{R}$ by 
    \[ g(x) =m^2 + m^2 x^2 - m - m x - 2 m^2 x k + 3 k^2 + k,\]
    we have $g(a) = 2(\dim G-\dim \mathcal{S}_k(V_1\otimes V_2))\geq 0$. The discriminant of $g(x)$ is $(-m-2m^2k)^2-4m^2(m^2-m+3k^2+k) = m^2h(m)$, where we define $h:\mathbb{R}\rightarrow \mathbb{R}$ by
    \[ h(x) = 1-4k-12k^2+4(k+1)x+4(k^2-1)x^2.\] In turn the discriminant of $h(x)$ is $16(k+1)^2-16(k^2-1)(1-4k-12k^2)=16(k+1)(12k^3-8k^2-4k+2)\geq 0$, so that the equation $y = h(x)$ has real roots; the positive root is 
    \[ m = \frac{-(k+1)+\sqrt{(k+1)(12k^3-8k^2-4k+2)}}{2(k^2-1)},\] which is easily seen to be always less than $2$. Thus as $m\geq 3$ we have $h(m)>0$, so that $g(x)$ has positive discriminant, and therefore the equation $y=g(x)$ has real roots
    \[ r_1,r_2 = \frac{1+2km\pm\sqrt{h(m)}}{2m},\] where $r_1<r_2$. We claim that $r_1<a$. Since $a\geq 1$ the claim is certainly true if $r_1<1$; we have $r_1<1 \Leftrightarrow h(m) > (2km-2m+1)^2 \Leftrightarrow 2km^2-2m^2+2m-3k^2-k>0$. If $k\leq 4$ the last inequality holds as $m\geq 3$, so we may assume $k\geq 5$. If we had $k\leq \frac{1}{2}m^2$ this would force $m\geq 4$, and then 
    \begin{align*}
        2km^2-2m^2+2m-3k^2-k &= \frac{2}{5}(k-5)m^2+2m+k(\frac{8}{5}m^2-3k-1) \\
        & \geq \frac{2}{5}(k-5)m^2+2m+k(\frac{8}{5}m^2-\frac{3}{2}m^2-1) \\
        & = \frac{2}{5}(k-5)m^2+2m+k(\frac{1}{10}m^2-1) \\
        & > 0;
    \end{align*}
    so we may assume $k\geq \frac{1}{2}m^2$. Write $k = \theta m^2$; as $k\leq \frac{1}{2}mn=\frac{1}{2}am^2$ we have $\frac{1}{2}<\theta\leq \frac{1}{2}a$. Then $r_1<2\theta$ if and only if
    \[1+2\theta m^3-\sqrt{4\theta^2m^6-12\theta^2 m^4 +3\theta m^3-4(1+\theta)m^2+4m+1}<4\theta m,\] which reduces to \[ m(\theta^2(m^2-4)-\theta-1)+(1+2\theta)>0,\] which is true for all $m\geq 3$ and $\theta > \frac{1}{2}$. Thus we do have $r_1<a$ as claimed; so as $g(a)\geq 0$ we must have $r_2<a$, whence $a>\frac{r_1+r_2}{2}= \frac{1+2km}{2m}>k$ as required.    

    Let $U\leq V_2$ be a non-degenerate subspace of dimension $k\dim V_1$. By Lemma~\ref{dense orbit equivalence semisimple case} we know that $SO(V_1)\otimes SO(U)$ has a dense orbit on $\mathcal{S}_k(V\otimes U)$. The dimension requirement is that $ m^2-m+m^2k^2-mk \geq  2 m^2k^2-3k^2-k$, which is absurd when $m\geq 3$ and $k\geq 1$. This contradicts $G$ having a dense orbit on $\mathcal{S}_k(V_1\otimes V_2)$.
\end{proof}

\begin{lemma}\label{semisimple reduction symplectic and orthogonal}
    Let $G=Sp(V_1)\otimes SO(V_2)$ with $4\leq \dim V_1\leq \dim V_2$ or $3\leq \dim V_2\leq \dim V_1$. Then $G$ does not have a dense orbit on $\mathcal{S}_k(V_1\otimes V_2)$ for all $2\leq k\leq \frac{1}{2}\dim V_1\dim V_2$.
\end{lemma}

\begin{proof}
    This is similar to Lemma~\ref{semisimple reduction both orthogonal}. Note that here $k\geq 2$. 
\end{proof}

\begin{lemma}\label{semisimple reduction both symplectic}
    Let $G=Sp(V_1)\otimes Sp(V_2)$ with $4\leq \dim V_1\leq \dim V_2$. Then if $G$ has a dense orbit on $\mathcal{S}_k(V_1\otimes V_2)$ we must have $k=1$.
\end{lemma}

\begin{proof}
    Let $m=\frac{\dim V_1}{2}$, $n = \frac{\dim V_2}{2}$ and assume that $G$ does have a dense orbit on $\mathcal{S}_k(V_1\otimes V_2)$. Then \[\dim G = 2m^2+2n^2+m+n\geq 4mnk-f(k) = \dim\mathcal{S}_k(V_1\otimes V_2), \] where $f(k) = \frac{3k^2+k}{2}$. Similar calculations as in Lemma~\ref{semisimple reduction both orthogonal} show that $k\dim V_1\leq \dim V_2$. By Lemma~\ref{dense orbit equivalence semisimple case} we then deduce that $Sp(V_1)\otimes Sp(U_2)$ has a dense orbit on $\mathcal{S}_k(V_1\otimes U_2)$, where $\dim U_2 = k\dim V_1$. Then dimensional considerations rule out $k\geq 2$.
\end{proof}

\begin{proposition}\label{sp6 sp6 proposition}
     Let $G=Sp(V_1)\otimes Sp(V_1)$ with $\dim V_1\geq 6$. Then $G$ does not have a dense orbit on $\mathcal{S}_1(V_1\otimes V_1)$. 
\end{proposition}

\begin{proof}
    Let $2n = \dim V_1$ and $e_1,\dots,e_n,f_n,\dots,f_1$ be the standard basis of $V_1$.
    Let \[Y=\{\langle \sum_{i=1}^{n} a_i e_i\otimes f_i\rangle: \sum a_i^2 = 0\},\] a subvariety of $\mathcal{S}_1(V_1\otimes V_1)$. Let \[\hat{Y}=\{\langle \sum_{i=1}^{n} a_i e_i\otimes f_i\rangle: \sum a_i^2 = 0, a_i\neq a_j\text{ if }i\neq j \},\] a dense subset of $Y$. Let $y\in\hat{Y}$. Then $\mathrm{Tran}_G(y,Y)$ contains an $A_1^n$ stabilising all elements of $y$, which projects onto each $Sp(V_1)$ as $\bigcap Sp(V_1)_{\langle e_i,f_i\rangle}$. By assumption on the $a_i$'s it is easy to see that this is the connected component of $\mathrm{Tran}_G(y,Y)$. Since $\dim G-\dim \mathrm{Tran}_G(y,Y) = 4n^2+2n -3n = 4n^2-2-(n-2)=\dim \mathcal{S}_1(V_1\otimes V_1) - \dim Y$, we find that $y$ is $Y$-exact. Therefore by Lemma~\ref{loc to a subvariety open set lemma} we conclude that there is an open dense subset of the variety of singular $1$-spaces of $V_1\otimes V_1$ such that all stabilizers are $3n$-dimensional. Therefore $3n$ is the lower bound for the dimension of the stabilizer of any singular $1$-space. Since $\dim G-3n = 4n^2-2-(n-2) > \dim \mathcal{S}_1(V_1\otimes V_1)$ when $n\geq 3$, we conclude that $G$ does not have a dense orbit on $\mathcal{S}_1(V_1\otimes V_1)$. 
\end{proof}

\begin{lemma}\label{sp2 spn general reduction to k=1,2,3,4,maximal}
    Let $G=Sp_2\otimes Sp(V_2)\leq SO(V)$ with $\dim V = 2\dim V_2 \geq 20$. Suppose that $G$ has a dense orbit on $\mathcal{S}_k(V)$ for some $2\leq k\leq \dim V_2$. Then $k$ is either $2,3,4,{\dim V_2}-1$, $\dim V_2'$ or $\dim V_2''$.
\end{lemma}

\begin{proof}
    Calculations similar to the ones in Lemma~\ref{semisimple reduction both orthogonal} show that if $k<\dim V_2-1$ we have $2k\leq \dim V_2$. We can then apply Lemma~\ref{dense orbit equivalence semisimple case} to determine that $Sp_2\otimes Sp_{2k}=Sp(V_1)\otimes Sp(U)$ must have a dense orbit on $\mathcal{S}_k(V_1\otimes U)$. Dimensional considerations then give $ 3+2k^2+k\geq 4k^2-\frac{3}{2}k^2-\frac{k}{2}$, which implies $k\leq 4$. 
\end{proof}

\begin{lemma}\label{sp2 spn reduction small dimension}
    Let $G=Sp_2\otimes Sp(V_2)\leq SO(V)$ with $\dim V = 2\dim V_2\leq 16$. Suppose that $G$ has a dense orbit on $\mathcal{S}_k(V)$ for some $2\leq k\leq \dim V_2-2$. Then either $k=2$ or $(k,\dim V_2)$ is one of $(3,6),(3,8),(4,6),(4,8),(6,8)$. 
\end{lemma}

\begin{proof}
This is simply a matter of checking the dimension of $\mathcal{S}_k(V)$ in all finitely many cases.
\end{proof}

\begin{proposition}\label{sp2 spn maximal totally singular}
    Let $G=Sp_2\otimes Sp_{2n} \leq SO_{4n}=SO(V)$. Then $G$ has a dense orbit on $\mathcal{S}_{2n}'(V)$ (and $\mathcal{S}_{2n}''(V)$) if and only if $n = 1,2,3$.
\end{proposition}
\begin{proof}
    When $n=1,2$ the group $G$ is spherical in $SO(V)$. Therefore assume that $n\geq 3$. Let $ V_1=\langle e,f\rangle $ and $V_2=\langle e_1,\dots,e_n, f_n,\dots,f_1\rangle$ so that $G=Sp(V_1)\otimes Sp(V_2)$ and the given bases are the standard bases for $V_1$ and $V_2$. Given $\mathbf{a} = (a_1,\dots,a_n)\in K^n$ define $W_{\mathbf{a}} = \langle (e+a_i f)\otimes e_i,  (e+a_i f)\otimes f_i:1\leq i\leq n\rangle$, a $2n$-dimensional totally singular subspace of $V$. Define
    \[Y = \{W_{\mathbf{a}}: \mathbf{a}\in K^n\},\] an $n$-dimensional subvariety of $\mathcal{S}_{2n}'(V)$. Let \[\hat{Y} = \{W_{\mathbf{a}}: \mathbf{a}\in K^n,a_1\dots a_n\neq 0,\,a_i\neq a_j \text{  for all }i\neq j\},\] a dense subset of $Y$.
    Let $y\in \hat{Y}$. It is easily seen that $\mathrm{Tran}_G(y,Y)$ has connected component $A_1^{n+1} = Sp(V_1)\otimes \bigcap Sp(V_2)_{\langle e_i+f_i \rangle}$, while $G_y^0 = 1\otimes Sp(V_2)_{\sum \langle e_i+f_i \rangle}\simeq A_1^n$. Since \[\dim G-\dim \mathrm{Tran}_G(y,Y) = \dim\mathcal{S}_{2n}'(V)-\dim Y,\] we have that $\hat{Y}$ is $Y$-exact. Therefore by Lemma~\ref{loc to a subvariety open set lemma} and Corollary~\ref{minimum dimension generic} we know that $\dim A_1^n$ is the minimum dimension for the stabilizer of any $y\in \mathcal{S}_{2n}'(V)$. Dimensional considerations rule out $n\geq 4$, while for $n=3$ we have $\dim G-\dim A_1^3 = 15=\dim\mathcal{S}_{2n}'(V).$ By Lemma~\ref{lemma interesction maximal totally singular}, changing the definition of $W_{\mathbf{a}}$ by swapping the first two generators $(e+a_1f)\otimes e_1,(e+a_1f)\otimes f_1$ with $e\otimes (e_1+a_1f_1),f\otimes (e_1+a_1f_1)$ leads to the same result for the action on $\mathcal{S}_{2n}''(V)$.
\end{proof}

\begin{proposition}\label{almost maximal totally singular}
    Let $G=Sp_2\otimes Sp_{2n} \leq SO(4n)=SO(V)$. Then $G$ has a dense orbit on $\mathcal{S}_{2n-1}(V)$ if and only if $n = 1,2,3$.
\end{proposition}
\begin{proof}
    Suppose that $G$ has a dense orbit on $\mathcal{S}_{2n-1}(V)$. Every $y\in \mathcal{S}_{2n-1}(V)$ is contained in precisely one element of $\mathcal{S}_{2n}'(V)$ and one of $\mathcal{S}_{2n}''(V)$.
    Let $\mathcal{O}$ be the dense orbit of $G$ on $\mathcal{S}_{2n-1}(V)$; then its complement $\mathcal{S}_{2n-1}(V) \setminus \mathcal{O}$ is contained in a proper closed subvariety $X$ of $\mathcal{S}_{2n-1}(V)$. Let $Z$ be the set of elements of $\mathcal{S}'_{2n}(V)$ all of whose hyperplanes lie in $X$; then $Z$ is a proper closed subvariety of $\mathcal{S}'_{2n}(V)$, so its complement $\mathcal{S}'_{2n}(V) \setminus Z$ is a dense subset of $\mathcal{S}'_{2n}(V)$ with the property that any of its elements has a hyperplane lying in $\mathcal{O}$. Thus given two elements of $\mathcal{S}'_{2n}(V) \setminus Z$, we can choose hyperplanes within them and an element of $G$ which sends one hyperplane to the other and therefore one element of $\mathcal{S}'_{2n}(V) \setminus Z$ to the other; so $\mathcal{S}'_{2n}(V) \setminus Z$ lies in a single $G$-orbit, and therefore $G$ has a dense orbit on $\mathcal{S}'_{2n}(V)$. Replacing $\mathcal{S}'_{2n}(V)$ by $\mathcal{S}''_{2n}(V)$ shows that $G$ also has a dense orbit on $\mathcal{S}''_{2n}(V)$. By Proposition~\ref{sp2 spn maximal totally singular} we therefore have $n=1,2,3$. 
    When $n=1,2$ the group $G$ is spherical in $SO(V)$. Therefore assume that $n= 3$. Given $y$ in the dense $G$-orbit on $\mathcal{S}_{6}'(V)$, the group induced by $G_y$ on $y$ is $A_1^3$ acting as a sum of three natural modules for $A_1$. Therefore $G_y$ has a dense orbit on $5$-spaces of $y$, concluding that $G$ has a dense orbit on $\mathcal{S}_{5}(V)$.
\end{proof}

\begin{proposition}\label{sp2 sp6 k=3}
    Let $G=Sp_2\otimes Sp_6 \leq SO(12)=SO(V)$. Then $G$ has a dense orbit on $\mathcal{S}_3(V)$.
\end{proposition}

\begin{proof}
    Let $ V_1=\langle e,f\rangle $ and $V_2=\langle e_1,e_2,e_3,f_3,f_2,f_1\rangle$ so that $G=Sp(V_1)\otimes Sp(V_2)$ and the given bases are the standard bases for $V_{nat}$. Let \[
    W = \langle e\otimes e_1+f\otimes e_2,\quad e\otimes f_2+f\otimes f_1,\quad e\otimes (e_2+e_3)+f\otimes (f_3-f_2)\rangle, 
    \] a totally singular $3$-space of $V$. Let $S=G_W.$ Let $g\in S$ such that $g=1\otimes g_1$. Then $g_1$ fixes $\langle e_1,f_2,e_2+e_3\rangle$ and $\langle e_2,f_1,f_3-f_2\rangle$, and consequently their radicals $\langle e_1\rangle$ and $\langle f_1 \rangle$. One then quickly reaches the conclusion that $g_1$ must also fix $\langle e_2 \rangle$, $\langle f_2 \rangle$, $\langle e_2+e_3 \rangle $ and $\langle f_3-f_2\rangle$. However this would mean that $g$ acts on $W$ by sending $e\otimes (e_2+e_3)+f\otimes (f_3-f_2)\mapsto \lambda e\otimes (e_2+e_3)+\lambda^{-1}f\otimes (f_3-f_2)$, implying that $g_1 =\pm 1$. To conclude we observe that given $g = g_1\otimes 1$, by a simple application of Witt's Lemma we can find $g' = 1\otimes g_2$ such that $gg'\in S$. Therefore $S^0=Sp_2$. Since $\dim G-\dim S = 21 = \dim\mathcal{S}_3(V)$, we conclude that $G$ has a dense orbit on $\mathcal{S}_3(V)$.
\end{proof}

\begin{proposition}\label{sp2 sp6 k=4 no dense orbit}
    Let $G=Sp_2\otimes Sp_6 \leq SO_{12}=SO(V)$. Then $G$ has no dense orbit on $\mathcal{S}_4(V)$.
\end{proposition}

\begin{proof}
    Let $ V_1=\langle e,f\rangle $ and $V_2=\langle e_1,e_2,e_3,f_3,f_2,f_1\rangle$ so that $G=Sp(V_1)\otimes Sp(V_2)$ and the given bases are the standard bases for $V_{nat}$. The stabilizer of an element in a dense orbit would have dimension $1$. We will show that already the group $G_2:=1\otimes Sp_6 < G$ acts on $\mathcal{S}_4(V)$ with stabilizers that are at least $3$-dimensional.
    Let $W_{abcd}$ be the totally singular $4$-space spanned by 
        \[ \begin{array}{ll}%
e\otimes e_1+f\otimes (ae_1+be_2+ce_3+de_4),  &  e\otimes e_2+f\otimes (be_1+de_2),\\
e\otimes f_1+f\otimes (af_1+bf_2+cf_3+de_4),  &  e\otimes f_2+f\otimes (bf_1+df_2).\\ 
\end{array}\]
     Let $Y = \{ W_{abcd}:a,b,c,d\in K\}$, a $4$-dimensional subvariety of $\mathcal{S}_4(V)$. Let $\hat{Y} = \{ W_{abcd}:a,b,c,d\in K^*\}$, a dense subset of $Y$. Take $y=W_{abcd}\in \hat{Y}$. Since $y\subseteq e\otimes \langle e_1,f_2,e_2,f_2\rangle +f\otimes V_6$, the stabilizer $(G_2)_y$ must preserve $\langle e_1,f_1,e_2,f_2\rangle$, and therefore its orthogonal complement $\langle e_3,f_3\rangle$. Since for all $g\in (G_2)_y$ we must have $g.(e\otimes e_2+f\otimes (be_1+de_2))\in y$ and $g.(e\otimes f_2+f\otimes (bf_1+df_2))\in y$, we get that $\langle e_2,f_2\rangle$ must also be preserved by $(G_2)_y$. Therefore $(G_2)_y\leq A_1^3$. It is now immediate to see that the image of any $e_i$ or $f_i$ in $\langle e_i,f_i\rangle $, completely determines the element $g\in (G_2)_y$. Since the standard diagonal subgroup $A:=A_1\leq A_1^3$ fixes $y$, we must then have $(G_2)_y = A$. Now assume that $g\in \mathrm{Tran}_{G_2}(y,Y)$. Again we have $g\in A_1^3$, and since $A$ clearly fixes any element of $Y$, it is the stabilizer of $g.y$. Therefore $g\in N_{A_1^3}(A)$, which is a finite extension of $A$. This shows that $\dim \mathrm{Tran}_{G_2}(y,Y) = 3$, and therefore $\codim \mathrm{Tran}_{G_2}(y,Y) = 18 = 22-4 = \codim Y$. This shows that all points in $\hat{Y}$ are $Y$-exact and Lemma~\ref{loc to a subvariety open set lemma} allows us to conclude that $3=\dim A_1$ is the lower bound for the dimension of any stabilizer for the $G_2$-action on $\mathcal{S}_4(V)$. In particular this proves that $G$ has no dense orbit on $\mathcal{S}_4(V)$.
\end{proof}

\begin{proposition}\label{sp2 sp8 k=4 no dense orbit}
    Let $G=Sp_2\otimes Sp_8 \leq SO_{16}=SO(V)$. Then $G$ has no dense orbit on $\mathcal{S}_4(V)$ and on $\mathcal{S}_6(V)$.
\end{proposition}

\begin{proof}
 Let $ V_1=\langle e,f\rangle $ and $V_2=\langle e_1,e_2,e_3,e_4,f_4,f_3,f_2,f_1\rangle$ so that $G=Sp(V_1)\otimes Sp(V_2)$ and the given bases are the standard bases for $V_1$ and $V_2$. The stabilizer of an element in a dense orbit would have dimension $2$. We will show that already the group $G_2:=1\otimes Sp_8 < G$ acts on $\mathcal{S}_4(V)$ with stabilizers that are at least $3$-dimensional.
    Given $ \mathbf{a}\in K^7$ let $W_{\mathbf{a}}$ be the totally singular $4$-space spanned by vectors
    \[ \begin{array}{ll}%
e\otimes e_1+f\otimes (a_1e_1+a_2e_2+a_3e_3+a_4e_4),  &  e\otimes e_2+f\otimes (a_2e_1+a_5e_2+a_6e_3+a_7e_4),\\
e\otimes f_1+f\otimes (a_1f_1+a_2f_2+a_3f_3+a_4f_4),  &  e\otimes f_2+f\otimes (a_2f_1+a_5f_2+a_6f_3+a_7f_4).\\ 
\end{array}\]
Let $Y = \{ W_{\mathbf{a}}:\mathbf{a}\in K^7\}$, a $7$-dimensional subvariety of $\mathcal{S}_4(V)$. Let $\hat{Y}_1 = \{ W_{\mathbf{a}}:\mathbf{a}\in (K^*)^7\}$, a dense subset of $Y$. The standard diagonal $A_1\leq A_1^4 = \bigcap (G_2)_{\langle e_i,f_i\rangle}$ fixes any $y\in Y$. Call this $A_1$-subgroup $A$. Let $y\in \hat{Y}$. We will now prove that the connected component of $(G_2)_y$ is $A$. We begin by observing that  $(G_2)_y$ fixes $\langle e_1,f_1,e_2,f_2 \rangle$ and therefore $(G_2)_y\leq C_2C_2 $. Let $\pi_i (A)$ ($i=1,2$) denote the projection of $A$ onto each $C_2$. Assume $p\neq 2$.
Then $\pi_i (A)$ is a diagonal $A_1$ in $C_2$, which lies in two opposite parabolic subgroups of $C_2$, acting irreducibly on their unipotent radicals, and in precisely one Levi subgroup $L_i$. Also, $\pi_i (A)$ is maximal in infinitely many $A_1^2$'s stabilising a sum of two non-degenerate $2$-spaces. Let $M$ be a minimal connected overgroup of $A$ in $C_2C_2$. Then $\pi_i(M)$ is one of $A_1$, $A_1^2$, $L_i$, $U_3A_1$. If $\pi_1(M) \simeq \pi_2(M)$, then $M$ is a diagonal $\pi_1(M)$ in $C_2C_2$ by minimality. Suppose that $\pi_1(M) \not\simeq \pi_2(M)$. If $\pi_i(M)\not\leq L_i$ for $i=1$ or $i=2$, then $M$ contains $\pi_1 (A)\times \pi_2(A)$.  
Summarising, the minimal connected ovegroups of $A$ in $C_2C_2$ are as follows:
\begin{enumerate}[label=(\roman*)]
    \item $A_1^2$ diagonal in $C_2C_2$;
    \item $A_1^2 = \pi_1 (A)\times \pi_2(A)$;
    \item $A T_1$, where $T_1< Z(L_1)\times Z(L_2)$ is a $1$-dimensional torus;
    \item $U_3A$ diagonal in $C_2C_2$.
\end{enumerate}
Assume we are in the first case and $M=A_1^2$ is diagonal in $C_2C_2$. Then $\pi_1 (A)$ is maximal in an $A_1^2$ fixing $\langle e_1+\lambda_1 e_2,f_1+\lambda_1 f_2\rangle \perp \langle e_1-\lambda_1^{-1} e_2,f_1-\lambda_1^{-1} f_2\rangle$, while $\pi_2 (A)$ is maximal in an $A_1^2$ fixing $\langle e_3+\lambda_2 e_4,f_3+\lambda_2 f_4\rangle \perp \langle e_3-\lambda_2^{-1} e_4,f_3-\lambda_2^{-1} f_4\rangle$, for some $\lambda_1,\lambda_2\in K^*$. Let $y\in \hat{Y}_1$ and assume that $M$ fixes $y$. Considering a $1$-dimensional torus in $M$ but not in $A$, one finds four independent equations in terms of the entries of $\mathbf{a}$ and $\lambda_1,\lambda_2$, that all need to be satisfied since $M$ fixes $y$. As the variety of diagonal $A_1^2$’s from case $(i)$ is $2$-dimensional, the elements $y$ of $Y$ such that there is some such diagonal $A_1^2$ fixing $y$ lie in a subvariety of $Y$ whose codimension is at least $2$. Therefore there is a dense subset $\hat{Y}_2$ of $Y$ with the property that no minimal connected overgroup of $A$ of type $(i)$ fixes some $y\in \hat{Y}_2$. The same can be quickly deduced for the other cases. This shows that there exists a dense subset $\hat{Y}$ of $Y$ such that $A$ is the connected component of the stabilizer of any $y\in \hat{Y}$. Now take $y\in \hat{Y}$ and $g\in \mathrm{Tran}_{G_2}(y,Y)$. Again we find $g\in C_2C_2$, and since $A$ clearly fixes all elements of $Y$, it must be the connected component of the stabilizer of $g.y$. Therefore $g\in N_{C_2C_2}(A)$. The connected component of $N_{C_2C_2}(A)$ is $AT_2$, where $T_2$ is a $2$-dimensional torus, and therefore $\codim \mathrm{Tran}_{G_2}(y,Y) \geq 31 $. However $\codim Y = 31$, which means that $\codim \mathrm{Tran}_{G_2}(y,Y) = 31 $ and that all points of $\hat{Y}$ are $Y$-exact. By \ref{loc to a subvariety open set lemma} this proves that the minimal dimension for the stabilizer of any totally singular $4$-space of $V$ is $3$.

If instead $p=2$ the reasoning is similar, with the difference that $\pi_i(A)$ is contained in a single parabolic subgroup of $C_2$ and in no Levi subgroup, and has connected centralizer $U_1$. Therefore the minimal connected ovegroups of $A$ in $C_2C_2$ are as follows:
\begin{enumerate}[label=(\roman*)]
    \item $A_1^2$ diagonal in $C_2C_2$;
    \item $A_1^2 = \pi_1 (A)\times \pi_2(A)$;
    \item $U_1 A $, where $U_1\leq C_{C_2}(\pi_1(A))\times C_{C_2}(\pi_2(A))$ is a $1$-dimensional unipotent subgroup.
\end{enumerate}
The same analysis then concludes. The case $\mathcal{S}_6(V)$ is entirely similar. 
\end{proof}

\renewcommand*{\proofname}{Proof of Theorem~\ref{maximal semisimple theorem}.}
\begin{proof}
Recall that $V=V_1\otimes V_2$. If \( G = SO(V_1) \otimes SO(V_2) \leq SO(V) \),
Lemma~\ref{semisimple reduction both orthogonal} shows that \( G \) has no dense orbit on \( \mathcal{S}^k(V) \). If \( G = Sp(V_1) \otimes SO(V_2) \),
Lemma~\ref{semisimple reduction symplectic and orthogonal} likewise shows that \( G \) has no dense orbit on \( \mathcal{S}^k(V) \) if \( k \geq 2 \);
if instead \( k = 1 \) then \( \mathcal{S}^k(V) = \mathcal{G}^k(V) \),
and so \( G \) has a dense orbit on \( \mathcal{S}^k(V) \) if and only if \( K^*G \) has a dense orbit on \( V \),
i.e., if and only if \( (K^*G, V) \) is a prehomogeneous vector space,
giving case (i) in the statement of Theorem~\ref{maximal semisimple theorem}.

For the remainder of the argument assume \( G = Sp(V_1) \otimes Sp(V_2) \subseteq SO(V) \)
with \( \dim V_1 \leq \dim V_2 \), and \( G \) has a dense orbit on \( \mathcal{S}^k(V) \). If \( k = 1 \), Proposition~\ref{sp6 sp6 proposition} and Lemma~\ref{dense orbit equivalence semisimple case} between them show that we must have \( \dim V_1 = 2 \) or $4$; in both possibilities \cite[Thm. 3]{rizzoli} shows that \( G \) does indeed have a dense orbit on \( \mathcal{S}^k(V) \), giving cases (ii) and (vii) respectively in the statement of Theorem~\ref{maximal semisimple theorem}. Now suppose \( k \geq 2 \). Lemma~\ref{semisimple reduction both symplectic} shows that we must have \( \dim V_1 = 2 \);
write \( \dim V_2 = 2n \). If \( n \geq 5 \), Lemma~\ref{sp2 spn general reduction to k=1,2,3,4,maximal} shows that \( k = 2, 3, 4, 2n - 1, (2n)' \) or \( (2n)'' \);
if instead \( n \leq 4 \), Lemma~\ref{sp2 spn reduction small dimension} shows that either \( k = 2 \) or \( (k, 2n) = (3, 6), (3, 8), (4, 6), (4, 8) \) or \( (6, 8) \). If \( k = (2n)' \) or \( (2n)'' \), Proposition~\ref{sp2 spn maximal totally singular} shows that we must have \( n = 1, 2, 3 \)
(and that then \( G \) does indeed have a dense orbit), giving case (v); if \( k = 2n - 1 \), Proposition~\ref{almost maximal totally singular} shows the same thing, giving case (vi). Propositions \ref{sp2 sp6 k=4 no dense orbit} and \ref{sp2 sp8 k=4 no dense orbit} show that \( (k, 2n) \neq (4, 6), (4, 8) \) or \( (6, 8) \),
and then Lemma~\ref{dense orbit equivalence semisimple case} shows that \( k \neq 4 \). If \( k = 3 \), Proposition~\ref{sp2 sp6 k=3} shows that \( G \) does have a dense orbit if \( n = 3 \), and then Lemma~\ref{dense orbit equivalence semisimple case} shows that the same is true if \( n > 3 \), giving case (iv). Finally if \( k = 2 \) the sphericality of \( Sp_2 \otimes Sp_4 \) shows that \( G \) does have a dense orbit if \( n = 2 \), and then Lemma~\ref{dense orbit equivalence semisimple case} shows that the same is true if \( n > 2 \), giving case (iii).
\end{proof}

\section{Proof of Theorem~\ref{theorem double coset density}}\label{section proof of last theorem}
We conclude with the proof of Theorem~\ref{theorem double coset density}.
\renewcommand*{\proofname}{Proof of Theorem~\ref{theorem double coset density}.}
\begin{proof}

    By Theorem~\ref{subgroupstructure} either both $G$ and $H$ are parabolic subgroups, and by the Bruhat decomposition we have finitely many $(G,H)$-double cosets in $\Gamma$; or they are both reductive, and by \cite[Thm.~A]{brundan} there is a dense double coset if and only if there is a single double coset (hence a factorization $\Gamma = GH$); or one of the two subgroups, say $H$, is a maximal parabolic, and $G$ is reductive. 
    It remains to deal with this last case, so assume that $H =P_k$ is a maximal parabolic and $G$ is reductive. If $\Gamma = SL(V)$ or if $\Gamma = Sp(V)$ with $k=1$, there is a dense $(G,H)$-double coset in $\Gamma$ if and only if $G$ acts on the Grassmannian $\mathcal{G}_k(V)$ with a dense orbit. This is equivalent to $GL_k\otimes G$ acting on $K^k\otimes V$ with a dense orbit - which is equivalent to the pair $(GL_k\otimes G, K^k\otimes V)$ being a prehomogeneous vector space, as classified by \cite{satokimura}\cite{chen1}\cite{chen2}. 

    Now assume that $\Gamma\neq SL(V)$ and $k\neq 1$ if $\Gamma = Sp(V)$. There are $4$ options for $G$.
    The first option is for $G$ to be simple and irreducible on $V$, in which case there is a dense $(G,H)$-double coset if and only if $G$ has a dense orbit on the variety of totally singular subspaces corresponding to $\Gamma/H$, as classified by Theorem~\ref{density spaces theorem}.
    
    The second option is for $G$ to be the connected component of the stabilizer of an orthogonal sum, i.e. $G=Sp(V_1)\times Sp(V_2)$ and $\Gamma = Sp(V_1\perp V_2)$ or $G=SO(V_1)\times SO(V_2)$ and $\Gamma = SO(V_1\perp V_2)$, and in both cases $G$ is a spherical subgroup of $\Gamma$. Thirdly, we can have $G$ being the stabilizer of a degenerate but non-singular $1$-space of $V$, where $p=2$ and $V$ is orthogonal. In this case $G$ is again spherical in $\Gamma$.

    Lastly, $G$ can be semisimple but not simple, acting irreducibly and tensor decomposably on $V$. The possibilities for $(G,k)$ are then given by Theorem~\ref{maximal semisimple theorem}. 
\end{proof}

\appendix
\section{Magma code}\label{magma code appendix}

\begin{lstlisting}[label=verb1,caption=Double covers of Sym(6) and Alt(7)]
//Construct the double cover of Sym(6).
G := Sym(6);
F := FPGroup(G);
F2 := pCover(G, F, 2);
G2 := PermutationGroup(F2);

//There is a single conjugacy class of elements of order 5. Such an element generates <x>
X2 := sub<G2|ConjugacyClasses(G2)[8][3]>;

/* 
List the chief factors of overgroups of X2 in G2 that do not normalise X2.
Note how they all contain a double cover of Alt(5).
To run the check for Alt(7), change G to Alt(7) and set X2 := sub<G2|ConjugacyClasses(G2)[6][3]>;
*/

im:=IntermediateSubgroups(G2,X2);
imNN:=[H:H in im|IsNormal(H,X2) eq false];
for H in imNN do
	ChiefFactors(H);
end for;

// We now determine how 2.Alt(5) = SL(2,5) acts on V

G := SL(2,5);
C := CharacterTable(G);

chi := C[6];
IsSymplecticCharacter(chi);
sym2chi := Symmetrization(chi,[2,0]);
[InnerProduct(C[i],sym2chi) : i in [1..#C]];

// The output [ 0, 0, 0, 1, 1, 0, 1, 0, 0 ] indicates that 2.Alt(5) acts on V as 3+3+4.

\end{lstlisting}

\begin{lstlisting}[label=verb2,caption=The case $2.2^4.Sym(5)$]
// Get 2.2^4.Sym(5) directly from Sp(4,7)

cms := ClassicalMaximals("S",4,7);
G:= cms[#cms-1];
X := sub<G|ConjugacyClasses(G)[22][3]>;
im:=IntermediateSubgroups(G,X);
imNN:=[H:H in im|IsNormal(H,X) eq false];
for H in imNN do
	ChiefFactors(H);
end for;

// The first element of imNN is the subgroup 2.2^4.5

G:=imNN[1];
C := CharacterTable(G);

chi:= C[6];
sym2chi := Symmetrization(chi,[2,0]);
for i in [1..#C] do
	if InnerProduct(C[i],sym2chi) eq 1 then
		C[i];
	end if;
end for;

// The output consists of the two consitutents of S^2(chi), two distinct self dual characters of degree 5.
\end{lstlisting}

\begin{lstlisting}[label=verb3,caption=The case $2.Alt(6)$ in characteristic 5]
// Get the subgroup M = 2.Alt(6) in Sp(4,5)
cms:=ClassicalMaximals("S",4,5);
M:=cms[#cms];

// There are 2 conjugacy classes of elements of order 5, leading to the same result
U := sub<M|ConjugacyClasses(M)[6][3]>;

/* 
List the chief factors of overgroups of U in M that do not normalise U.
Note how they all contain a double cover of Alt(5) = SL(2,5).
In each case determine the composition factors for the action on V, by taking the symmetric square of their module.
Also check that all composition factors are indeed absolutely irreducible.
*/

im:=IntermediateSubgroups(M,U);
imNN:=[H:H in im|IsNormal(H,U) eq false];
for H in imNN do
	ChiefFactors(H);
	V_H := GModule(H);
	CompositionFactors(V_H);
	CompositionFactors(SymmetricSquare(V_H));
	// Output "true" as all composition factors of the symmetric square are absolutely irreducible.
	&and[IsAbsolutelyIrreducible(comp) : comp in CompositionFactors(SymmetricSquare(V_H))]; 
end for;
\end{lstlisting}

\begin{lstlisting}[label=verb4,caption=The case $2.2^4.5$ in characteristic 5]
// Get the subgroup M = 2.2^4.Alt(5) in Sp(4,5)
cms:=ClassicalMaximals("S",4,5);
M:=cms[#cms-1];

// There are 2 conjugacy classes of elements of order 10, leading to the same result
U := sub<M|ConjugacyClasses(M)[15][3]>;

/* 
List the chief factors of overgroups of U in M that do not normalise U.
In each case determine the composition factors for the action on V, by taking the symmetric square of their module.
Also check that all composition factors are indeed absolutely irreducible.
*/

im:=IntermediateSubgroups(M,U);
imNN:=[H:H in im|IsNormal(H,U) eq false];

for H in imNN do
	ChiefFactors(H);
	V_H := GModule(H);
	CompositionFactors(V_H);
	CompositionFactors(SymmetricSquare(V_H));
	// Output "true" as all composition factors of the symmetric square are absolutely irreducible.
	&and[IsAbsolutelyIrreducible(comp) : comp in CompositionFactors(SymmetricSquare(V_H))]; 
end for;

// Check that 2.2^4.5 has two self-dual non-isomorphic composition factors on V
H := imNN[1];
V_H := GModule(H);
comps := CompositionFactors(SymmetricSquare(V_H));
&and[IsSelfDual(comp) : comp in comps];
IsIsomorphic(comps[1],comps[2]);

\end{lstlisting}

\begin{lstlisting}[label=verb5,caption=Groebner basis for I in characteristic 7]
    //Construct C_3 and its lambda_2 representation over the field of fractions of a polynomial ring over GF(7).
R<a1,a2,a3,a4,a5,a6,a7,a8,a9,t1,t2,t3,t1inv,t2inv,t3inv> := PolynomialRing(GF(7),15);
F<b1,b2,b3,b4,b5,b6,b7,b8,b9,x1,x2,x3,x1inv,x2inv,x3inv> := FieldOfFractions(R);
C := GroupOfLieType("C3",F:Isogeny:="SC");
f:=HighestWeightRepresentation(C,[0,1,0]);
V3:=VectorSpace(F,3);

//Define an arbitrary g in the standard Borel.
g:=elt<C|<3,a9>,<5,a8>,<7,a7>,<6,a6>,<8,a5>,<9,a4>,<2,a3>,<4,a2>,<1,a1>,V3![t1,t2,t3]>;

// Define a basis for W7, which is easily seen to correspond to W^*.
V:=VectorSpace(F,14);
v1:=V.1+4*V.11+3*V.12;
v2:=V.3+3*V.14;
v3:=V.4+4*V.14;
v4:=V.5+5*V.6;
v5:=V.7+3*V.8;
v6:=V.9+5*V.10;
v7:=V.13;
W7:=sub<V|v1,v2,v3,v4,v5,v6,v7>;
//Extend the basis
B:=[v1,v2,v3,v4,v5,v6,v7,V.2,V.6,V.8,V.10,V.11,V.12,V.14];
VB:=VectorSpaceWithBasis(B);

// Build set of generators of the ideal I
polys := {x1*x1inv-1,x2*x2inv-1,x3*x3inv-1};
for v in Basis(W7) do
	for poly in Coordinates(VB,v*f(g))[8..14] do
		Include(~polys,poly);
	end for;
end for;

// Build ideal I and determine its Groebner basis. It takes under 1 second.
I := ideal<R|[Numerator(p):p in polys]>;
time GroebnerBasis(I);
\end{lstlisting}

\begin{lstlisting}[label=verb6,caption=Groebner basis routine in characteristic not 7]
/*
The following function is a wrapper for the routine of finding the Groebner basis
of the system of polynomials that determine the stabilizer in P_3 of W^*.
The variable field can be Rationals() or a finite field, while j is an integer 
between 1 and 6, corresponding to the 6 possibilities for the Bruhat
decomposition of an element in P_3. If field == Rationals(), the function also returns
a list of primes that need to be checked individually.
*/

findGroebnerBasis := function(field, j)
	
	// Define the group and the representation. 
    R<a1, a2, a3, a4, a5, a6, a7, a8, a9, b1, b2, b3, b4, b5, b6, b7, b8, b9, 
       t1, t2, t3, t1inv, t2inv, t3inv, om, i> := PolynomialRing(field, 26);
    F<a1_f, a2_f, a3_f, a4_f, a5_f, a6_f, a7_f, a8_f, a9_f, b1_f, b2_f, b3_f, b4_f, 
       b5_f, b6_f, b7_f, b8_f, b9_f, t1_f, t2_f, t3_f, t1inv_f, t2inv_f, t3inv_f, 
       om_f, i_f> := FieldOfFractions(R);
    C := GroupOfLieType("C3", F : Isogeny := "SC");
    f := HighestWeightRepresentation(C, [0, 1, 0]);

    V3 := VectorSpace(F, 3);
    V := VectorSpace(F, 14);

    v1 := V.8 - om^2 * V.7;
    v2 := V.5 - i * V.6;
    v3 := V.4 + i * V.3;
    v4 := V.14 - i * V.2;
    v5 := V.1 - i * V.13;
    v6 := V.11 + i * V.12;
    v7 := V.10 + i * V.9;

    // Here Wdd is the subspace W^\ddag, while the span of all 7 vectors v1,...,v7 is W^*.
    Wdd := sub<V | v1, v2, v3, v5>;

    // Extend the basis for Wdd to a basis for the whole module
    B := [v1, v2, v3, v5, V.8, V.6, V.4, V.13, V.2, V.9, V.10, V.11, V.12, V.14];
    VB := VectorSpaceWithBasis(B);

    // Define the list of (preimages) of Weyl group elements that belong to P_3
    ns := [Identity(C), elt<C | 1>, elt<C | 3, 2, 3>, elt<C | 3, 2, 1, 3>, 
           elt<C | 1, 3, 2, 3>, elt<C | 1, 3, 2, 1, 3>];
    // Each such Weyl group element has a corresponding u^-, generated by positive root elements
    // that are sent to negative root elements by the Weyl group element.
    u_minuss := [Identity(C), elt<C | <1, b9>>, elt<C | <3, b1>, <5, b2>, <7, b3>>, 
                 elt<C | <3, b1>, <6, b4>, <9, b6>, <1, b9>>, 
                 elt<C | <3, b1>, <5, b2>, <7, b3>, <6, b4>>, 
                 elt<C | <3, b1>, <5, b2>, <6, b4>, <9, b6>, <1, b9>>];
    n := ns[j];
    u_minus := u_minuss[j];

    // Write an arbitrary elemnt g belonging to the double coset B n B.
    g := elt<C | <3, a1>, <5, a2>, <7, a3>, <6, a4>, <8, a5>, <9, a6>, 
             <2, a7>, <4, a8>, <1, a9>, V3![t1, t2, t3]> * n * u_minus;

    // Initialise the list of polynomials, encoding the fact that the ti's are non-zero,
    // and that om and i are primitive fourth and third roots of unity respectively.
    polys := {t1_f * t1inv_f - 1, t2_f * t2inv_f - 1, t3_f * t3inv_f - 1, om_f + om_f^2 + 1, i_f^2 + 1};

    // Complete set of polynomials by adding the conditions required for g to fix Wdd
    for v in Basis(Wdd) do
        for poly in Coordinates(VB, v * f(g))[5..14] do
            Include(~polys, poly);
        end for;
    end for;

    // If field is finite, output the Groebner basis
    if IsFinite(field) then
        A1, A2 := GroebnerBasis([Numerator(p) : p in polys]);
        return A1;
    end if;

    // Otherwise also output the list of primes the F4 algorithm divided by. These
    // need to be checked individually by running the function again.
    if field eq Rationals() then 
        SetGBGlobalModular(false);
        A1, A2, A3 := GroebnerBasis([Numerator(p) : p in polys] : ReturnDenominators := true);
        return A1, A3;
    end if;

end function;
\end{lstlisting}

\begin{lstlisting}[label=verb7,caption=Executing the Groebner basis search using the function findGroebnerBasis]
for i in [1..6] do

	time B, badPrimes := findGroebnerBasis(Rationals(), i);
	"The case i = ", i;
	B;

	for p in [x : x in badPrimes | x in [3,7] eq false] do
		"Checking the prime p = ", p;
		findGroebnerBasis(GF(p), i);
	end for;

end for;
\end{lstlisting}

\begin{lstlisting}[label=verb8,caption=Code for subgroups of Sp_6(q)]
// SL(2,7) case (p not 2).

G := SL(2,7);

// Range over all conjugacy classes of subgroups isomorphic to S^\dag

for rec in Subgroups(G) do
	S := rec`subgroup;
	if IdentifyGroup(S) eq <21,1> then
		// Find all intermediate subgroups between S and G
		im:=IntermediateSubgroups(G,S);
		// Filter out the ones that normalise S
		imNN:=[H:H in im|IsNormal(H,S) eq false];
		// Print the ChiefFactors of such subgroups
		for H in imNN do
			ChiefFactors(H);
		end for;
	end if;
end for;
// There are no non-normalising overgroups.


// SL(2,13) case (p not 2).

G := SL(2,13);

// Range over all conjugacy classes of subgroups isomorphic to S^\dag

for rec in Subgroups(G) do
	S := rec`subgroup;
	if Order(S) eq 21 then
		S;
	end if;
end for;
// There are no subgroups of order 21.


// U(3,3) case.

G := SU(3,3);

// Range over all conjugacy classes of subgroups isomorphic to S^\dag

for rec in Subgroups(G) do
	S := rec`subgroup;
	if Order(S) eq 21 and IdentifyGroup(S) eq <21,1> then
		// Find all intermediate subgroups between S and G
		im:=IntermediateSubgroups(G,S);
		// Filter out the ones that normalise S
		imNN:=[H:H in im|IsNormal(H,S) eq false];
		// Print the ChiefFactors of such subgroups
		for H in imNN do
			ChiefFactors(H);
		end for;
	end if;
end for;
// The only possibility is PSL(2,7).


// J_2 case.

// Construct 2.J2 by taking the appropriate maximal subgroup of Sp_6(5).
G := ClassicalMaximals("S",6,5)[10];

// Range over all conjugacy classes of subgroups isomorphic to S^\dag

for rec in Subgroups(G) do
	S := rec`subgroup;
	if Order(S) eq 21 and IdentifyGroup(S) eq <21,1> then
		// Find all intermediate subgroups between S and G
		im:=IntermediateSubgroups(G,S);
		// Filter out the ones that normalise S
		imNN:=[H:H in im|IsNormal(H,S) eq false];
		// Print the ChiefFactors of such subgroups
		for H in imNN do
			ChiefFactors(H);
		end for;
	end if;
end for;

// The only possibilities are PSL(2,7), SL(2,7), U(3,3).
\end{lstlisting}

\renewcommand*{\proofname}{Proof.}
\printbibliography
\end{document}